\documentclass{article}
\usepackage{amsmath,amsthm,amsfonts,amssymb,esint,mathtools,extarrows,witharrows,enumitem}

\newcommand{\ov}{\overline} 
\renewcommand{\a }{\alpha }
\renewcommand{\b }{\beta }   
\renewcommand{\d}{\delta }    
\renewcommand{\l}{\lambda }
   
\newcommand{\R}{\mathbb{R}}

   \def\O{\Omega}
\def\e{\epsilon}
\def\S{\Sigma} 
\def\n{\nabla}

\def\p{\partial}

\def\a{\alpha}
\def\b{\beta}
\def\<{\langle}
\def\>{\rangle}
\def\n{\nabla}

\def\t{\mathbf{t}}
\def\j{\mathbf{j}}

\def\gg{\mathfrak{g}}

\def\O{\Omega}
\def\p{\partial}
\def\e{\epsilon}

\def\De{\Delta}

\def\ve{\varepsilon}

\def\a{\alpha}
\def\b{\beta}
\def\g{\gamma}
\def\d{\delta}
\def\H{\mathcal{H}}

\def\l{\lambda}
\def\s{\sigma}
\def\wt{\widetilde}
\def\ve{\varepsilon}
\def\wh{\widehat}

 \def\v{\varphi}
  \def\PP{\mathbb{P}}
  \def\NN{\mathbb{N}}
  \def\CC{\mathbb{C}}
  \def\RR{\mathbb{R}}
  
\makeatletter
\renewcommand*\env@matrix[1][*\c@MaxMatrixCols c]{%
  \hskip -\arraycolsep
  \let\@ifnextchar\new@ifnextchar
  \array{#1}} 
\makeatother

\newcommand\blfootnote[1]{
  \begingroup
  \renewcommand\thefootnote{}\footnote{#1}%
  \addtocounter{footnote}{-1}%
  \endgroup
}

\newtheorem{remark}{Remark}[section]    
\newtheorem{lemma}{Lemma}[section]
\newtheorem{proposition}{Proposition}[section] 
\newtheorem{thm}{Theorem}
\newtheorem{definition}{Definition}[section]

\newtheorem{thmx}{Theorem}

\numberwithin{equation}{section}

\title
{
On
constant mean curvature 1-immersions of surfaces  
into hyperbolic 3-manifolds.
}

\begin{document} 

\author{Gabriella Tarantello and Stefano Trapani}

\maketitle

\begin{abstract} 
The work of Bryant \cite{Bryant} revealed striking analogies between constant mean curvature (CMC) $1$-immersions of surfaces into the hyperbolic space $\mathbb H^3$ (Bryant surfaces) and  minimal immersions into the euclidean space $\mathbb E^3.$ Ever since, the role of (CMC) $1$-immersions in hyperbolic geometry has been widely explored, see e.g. \cite{Rosenberg} and references therein.

In account of
 \cite{Uhlenbeck}, \cite{Goncalves_Uhlenbeck} and after \cite{Tar_2}, for a given surface $S$ (closed, orientable and of genus $\gg \geq2$)  here we pursue the existence and uniqueness of  (CMC) $1$-immersions of $S$ into hyperbolic 3-manifolds. 

It has been shown in \cite{Huang_Lucia_Tarantello_2} that,
for  $\vert c \vert <1$, the moduli space
of  (CMC) $c$-immersions of $S$ into hyperbolic $3$-manifolds can be parametrised  by elements of the tangent bundle of the Teichm\"uller space of the surface $S.$ 
In turn in \cite {Tar_2} it was pointed out  that (CMC) $1$-immersions  enter as "critical" objects, in the sense that they can be attained only as limits of (CMC) $c$-immersions, 
as $|c| \to 1^-.$ However, the passage to the limit can be prevented by possible blow-up phenomena, and at the limit ($|c| \to 1^-$) we could end up  at best with an immersed surface having conical singularities supported at finitely many points (the blow-up points).

If the genus $\gg=2$ then  blow up can occur at a single point, and in  \cite{Tar_2} it was shown how it could be prevented and the passage to the limit ensured in terms of the  image $Z$ of Kodaira map given in \eqref{kodaira map}. 
In this note we show that actually blow-up can occur only at one of the six Weierstrass points of the surface. Thus, in Theorem \ref{thm1} and Theorem \ref{thmg2} we establish existence and uniqueness results under a sufficient  "compactness"  condition, which in fact turns out to be also necessary, as shown in \cite{Tar_Tra3} .

In addition we analyze the case of higher genus, where multiple (up to $\gg - 1$) blow-up points can occur. In this case, for any  $1 \leq \nu \leq \gg - 1,$ we identify in the $\nu$-secant variety of $Z$ the appropriate replacement of  $Z$ (relative to $\nu=1$),  see Proposition \ref{***}.
Moreover, in  Theorem \ref{thm3} we improve in a substantial way the asymptotic analysis of \cite{Tar_2}, which concerns only the case of "blow-up" with minimal mass. As a consequence, we cover the case of genus $\gg=3$ (see Theorem \ref{thm1.1}), and  provide relevant contributions for arbitrary genus.

\end{abstract}


\section{Introduction}\label{introduction} 

\blfootnote
{
MSC: 35J50, 35J61, 53C42, 32G15, 30F60.
Keywords: 
Blow-up Analysis, Minimiser of a Donaldson functional,  CMC 1-immersions, Grassmannian, hyperelliptic curves. 
}

Let $S$ be an oriented closed surface with genus $\mathfrak{g}\geq 2$ and denote 
by $\mathcal{T}_{\mathfrak{g}}(S)$ the Teichm\"uller space of conformal structures on $S$, modulo biholomorphisms in the homotopy class of the identity. Hence, any element $X\in \mathcal{T}_{\mathfrak{g}}(S)$ identifies a Riemann surface.  

The goal of this note is to establish existence and uniqueness of constant mean curvature (CMC) 1-immersions of $S$ into hyperbolic $3$-manifolds (i.e. with sectional curvature $-1$) in terms of elements of the tangent bundle  of the Teichm\"uller space $\mathcal{T}_{\mathfrak{g}}(S)$.

More in general, we shall consider (CMC) $c$-immersion, namely immersion with (prescribed) value $c$ of the mean curvature.  \\
As observed by Bryant \cite{Bryant}, the value $c=1$ plays a special role in this context, as indeed (CMC) 1-immersions of surfaces into the hyperbolic space $\mathbb H^3$ share striking analogies with the (cousins) minimal immersions into the Euclidean space $\mathbb E^3$,  see \cite{Rosenberg} and also \cite{Rossman_Umehara_Yamada}, \cite{Umehara_Yamada}. As a matter of fact, the value $c=1$ enters also as a "critical" parameter in our analysis of (CMC) $c$-immersions, as discussed below (see also \cite{Tar_2}). 

To be more precise, we recall (from \cite{Uhlenbeck} and \cite{Goncalves_Uhlenbeck}) that, 
 if the Riemann surface $X\in \mathcal{T}_{\mathfrak{g}}(S)$ is immersed with constant mean curvature  $c$ into a hyperbolic 3-dimensional manifold $(N, \hat{g}),$ then for  
$\hat{g}=(\hat{g}_{ij})$ and $1\leq i, j\leq 3$, the Riemann curvature tensor $R_{ijkl}$ of $(N, \hat{g})$  satisfies:
\begin{equation}\label{0.1}
R_{ijkl}=-(\hat{g}_{ik}\hat{g}_{jl}-\hat{g}_{il}\hat{g}_{jk}),
\end{equation}
and actually the system \eqref{0.1} reduces to six independent equations for the six independent components of the Riemann tensor. 

By introducing Fermi coordinates: $(z, r)\in X\times(a, -a)$ in a tubular neighbourhood of the surface (with small  $a>0$) 
then the metric tensor  satisfies: $\hat{g}_{i3}(z, r)=\delta_{i3},$ for $i=1,2,3$,  while the remaining (three) components: $\hat{g}_{ij}(z, r)$, $1\leq i \leq j \leq 2$    in view of (\ref{0.1}) satisfy:
\begin{equation}\label{0.2}
R_{i3j3}=-\hat{g}_{ij}.
\end{equation}
Interestingly, in the $(z, r)-$coordinates, (\ref{0.2}) defines a $2^{nd}$ order system of ODE's for $\hat{g}_{ij}$ with respect to the variable $r$ (and $z$ fixed), see \cite{Huang_Lucia_Tarantello_1}.
 So for $|r|< a$ (and every $z \in X$) the metric $(\hat{g}_{ij})$ is uniquely identified  by the corresponding Cauchy data at $r=0.$ Such initial data are expressed naturally in terms of the pullback metric $g$ on $X$ and the second fundamental form $II$. Locally, in the holomorphic $z-$coordinates on $X$,
 we have:
 \begin{equation}\label{0.3}
 g=\lambda dzd\bar{z}\ \ \  \text{and}\ \ \  II=h(dz)^2+c\lambda dzd\bar{z}+\bar{h}(d\bar{z})^2 
 \end{equation}
with smooth functions : $\lambda=\lambda(z, \bar{z})>0$ and $h=h(z, \bar{z})\in \mathbb{C}$, constrained by the remaining (three) independent equations
in (\ref{0.1}) as given by:
\begin{align}
&R_{ijl3}=0 \label{0.4}\\
&R_{1212}=-(\hat{g}_{11}\hat{g}_{22}-\hat{g}_{12}^2) \label{0.5},
\end{align}
see \cite{Uhlenbeck,Goncalves_Uhlenbeck} for details.
Actually, by Bianchi identity, it suffices that  (\ref{0.4})  and (\ref{0.5}) are satisfied at $r=0$, in order to hold for any $r\neq0$. Furthermore, for $r=0$ the (two independent) equations in (\ref{0.4})
combine to define the (complex) \underline{Codazzi equation} (with respect to the complex structure of $X$) so that, locally it is expressed by the Cauchy Riemann equation for 
$h$ in \eqref{0.3},  see (\cite{Lawson}) for details.

In other words, we find that 
$\a$ :=(2,0)-part of the second fundamental form II, must correspond to a holomorphic quadratic differential (Hopf differential). 
Thus, by letting $C_2(X)$ the space of holomorphic quadratic differentials on $X$,  we have: $\a\in C_2(X).$ 
 
 We recall that $C_2(X)$ is a finite dimensional complex linear space formed by holomorphic sections of $X$ valued in $K_X\otimes K_X,$ where $K_X$ is the canonical bundle of $X.$ In other wards, the Codazzi equation is 
 globally formulated in terms of the $\bar\p$ operator  relative to the complex structure of $K_X\otimes K_X$ induced by $X$ as follows: 
\begin{equation}\label{0.6}
\bar{\partial}\alpha=0 
\end{equation}

While for $r=0$, the equation (\ref{0.5})  gives the \underline{Gauss equation}, expressing compatibility between the (intrinsic) Gauss curvature $K_g$ of  $(X, g)$ and its extrinsic expression.
So (by recalling (\ref{0.3})) we have: 
\begin{equation}\label{0.7}
K_g=-1+c^2-4\|\alpha\|_g^2
\end{equation}
with $\|\cdot\|_g$ the norm taken with respect to the hermitian product on $K_X\otimes K_X$ induced by  the metric $g$. 

For convenience, in the following we denote by $E=T^{1,0}_X$ the holomorphic tangent bundle of $X,$ with dual $E^*=K_X.$ Both holomorphic line bundles $E$ and $E^*$ (and their tensor products)  will be equipped with the complex structure induced by $X.$ 

Furthermore, we let $g_X$ be the unique hyperbolic metric on $X$ (as given by the uniformization theorem)  and denote by $\|\cdot\|$ the norm corresponding to the hermitian product on $E$ and $E^*$  (and their tensor products) induced by $g_X$. 

By conformal equivalence, we have: 

 $g=e^ug_X \quad \|\alpha\|_g=\|\alpha\|e^{-u}\ \text{ and }\ K_g=e^{-u}(-\frac12\Delta u-1)$ \\ 
 with $u$ a smooth function on $X$ and $\Delta$ the Laplace Beltrami operator in $(X, g_X)$. 
 
  As a consequence, (\ref{0.7}) can be turned into the following elliptic equation (of Liouville type)
 \begin{equation}\label{0.8}
 -\Delta u=2-2(1-c^2)e^u-8\|\alpha\|^2e^{-u}.
\end{equation}
Interestingly,  every solution $(u, \alpha)$ of the Gauss-Codazzi equations (\ref{0.8})-(\ref{0.6}) identifies a "germ" of hyperbolic 3-manifolds of $X.$ Namely, as indicated by Taubes \cite{Taubes}, it identifies a hyperbolic 3-manifold $(N, \bar{g})$ ($N\simeq X\times\mathbb{R}$ not necessarily complete) where $X$ is immersed as a surface with constant mean curvature $c,$ with pullback metric $g=e^{u}g_X$ and second fundamental form identified by $g$ and $\alpha$ as its (2,0)-part. Furthermore, $(N, \bar{g})$ is \underline{uniquely} identified by the solution pair $(u, \alpha),$ up to local diffeomophisms of tubolar neighbourhoods of $X,$ see \cite{Taubes}.  \\
At this point, it may be tempting to describe the solution set of the Gauss-Codazzi equations \eqref{0.6}-\eqref{0.8} in terms of the pair 
$(X,\alpha) \in \mathcal{T}_{\mathfrak{g}}(S) \times C_{2}(X),$ which provides a local trivialization for the cotangent bundle of $\mathcal{T}_{\mathfrak{g}}(S).$ In this way, we would attain a parametrization of  (CMC) $c$-immersions into hyperbolic 3-manifolds
 by elements of the cotangent bundle $T^{*}( \mathcal{T}_{\mathfrak{g}}(S)).$ 

However, 
as discussed in \cite{Huang_Lucia} and \cite{Huang_Lucia_Tarantello_1},  for a given $\a\in C_2(X)$ a solution of \eqref{0.8} may not exist, or when it exists, 
it may not be unique (see also \cite{Hung_Loftin_Lucia}).
So, in general, the pair $(X,\a)$ is not suitable to parameterized (CMC) $c$-immersions. 
Instead, it has proved more successful the "dual" approach 
by Goncalves and Uhlenbeck in \cite{Goncalves_Uhlenbeck}, where the authors propose
to parametrize (CMC) $c$-immersions of $S$
into hyperbolic $3$-manifolds, 
in terms of elements of the tangent bundle  of the Teichm\"uller space $\mathcal{T}_{\mathfrak{g}}(S)$.

By recalling the isomorphism:
$
C_{2}(X) \simeq (\mathcal{H}^{0,1}(X,E))^{*}
$ 
with $E=T^{1,0}_{X}$ and
$\mathcal{H}^{0,1}(X,E)$ the Dolbeault (0,1)-cohomology group, see \eqref{Dol}
and  \cite{Griffiths_Harris},  
we have a parametrization of the tangent bundle of
$\mathcal{T}_{\mathfrak{g}}(S)$ by the pairs:
$
(X,[\beta])\in \mathcal{T}_{\mathfrak{g}}(S) \times \mathcal{H}^{0,1}(X,E) 
$. 

Accordingly, (as anticipated by \cite{Goncalves_Uhlenbeck}) it has been proved
in \cite{Huang_Lucia_Tarantello_2} that the following holds:

\begin{thmx}[\cite{Goncalves_Uhlenbeck},\cite{Huang_Lucia_Tarantello_2}]
\label{thm_A}
For given $c\in (-1,1)$ there is a one-to-one correspondence between the space of constant mean curvature $c$-immersions of $S$ into a (germ of) hyperbolic $3$-manifolds  and the 
tangent bundle of 
$\mathcal{T}_{\mathfrak{g}}(S)$, parametrized by the pairs: 
$
(X,[\beta])\in \mathcal{T}_{\mathfrak{g}}(S) \times \mathcal{H}^{0,1}(X,E),
\; 
E=T_{X}^{1,0}.  
$ 
\end{thmx}	
As a matter of fact, for fixed $c\in (-1,1),$   
the datum $(X,[ \beta ])$  identifies the \underline{unique} solution $(u=u_c 
 \, ,\a=\a_c)$ of the Gauss-Codazzi equations \eqref{0.6}-\eqref{0.8} subject to the constaint:
\begin{eqnarray}\label{*}
   *_E^{-1}(e^{-u_c}\a_c) \in[\b]
\end{eqnarray} where $*_E$ is the Hodge star operator relative to the metric $g_X,$ with inverse $(*_E)^{-1},$ acting between (dual) forms valued on $E$ and $E^*$ respectively, for details see \eqref{2.8*hodge_operator_intro} below.

To interpret \eqref{*} let us recall Dolbeault decomposition valid for any Beltrami differential $\beta$ (i.e. $(0,1)$-form 
 valued in $E$) as follows: 
$$
\beta = \beta_{0} + \bar{\partial}\eta
$$
with  $\beta_{0}$ \underline{harmonic} (with respect to $g_{X}$) and  
$\eta$ a smooth section of X valued on E.

Hence, the (0,1)-cohomology class  $[\beta] \in \mathcal{H}^{0,1}(X,E)$ is identified
by the unique harmonic differential $\beta_{0} \in [\beta].$ 

In this way, we attain also the (dual) isomorphism: $\mathcal{H}^{0,1}(X,E)) \simeq (C_2(X))^*,$ as discussed in details in Section 2.

Furthermore, for a solution pair: $(u\, , \a)$ of the Gauss Codazzi equations (\ref{0.8})-(\ref{0.6}), we can easily formulate the constraint \eqref{*} by setting:
\begin{equation}\label{constraint}
g=e^u g_X \quad \text{  and  } \a=e^{u}*_E(\beta_{0} + \bar{\partial}\eta)
\end{equation}
with $(u\, , \eta)$ satisfying:
\begin{equation}\label{system_of_equations_introbis}
\left\{
\begin{matrix*}[l]
\Delta u +2 -2te^{u} -8e^{u}\Vert \beta_{0}+\overline{\partial}\eta \Vert^{2} =0  &  \;\text{ in }\;  &  X  \\
\overline{\partial}(e^{u}*_{E}(\beta_{0}+\overline{\partial}\eta))=0 &  \;\text{}\;  &   \\ 
\end{matrix*}
\right.
\end{equation} and $t=1-c^2.$

Interestingly, the solutions of the "constraint" Gauss-Codazzi equations (\ref{system_of_equations_introbis}) will correspond to the critical points of the so called Donaldson functional $F_{t}$ ($t=1-c^{2}$) introduced in \cite{Goncalves_Uhlenbeck} and defined in \eqref{F_t} below.
\\
Indeed, Theorem \ref{thm_A} 
is established in \cite{Huang_Lucia_Tarantello_2} by showing precisely that, for $t>0$ the functional $F_t$ admits a \underline{unique} critical point given by its global minimum.

At this point it is natural to ask whether, 
for a given pair $(X,[\beta])$ analogous (CMC) $c$-immersions do exist also when $\vert c \vert \geq 1$. 

While we have an evident non-existence result when $[\beta]=0$
(see Section \ref{Asymptotics} for details), for $[\beta]\neq 0$ such question becomes extremely delicate and  difficult to tackle in general. Indeed for $t\leq0,$ the functional  $F_t$ may be unbounded from below and it is not obvious how to detect possible critical points. Even when $t=0,$ that is $|c|=1,$ such a task is rather involved.  
For example, by \cite{Bryant} we know that
(CMC) $1$-immersions of surfaces into the hyperbolic space $\mathbb{H}^{3}$ develop "smooth ends" \cite{Rossman_Umehara_Yamada}, \cite{Umehara_Yamada}, which in a (conformal) compact setting are captured by "punctures" at finitely many points.
Those points will occur naturally in our analysis as ``blow-up" points carrying "quantized" blow-up mass. Indeed, it was shown in \cite{Tar_2} that (CMC) $1$-immersions can be attained only as limit of the (CMC) $c$-immersions of Theorem \ref{thm_A}, as $|c| \to 1^-$, see Theorem \ref{thmprimo} below. However, as we shall see, such passage to the limit may be prevented by the occurrence of "blow-up" phenomena (as $|c| \to 1^- $) and reasonably  in this situation one expected to end up with an immersed "limiting" surface with "conical" singularities at the blow up points.

Therefore, to obtain regular (CMC) $1$-immersions,  we must understand how "blow-up" can be rule out. 

When the genus $\mathfrak{g}=2$, this goal was attained in \cite{Tar_2}, where the existence of (CMC) $1$-immersions is formulated in terms of the Kodaira map:

\begin{equation}\label{Kodaira_map_Intro}
\tau:X\longrightarrow \mathbb{P}(V^{*}),\; V=C_{2}(X)
\end{equation}
described in section 12.1.3 of \cite{Donaldson_Book}. 
Recall that $\tau$ defines a holomorphic map of $X$ into the projective space: $\mathbb{P}(V^*)\simeq \mathbb{P}(\mathcal{H}^{0,1}(X,E))$, $E=T^{1,0}_{X}$. 

The role of the projective space 
$\mathbb{P}(\mathcal{H}^{0,1}(X,E))$ 
is readily explained once we observe that, if the pair $(X,[\beta])$ yields to a (CMC) $1$-immersion subject the constraint \eqref{constraint}, then $[\beta]\neq 0$ and by a simple scaling argument, we obtain a (CMC) 1-immersion corresponding to the data
$(X,\lambda [\beta])$, for all $\lambda \in \mathbb{C}\setminus \{  0 \} $. 
 
Note that, for $E=T^{1,0}_{X}$ we have:  
$\mathbb{P}(\mathcal{H}^{0,1}(X,E)) \simeq \mathbb{P}^{3\mathfrak{g}-4}$
and consequently, 
$\dim_{\mathbb{C}}\mathbb{P}(\mathcal{H}^{0,1}(X,E)) \geq 2$ for $\mathfrak{g}\geq 2.$ Since the image 
$\tau(X)$ defines a complex curve  into $\mathbb{P}(\mathcal{H}^{0,1}(X,E)),$ we get that: $\tau(X) \subsetneq \mathbb{P}(\mathcal{H}^{0,1}(X,E)),$ and 
actually, $\mathbb{P}(\mathcal{H}^{0,1}(X,E))\setminus \tau(X)$ defines a Zariski open subset  (hence dense) in $\mathbb{P}(\mathcal{H}^{0,1}(X,E)).$

For $[\beta] \in \mathcal{H}^{0,1}(X,E) \setminus \{  0 \} $ we let,
\begin{equation} \label{betap}
[\beta]_{\mathbb{P}} = \{ [\lambda \beta], \quad \lambda \in \mathbb{C}\setminus \{  0 \} \} \in \mathbb{P}(\mathcal{H}^{0,1}(X,E))
\end{equation}
the projective representative of the class $[{\beta}]$ in $\mathbb{P}(\mathcal{H}^{0,1}(X,E)).$
The following holds:

\begin{thmx}[\cite{Tar_2}]\label{thmB}
If $\mathfrak{g}=2$, then to every  
$(X,[\beta]) 
\in 
\mathcal{T}_{\mathfrak{g}}(X) \times (\mathcal{H}^{0,1}(X,E)\setminus \{  0 \} )$
$E=T^{1,0}_{X}$, with projective representative 
$[\beta]_{\mathbb{P}}\not \in \tau(X)$, 
there  correspond a \underline{unique} (CMC) 1-immersion of $X$  
into a (germ of) hyperbolic $3$-manifold $N(\simeq S \times \R$), with pull back metric $g$ and $(2,0)$-part $\a$ of the second fundamental form $II$ satisfying:
\begin{eqnarray}\label{1.***}
g=e^u g_X \quad \text{  and  } \quad *_E^{-1}(e^{-u}\a)\in [\b],
\end{eqnarray}with $*^{-1}_E$ the inverse of the Hodge star operator $*_E$ 
\end{thmx}	

The goal of this note is to extend Theorem \ref{thmB} in two directions. 

Firstly, still in the case of genus $\gg=2$, we shall show that the conclusion of Theorem \ref{thmB} remains valid under a much weaker (in fact sharp) assumption, namely that $[\b]_{\PP}$ is not in the image by $\tau$ of the \underline{six} Weierstrass point of $X$. 

More precisely, we recall that every Riemann surface of genus $\gg=2$ is hyperelliptic. Hence it admits a unique biholomorphic hyperelliptic involution (see \cite{Miranda}, \cite{Griffiths_Harris})
\begin{eqnarray*}
    j: X\to X
\end{eqnarray*} with exactly $2(\gg+1)=6$ (for $\gg=2$) distinct fixed points. Moreover, for $\gg=2$, those  points coincide with the Weierstrass points of $X$ (cf. \cite{Miranda}). We prove:
\begin{thm}\label{thm1}
    If $\gg=2$, and $(X,[\beta]) 
\in 
\mathcal{T}_{\mathfrak{g}}(X) \times (\mathcal{H}^{0,1}(X,E)\setminus \{  0 \} )$ satisfies:  
\begin {equation} \label {weierstrass} [\beta]_{\mathbb{P}}\not \in \{\tau(q), \text{ with }\, q\in X: j(q)=q\}, \end{equation} then there exist a \underline{unique} (CMC) 1-immersion of $X$  
into a hyperbolic   $3$-manifold $N(\simeq X \times \R$) satisfying \eqref{1.***}.
\end{thm}

Actually, the sufficient condition \eqref{weierstrass} (for existence and uniqueness) is also necessary, in the sense that, when it fails then the (CMC) $c$-immersions in Theorem \ref{thm_A} do not pass to the limit, as $|c| \to 1^-,$ and so non-existence holds. 
This fact will be proved in a forthcoming paper \cite {Tar_Tra3}, by means of the equivalent formulation of problem \eqref{system_of_equations_introbis} in terms of Hitchins' sefduality equations \cite{Hitchin} with respect to a suitable nilpotent $SL(2,\mathbb{C})$ Higgs bundle, see \cite{Alessandrini_Li_Sanders} and \cite{Huang_Lucia_Tarantello_2} for details.
 
In Theorem \ref{thmg2} below we shall provide a more complete version of Theorem \ref{thm1}, and it  will be proved in Section \ref{mainTheorems} on the basis of  the invariance under bi-holomorphism: $X\to X$ of the (constrained) Gauss-Codazzi equations  (\ref{system_of_equations_introbis}) and correspondingly of the Donaldson functional $F_t$ in \eqref{F_t}, see Appendix 2 for details.

Our other contribution will be to establish a suitable extension of Theorem \ref{thmB} for higher genus. 

Our most complete result will be specific of genus $\gg=3$ and it is stated in Theorem \ref{thm2} below.  However, from our investigation it emerges a rather convincing picture, in the sense that we expect Theorem \ref{thm2} to remain valid in the general case of $\gg \geq3.$  

To handle the case of $\gg \geq3,$ where multiple (up to $\gg-1$) blow up points may occur,  firstly we need to identify a suitable replacement of the Kodaira map (defined in $X$) to a map defined on the symmetric product $X^{(\nu)}$ of $\nu$-copies of $X$ modulo permutations, for  $1\leq \nu\leq \gg-1.$ 
Recall that $X^{(\nu)}$ defines a smooth complex manifold of dimension $\nu$ (see \cite{Griffiths_Harris}), and it can be identified with the space of non zero effective divisors of degree $\nu \geq  1$ on $X$. Indeed, the $\nu$-ple representing an element in $X^{(\nu)}$ is identified with the effective  divisor $D=\sum_{j=1}^k n_j p_j$, where the \underline{support} of D: $supp \,D = \{p_1,\ldots,p_k\}\subset X$ is formed by the \underline{distinct} points contained in the given $\nu$-ple,  
and the \underline{multiplicity}  $n_j\in \NN$ of the point $p_j$ is defined by the number of times such point appears in the given $\nu$-ple, for $j=1,\ldots, k.$  Finally the degree of D is given by: $deg (D) =\sum_{j=1}^k n_j=\nu$. 

Divisors naturally arise also in connections with holomorphic quadratic differentials. Indeed, 
we know that every non-trivial holomorphic quadratic differential $\a\in C_2(X)$ admits $4(\gg-1)$ zeroes counted with multiplicity . So the zero set of $\a$ identifies in a natural way an effective divisor in $X^{(4(\gg-1))}$ which is denoted by: $\text{div}(\a)$.

For $1\leq \nu\leq \gg-1$ and an effective divisor $D\in X^{(\nu)}$ we let, 
$$
Q(D)=\{\a\in C_2(X): \text{div}(\a) \geq D\} \\
$$ namely, every $\a\in Q(D)$ must vanish at each point of the support of $D$ with greater or equal multiplicity.\\
In particular for $x_0 \in X,$ by taking $D = x_0$ we have: $Q(x_0) = \{ \alpha\in C_2(X): \alpha(x_0)=0\},$ and we recall: 
\begin{eqnarray}\label{0**} 
 [\beta]_{\mathbb{P}}= \tau(x_{0}) \iff \quad  \int_X \b\wedge \,\a=0, \quad \forall 
 \a\in Q(x_0).
 \end{eqnarray}
Thus, for $\gg\geq 2$ and for given $1\leq \nu\leq \gg-1$,  in Section \ref{Preliminaries}  we identify analytic irreducible subvarities $\wt\S_{\nu}\subset \PP(\H^{0,1}(X,E))$, such that :$$
    \wt{\S}_1 =\tau (X) (=\text{ image of Kodaira map )}\subset \wt\S_2\subset\cdots \subset \wt \S_{\gg-1} \subsetneq \PP(\H^{0,1}(X,E)),
$$
\begin{eqnarray}\label{1**}
[\b]_{\PP}\in \wt\S_\nu \iff \quad \exists \text{ divisor } D\in X^{(\nu)}: \int_X \b\wedge \,\a=0, \quad \forall 
 \a\in Q(D), \quad \quad 
\end{eqnarray}
$\text{dim}(\wt \S_{\nu}) \leq 2\nu-1<3\gg-4.$ \\ 

Again, $\PP(\H^{0,1}(X,E))\setminus \wt{\S}_{\gg-1}$ is a Zariski open set  (hence dense)  in $\PP(\H^{0,1}(X,E))$.\\

Please note that the "orthogonality" conditions  \eqref{0**} and  \eqref{1**} do not depend neither on the choice of the element in the projective class $[\b]_\PP$ nor on the chosen representative in the cohomology class $[\b]\in \H^{0,1}(X,E)$.\\

It was pointed out to us that geometrically, the sub-variety $\wt\S_{\nu}$ is exactly the $\nu$-secant variety of $Z=\tau(X)$ (cf. \cite{ACGH}), and we refer to Appendix \ref{append3} and Lemma \ref{secant} for details .  

We prove:
\begin{thm}\label{thm2}
If $\gg=3$ and $(X,[\b])\in \mathcal{T}_{\mathfrak{g}}(X) \times \mathcal{H}^{0,1}(X,E)$ satisfies:
$$
[\b]\neq0 \text{ and } [\b]_\PP\notin \wt\S_{2},
$$ 
where $\wt\S_{2}$ is the 2-secant variety of $\tau(X).$ Then there  exists a \underline{unique} (CMC) 1-immersion of $X$  
into a (germ of) hyperbolic $3$-manifold $N(\simeq X \times \R$) 
satisfying \eqref{1.***}.
\end{thm}

As in \cite{Tar_2} (and in view of Theorem \ref{thmprimo} below) we shall establish Theorem \ref{thm2} by means of 
a detailed asymptotic analysis of the (CMC) $c$-immersions given by 
Theorem \ref{thm_A} (with $\vert c \vert < 1$), as $|c| \to 1^{-}$.\\

In the framework of Hitchin 's self-duality theory (\cite{Hitchin}) in this way we pursue the limiting behaviour of suitable Higgs bundles under the $\mathbb{C}^*$ action. In this context  the role of secant varieties also appeared in \cite{Wilkin}.\\

\medskip

Thus, for the unique solution pair: $(u_c\, , \alpha_c),$ we need to analyze what happens, when $|c| \to 1^{-}$ . Due to the Liouville-type character of the Gauss equation \eqref{0.6}, we find that the function $ \xi_c$ := $-u_{c }$+ log($ \int_X ||\alpha_{c}||_{L^2}^{2}dA$) can blow-up only around finitely many points (blow-up points) see \cite{Brezis_Merle}, and as shown in \cite{Li_Shafrir}, \cite{Bartolucci_Tarantello} and \cite{Tar_1}, each blow-up point carries a "quantized" blow-up mass given by an integral multiple of $8\pi$, see also \cite{Lee_Lin_Wei_Yang}, \cite{Lee_Lin_Yang_Zhang} for related results.

The new difficulty here is that blow-up can occur at an accumulation point of different zeroes of the holomorphic quadratic differentials  $\a_c,$ as $ |c| \to 1^{-}.$  Thus in this situation, when $\gg \geq 3,$ we can no longer ensure that "blow-up" occurs with the  "concentration" property,  see \cite{Suzuki_Ohtusuka}, \cite{Lin_Tarantello}, \cite{Lee_Lin_Tarantello_Yang} \cite{Tar_1}, and therefore (as expected) we are lead to a "limiting" metric with conical singularities at the blow-up points with ``conical" angle  an integral multiple of $8\pi$ ( and not the usual $4\pi$ in view of our normalization of the conformal factor), see Section \ref{Asymptotics} for details.

Therefore, also to a blow-up situation we can associate an effective "blow-up" divisor  $D,$ whose support is formed by the set of (distinct) blow-up points with multiplicities given by the corresponding (integral) blow-up masses, see \eqref{sigma_q} and definition \ref{blowupdiv}.
By the Gauss-Bonnet Theorem we have that a "blow-up" divisor $D$ must satisfy: $D\in X^{(\nu)}$ with $1\leq \nu\leq \gg-1$. 

The very aim of our blow-up analysis will be to show that, if blow-up occurs and we let $D\in X^{(\nu)}$ the corresponding blow up divisor ($1\leq \nu\leq \gg-1$) then:
 \begin{eqnarray}\label{ortog}
    \int_X \b\wedge \a =0 \quad \forall \a \in Q(D) \ \mbox{with}\,\, D \,\mbox{ the blow up divisor}.
\end{eqnarray}  
More precisely, we show that \eqref{ortog} holds with a suitable divisor  $ 0 \neq \widetilde D \leq D,$ see Section \ref{Asymptotics} for details.

Unfortunately (by mere technical reasons), at present we can  show \eqref{ortog} only when we assume that each point in the support of the "blow-up" divisor admits multiplicity  at most $2.$

We recall that, based on the accurate blow-up analysis developed in \cite{Tar_1}, the ``orthogonality" condition \eqref{ortog}  was firstly pointed out in \cite{Tar_2}, when one assumes that blow-up occurs with the least blow-up mass $8\pi$. Namely when, all points in the support of the "blow-up" divisor $D$ admit multiplicity 1 (see Theorem \ref{thmB.3}). Such result applies directly to the case of genus $\gg=2$  and permits to establish Theorem \ref{thmB}. 

In Theorem 3 below, we show that \eqref{ortog}  remains valid also when we allow the "blow-up" divisor to contain points in its support  of multiplicity  2. 
In this way, we shall be able to handle the case of genus $\gg=3$ and obtain Theorem \ref{thm2}.

As we shall see in Section \ref{Asymptotics}, even under the above mentioned restriction on the multiplicity of the blow-up points, the blow-up analysis becomes immediately more involved, and in order to attain  \eqref{ortog} one needs to resolve new analytical difficulties beyond \cite{Tar_1}, \cite{Tar_2}.

However, on the basis of the analysis we have contributed here, it should be possible to elaborate additional scalings together with inductive arguments to remove such restriction, a task we hope to carry out in future work.

\section{Preliminaries}\label{Preliminaries}

Let $X\in \mathcal{T}_{\mathfrak{g}}(S)$ be a given Riemann surface with (unique) hyperbolic metric $g_{X}$ and induced scalar product
$\langle\cdot,\cdot\rangle$, norm $\Vert \cdot \Vert$ and volume 
element $dA$. 

For any given point $x_0 \in X$, 
we can introduce holomorphic $\{ z \}$-coordinates around $x_0$ centred at the origin (namely $x_0$ is mapped to $0$) and for small $r>0$ we denote by : 
\begin{equation}\label{palla} 
\begin{array}{l}
B(x_0;r), \,\, \mbox{the geodesic ball centred at $x_0$ with radius r},  \\ \,\,  
\Omega_r  \,\, \mbox{the image of $B(x_0;r)$ in $\mathbb {C}$, with } \, 0\in \Omega_r, \\
B_{r} \mbox{ the disc in $\mathbb {C}$ of center the origin and radius $r$.}
\end{array} 
\end{equation}
Hence, for $\delta>0$ sufficiently small we have: $B_{\delta}  \subseteq \Omega_r.$

Moreover,  for $z=x+iy\in \Omega_{r}$  ($r>0$ small) we have the following local expression of the conformal and Riemannian structure of $X$:
\begin{equation}\label{coord} 
\begin{array}{l}
\partial=\frac{\partial}{\partial z}=
\frac{1}{2}(\frac{\partial}{\partial x} - i \frac{\partial}{\partial y})
\; \text{ and } \; 
\bar{\partial}=\frac{\partial}{\partial \bar{z}}
=
\frac{1}{2}(\frac{\partial}{\partial x} + i \frac{\partial}{\partial y}) \\
dz=dx + i dy,\; d \bar{z} = dx-idy,\\
g_{X}=e^{2u_X} dzd\bar{z}\quad
u_X \; \text{smooth,} \\ 
dA = \frac{i}{2}e^{2u_{X}} dz \wedge d \bar{z},
\end{array} 
\end{equation}
with the wedge $\wedge$ product defined as usual on complex valued forms.

Furthermore, without loss of generality, we can consider the so called "normal" coordinates at $x_0,$ by assuming further that $u_X$ satisfies:
\begin{equation}\label{hypcoord} 
u_X(0) = |\nabla u_X(0)|=0. 
\end{equation}

In addition, in such local coordinate the Laplace-Beltrami operator $\Delta$ on $(X, g_X)$ can be expressed (locally) as follows:
$\Delta=4e^{-2u_X} \partial \bar{\partial}$ 
and in particular we have:
$4 \partial \bar{\partial} u_{X} = e^{2u_{X}}$ in $\Omega_r$ . 

Actually (with abuse of notation) in the sequel we also denote 
the \underline{flat} Laplacian by  
$\Delta=4 \partial \bar{\partial}$, 
unless confusion arises.  
\vskip0.5cm
Throughout this paper, we let:
\begin{eqnarray}\label{E}
E=T_X^{1,0} \text{ the holomorphic tangent bundle of $X$} 
\end{eqnarray}
with dual:
$$
E^*=(T_{X}^{1,0})^{*}=K_{X} \text{ the canonical bundle of $X.$}
$$

The holomorphic line bundles $E$ and $E^*$ will be equipped with the complex structure induced by $X$, and with an hermitian product induced by a metric $g$ (conformal to the metric $g_X$) defined in $X$.  
Thus, on sections and forms valued on $E$ (or $E^*$), we have a well defined  d-bar operator denoted by: $\bar\p$ and a fiber-wise hermitian product $\langle \cdot,\cdot \rangle_{g}$ and norm $\Vert \cdot \Vert_{g}.$ 

In the sequel, unless confusion arises, we shall drop the subscript $g_X$ in the hermitian product and norm induced by $g_X.$

We let:
$$ \begin{array}{l}
   A^{0}(E)=
\{\text{smooth sections of $X$ valued on $E$}\}, \\  A^{0,1}(X,E)
=
\{ \text{$(0,1)$-forms valued on $E$ }  \},
\end{array} $$
and recall that the elements in $A^{0,1}(X,E)=A^{0,1}(X,\mathbb{C}) \otimes E$ are also known as Beltrami differentials.    
By using the operator:
$$\bar\p=\bar\p_E
:
A^{0}(E)
\longrightarrow 
A^{0,1}(X,E).
$$  
we can define the $(0,1)$-Dolbeault cohomology group as follows:
\begin{equation}\label{Dol}
\mathcal{H}^{0,1}(X,E)
=
A^{0,1}(X,E)/\overline{\partial}(A^{0}(E)),
\end{equation}
so that, for any Beltrami differential $\beta \in A^{0,1}(X,E),$ there correspond the cohomology class: 
$$[\beta]=\{ \beta + \bar{\partial}\eta,
\; \forall \; \eta \in A^{0}(E) \}\in \mathcal{H}^{0,1}(X,E).$$

Similarly, we let: 

\begin{eqnarray*}
A^{1,0}(X,E^*)
= \{ \text{$(1,0)$-forms valued on $E^*$ }  \} = A^{1,0}(X,\mathbb{C}) \otimes E^*.
\end{eqnarray*}
We can consider the \underline{wedge product} to act on: $A^{0,1}(X,E) \times A^{1,0}(X,E^{*}),$ so for $\beta \in A^{0,1}(X,E)$ and $\alpha \in A^{1,0}(X,E^{*}),$ we have: $\beta \wedge \alpha
 \in A^{1,1}(X,\mathbb{C})$ 
satisfying the well-known properties of the wedge product, see \cite{Griffiths_Harris}  Chapter 2 Section 2.

Consequently, we obtain the 
bilinear form: 
\begin{equation}\label{wedge_product_map}
\begin{split}
 A^{1,0}(X,E^{*}) \times A^{0,1}(X,E) \longrightarrow \mathbb{C}
\;:\;
(\alpha,\beta)
\longrightarrow 
\int_{X}   \beta \wedge \alpha,
\end{split}
\end{equation}
which, by Serre duality (see \cite{Voisin}), is non-degenerate and induces the isomorphism:
\begin{equation}\label{isomorphism_A_1_0_To_A_0_1}
A^{1,0}(X,E^{*})
\simeq 
(A^{0,1}(X,E))^{*}.
\end{equation}
\medskip

For $x \in X,$  we consider the anti-linear Hodge star  operator: \\ $  *_{x}: $$ \,  [T_{x}(X)^{*}]^{0,1} \to  [T_{x}(X)^{*}]^{1,0} $  defined in the usual way,  see \cite{Wells}, \\ and we recall that a   map $L : V \to W$ between complex vector spaces is called anti-linear, if it is $\mathbb{R}$ linear and $L(iv) = -i L(v)  \ \forall v \in V.$

Also for   $\varphi \in [T_{x}(X)^{*}]^{0,1}\,\,\
 \mbox { and }  e, f \in E_x$ we recall the anti-linear isomorphisms: $$\sharp_x : E_x \to E^*_x\quad  \mbox{ and } \quad  *_{x}: \,  [T_{x}(X)^{*}]^{0,1} \otimes E_x \to  [T_{x}(X)^{*}]^{1,0}\otimes E^*_x $$ defined as follows: 

$\sharp_x(e)(f) = <f,e>_x,$ and $*_{x}(\varphi \otimes e) = *_{x}(\varphi) \otimes \sharp_x(e),$  for every $x \in X.$
\medskip

In this way, we obtain the Hodge star operator  defined on forms:  
 \begin{equation}\label{2.8*hodge_operator_intro}
*_{E} : A^{0,1}(X,E) \longrightarrow  A^{1,0}(X,E^{*}),
\end{equation}
where we  see that, for given $\beta \in A^{0,1}(X,E),$  the form 
 $*_{E}\beta \in A^{1,0}(X,E^*)$  is identified by the  condition: 
$$ 
\xi \wedge *_{E}\beta  
=
\langle \xi,\beta \rangle \,dA,
\quad \forall \; \xi \in A^{0,1}(X,E).$$ We have:
\begin{equation}\label{antiso}
< *(\beta_1), *(\beta_2)> = \overline{ < \beta_1, \beta_2>} =  < \beta_2, \beta_1> \,\,\, \mbox{ and }  *_E(i \beta) = -i *_E(\beta). 
\end{equation}

Hence, the Hodge operator $*_{E}$ depends on the metric $g_X$ and  it
defines an invertible, norm preserving operator, with inverse $*_E^{-1}.$ In fact, \eqref{2.8*hodge_operator_intro} expresses the (metric dependent) \underline{isomophism}  between
$A^{1,0}(X,E^{*})$ and $ A^{0,1}(X,E)$.\\ 

In local holomorphic  coordinates, for  $*=*_{E}$ we have:
 \begin{equation}\label{**}
  * dz = i d \bar{z}, \,\,\,   *d \bar{z} = - i d z, \quad  \sharp(\frac{\partial}{\partial z}) = \frac{e^{2u_{X}}}{2} dz, \end{equation} 
and in particular, for $\beta = \beta(z) (d \bar{z} \otimes\frac{\partial}{\partial z})$ there holds:
 \begin{equation}\label{**bis}
*\beta= *(\beta(z) d \bar{z}\otimes\frac{\partial}{\partial z})=\frac{-i}{2}\bar{\beta}(z)e^{2u_{X}}(dz)^2.
 \end{equation}

Moreover, for the local expression 
of the (fiberwise) norm (induced by $g_X$) of sections and forms there holds:
\begin{equation}\label{norme}
\begin{split}
\eta = \eta(z) (\frac{\partial}{\partial z}) \quad \mbox { then } \,\,\, \Vert \eta \Vert
= \; 
\vert \eta(z) \vert \frac{e^{u_{X}(z)}}{\sqrt{2}},
\;\; 
\eta\in A^{0}(E); 
\\ 
\beta = \beta(z) (d \bar{z} \otimes\frac{\partial}{\partial z}) \quad \mbox{ then } \,\,\, \Vert \beta  \Vert = \vert \beta(z) \vert,  
\;\;
\beta \in A^{0,1}(X,E); \\
\a= h(z)(dz)^{2} \,\, \mbox{ then } \Vert \alpha \Vert =2 \vert h  \vert e^{-2 u_X}, \,\,\, \a \in A^{1,0}(X,E^*).
\end{split}
\end{equation}

In general, for a given holomorphic line bundle  $L$,  we denote by $H^0(X,L)$ the space of global holomorphic sections of $L$. Clearly, $H^0(X,L)$ is a complex linear space, and since $X$ is compact, it admits finite dimension ( see \cite{Miranda} Proposition 3.16 ). Moreover, every $\alpha \in H^0(X,L) \setminus \{ 0 \}$ admits the same number of zeroes counted with multiplicity (see  \cite{Miranda} Lemma 1.5 ) and thus we can define the degree of $L,$ denoted by  $\deg L, $ as given by  the number of zeroes counted with multiplicity of a non trivial section in $H^0(X,L).$ 

Since, 
$$ \begin{array}{l}
C_{2}(X)
= H^{0}(X,\otimes^{2}(K_X))
=\{ 
\alpha \in A^{0}(X,\otimes^{2}(K_X))
\; : \; 
\bar{\partial} \alpha=0
\} \; \\ \\
=\{ \alpha \in A^{1,0}(X,E^{*})
\; : \;
\bar\partial 
 \alpha=0 \} \end {array}$$
we can use the Riemann-Roch Theorem for $L=\otimes^{2}(K_X),$ to find:
\begin{equation}\label{dim}
\dim_{\mathbb{C}}C_{2}(X)= 3(\mathfrak{g}-1),
\end{equation}
(see \cite{Miranda} and \cite{Narasimhan}).
\medskip
Furthermore, we know that,
$deg(K_X) = 2(\gg-1)$ (see \cite{Miranda} Chapter V Prop. 1.14), therefore:
\begin{eqnarray}\label{2.01}
\deg \otimes^2K_X = 4 (\gg-1), 
\end{eqnarray}
and consequently,
\begin{eqnarray}\label{2.03}
\text{any } \, \a\in C_2(X)\setminus\{ 0 \} \,\, \text {admits}\, \,  4(\gg-1) \text{ zeroes counted with multiplicity. } 
\end{eqnarray}

In local holomorphic $z$-coordinates around a given $x_0$ (centred at the origin), any $\a\in C_{2}(X)$ (as specified in \eqref{hypcoord}) takes the expression: 
\begin{equation}\label{alpha1}
\a= h(dz)^{2} \,\, \mbox{ (and } \Vert \alpha \Vert =2 \vert h  \vert e^{-2 u_X}) \,\,\, \mbox{ $h$ holomorphic around the origin.}
\end{equation}

In this way, it is clear what we mean by a zero of $\alpha$ and corresponding multiplicity, as indeed those notions are independent of the chosen holomorphic coordinates. 
In particular, if $q$ is a zero of $\alpha$ with multiplicity $n$, then 
in  local $z$-coordinates at $q$ centred at the origin (defined in $B_r( q ; r), $ $r>0 $ small) we have:
 $$\alpha = z^{n} \psi(z)(dz)^{2} \quad \mbox{ $\psi$ holomorphic and never vanishing
  in $\Omega_{r},$}$$ 
 and we have: $\Vert \alpha \Vert=2 \vert z \vert^{n} \vert \psi(z) \vert e^{-2 u_X}.$\\ 
Also observe that,  $\partial \bar{\partial} \ln \vert \psi \vert^{2}=0$ in $\Omega_{r}$, a property we shall use in the sequel.  
\vskip.1cm
By Stokes theorem, for $\alpha \in C_{2}(X)$ we have:
$
\int_{X} \bar{\partial}\eta  \wedge \alpha = 0,
\; \forall \; \eta \in A^{0}(E) 
$, 
and we see that the bilinear form \eqref{wedge_product_map} is 
well defined and non degenerate
when restricted on the space:
$C_{2}(X) \times \mathcal{H}^{0,1}(X,E)$, 
and it induces the isomorphism:
\begin{equation}\label{C_kappa_X_isometry}
\begin{split}
C_{2}(X) \simeq (\mathcal{H}^{0,1}(X,E))^{*}.
\end{split}
\end{equation}

By Dolbeault decomposition, any $\beta \in A^{0,1}(X,E)$
admits the unique decomposition: 
$$
\beta = \beta_{0} + \bar{\partial}\eta
\; \text{ with  $\beta_{0}$ \underline{harmonic} (with respect to $g_{X}$) and } \; 
 \eta \in A^{0}(E).
$$
Hence, every class  $[\beta] \in \mathcal{H}^{0,1}(X,E)$ is uniquely identified
by its harmonic representative $\beta_{0} \in [\beta]$ and moreover: $*_E\b_0\in C_2(X)$.
Thus, in analogy to \eqref{2.8*hodge_operator_intro}, we have the isomorphism:
$$
\mathcal{H}^{0,1}(X,E) \longrightarrow C_{2}(X)
:
[\beta] \longrightarrow  *_{E}\beta_{0},
$$ 
in other words, for $\alpha \in C_{2}(X) \subset A^{1,0}(X,E^{*})$ there exist a unique \underline{harmonic}
Beltrami differential $\beta_{0}$: $*_{E}\beta_{0}=\alpha$ or equivalently $*_{E}^{-1}\alpha=\beta_{0}$. 

Also notice that, for 
$[\beta]\in \mathcal{H}^{0,1}(X,E)$ with harmonic $\beta_{0} \in [\beta]$,
there correspond an (unique) element in $(C_{2}(X))^{*}$ defined as follows:
\begin{equation}\label{dual}
C_{2}(X)\longrightarrow \mathbb{C}
:
\alpha \longrightarrow \int_{X} \beta_{0} \wedge \alpha = 
\int_{X} (\beta_{0} +\bar{\partial}\eta) \wedge \alpha ,
\end{equation}
which indeed is well defined independently of any chosen element in the  class $ [\beta].$ In turn, we have that the dual space $(C_2(X))^*$ can be identified with the space of harmonic Beltrami differentials (with respect to $g_X$).\\

At this point, in view of \eqref{dim} and by recalling that the Teichm\"uller space  $\mathcal{T}_{\mathfrak{g}}(S)$ has the structure of a differential cell of real dimension $6(\mathfrak{g}-1)$ and  we have 
the well-known parametrization of $T^*(\mathcal{T}_{\mathfrak{g}}(S)),$ the cotangent bundle of 
$\mathcal{T}_{\mathfrak{g}}(S),$ given by the pairs:
$$(X,\alpha) \in \mathcal{T}_{\mathfrak{g}}(X)\times C_{2}(X),$$
see e.g. \cite{Jost} for details.
Consequently, in view of the isomorphism \eqref{C_kappa_X_isometry}, 
we derive that $T(\mathcal{T}_{\mathfrak{g}}(S))$ (the tangent bundle of $\mathcal{T}_{\mathfrak{g}}(S)$) is parametrized by the pairs: 
$$(X,[\beta]) \in \mathcal{T}_{\mathfrak{g}}(S) \times \mathcal{H}^{0,1}(X,E).$$ 

For $p\geq 1,$  we have
 the $L^{p}$-space of sections and forms valued on $E$, respectively 
as follows:
\begin{align}
&
L^{p}(X,E)
=
\{ \eta:X\longrightarrow E 
\; : \; 
\Vert  \eta \Vert_{L^{p}}:=
(\int_{X}\Vert \eta \Vert_{E}^{p}dA)^{\frac{1}{p}} < +\infty
\},
\notag
\\
&
L^{p}(A^{0,1}(X,E))
=
\{ 
\beta \in A^{0,1}(X,E)
\; : \; 
\Vert \beta \Vert_{L^{p}}
:=
(\int_{X}\Vert \beta \Vert_{E}^{p}dA)^{\frac{1}{p}} < +\infty
\},
\notag
\end{align}	
which define Banach spaces equipped with the given norm: 
$\Vert \cdot  \Vert_{L^{p}}.$ 

Also for $p\geq 1,$ we have the Sobolev space:
\begin{equation}\label{W_1_p}
W^{1,p}(X,E)
=
\{ \eta \in L^{p}(X,E)
\; : \; 
\bar{\partial} \eta \in L^{p}(A^{0,1}(X,E))
\}, 
\end{equation}
defining a Banach space equipped with the norm:
$$
\Vert \eta \Vert_{W^{1,p}}
=
\Vert \eta \Vert_{L^{p}} + \Vert \bar{\partial}\eta \Vert_{L^{p}},
\; \forall \; \eta \in W^{1,p}(X,E).
$$
Incidentally, we recall that, for the holomorphic line bundle $E=T_{X}^{1,0}$ in \eqref{E},  
the following Poincar\'e inequality holds,
\begin{equation}\label{poincare}
\Vert \eta \Vert_{L^{p}} \leq C_{p}\Vert \bar{\partial}\eta \Vert_{L^{p}},
\; \forall \; \eta \in W^{1,p}(X,E).
\end{equation}
for suitable $C_p>0$, see \cite{Huang_Lucia_Tarantello_2}.

 Furthermore, since $C_2(X)$ is finite dimensional, we have the equivalence of  all norms in $C_{2}(X),$ and it is usual
(by recalling the Weil-Patterson form \cite{Jost}) to consider the following $L^{2}$-norm:
\begin{equation}\label{norm}
\Vert \alpha \Vert_{L^{2}}
:=
(\int_{X} \langle  \alpha, \alpha\rangle dA)^{\frac{1}{2}}
\; \text{ for } \; 
\alpha \in C_{2}(X). 
\end{equation}
Indeed, it is conveniently computed with respect to a \underline{basis} in
$C_{2}(X)$ given as follows: 
\begin{equation}\label{basis}
\{s_{1},\ldots,s_N \} \subset C_{2}(X)
\; \text{ with } \; 
N=3(\mathfrak{g}-1)
:\;
\int_{X}\langle s_{j} , s_{k}\rangle dA
=
\delta_{j,k}
\end{equation}
with $\delta_{j,k}$ the Kronecker symbols. So, 
for $\alpha \in C_{2}(X)$, 
we may write:
$$\alpha=\sum_{j=1}^{\nu}b_{j}s_{j}, \;b_{j}\in \mathbb{C},
\; \text{ and } \; 
\beta_{0}=*_{E}^{-1}\alpha=\sum_{j=1}^{\nu}b_{j}*_{E}^{-1}s_{j}
$$ 
($\beta_{0}$ the associated harmonic Beltrami differential),
and we compute:
 \begin{equation*}
\Vert \alpha \Vert_{L^{2}}^{2}=\Vert \beta_{0} \Vert^{2}_{L^{2}}
=
\sum_{j=1}^{\nu}\vert b_{j} \vert^{2}.  
\end{equation*}
Clearly, any closed and bounded subsets of $C_2(X)$ (with respect to the $L^2$-norm) is compact. Thus, for example, if $\alpha_{n}\in C_2(X)$ satisfies 
$\Vert \alpha_{n} \Vert_{L^{2}}=1$ then it admits a convergent subsequence 
$\alpha_{n_{k}}\longrightarrow \alpha_{0}  \in C_{2}(X)$ 
with 
$\Vert \alpha_{0} \Vert_{L^{2}}=1$. 

\vskip 0.2cm
Next, we recall some well-known facts about divisors and some useful consequences of the Riemann-Roch theorem 
that we shall need
in the sequel, see e.g. \cite{Jost}, \cite{Miranda} and \cite{Narasimhan}.
\vskip 0.1cm
For given $\nu\in \NN$, let $X^{\nu} =X\times \cdots\times X$ be the Cartesian product of $\nu$-copies of $X$, and consider the symmetric product of $X$:
\begin{eqnarray*}
    X^{(\nu)}=X^{\nu} /\{ \text{permutaions}\},
\end{eqnarray*} which defines a smooth complex manifold of dimension $\nu,$ for the proof see Chapter 2
 page 236 of \cite{Griffiths_Harris}.

A non zero  \underline{effective divisor $D$ of degree $\nu$} on $X$ is given by the formal expression:
\begin{eqnarray}\label{divis}
    D=\sum_{j=1}^n n_j x_j,  \quad\, \text{ with } \, x_j\in X \quad \,  n_j \in \mathbb{N} \,  \quad \, \text{ and } \quad \sum_{j=1}^n n_j=\nu > 0.
\end{eqnarray} 
The support of $D$ is given by the set: $\text{ supp } D= \{x_1,\ldots, x_n\},$ (formed by \underline{distinct} point in $X$) while the positive integer $n_j\in \NN$ defines the \underline{multiplicity} of $x_j \in \text{supp }D,$ for $j=1,\ldots,n.$ 
Finally, the integer $\nu$ defines the  degree of $D$, namely: $deg (D):=\sum_{j=1}^n n_j=\nu$.

Any such divisor can be identified with the element of $X^{(\nu)}$ associated with the $\nu$-ple containing $n_j$ copies of the point $x_j$, $j=1,\ldots, n$. Hence, we shall refer to $X^{(\nu)}$ as the space of non zero  effective divisors of degree $\nu$.

Moreover, to any non trivial  $\alpha \in C_2(X)$, we can associate the so called  \underline{zero divisor} of $\alpha,$ denoted by $div(\alpha),$ as follows:
$$
div(\alpha)=\sum_{j=1}^n n_j q_j,
$$
where $\{q_1,\ldots, q_n\}$ are the \underline{distinct} zeroes of $\alpha$, and $n_j$ is the multiplicity of the zero $q_j$, $j=1,\ldots,n.$ 
From \eqref{2.03}, we have: $div (\alpha) \in X^{(4(\gg-1))}.$

\vskip0.2cm
Given $D_1\in X^{(\nu_1)}$ and $D_2\in X^{(\nu_2)}$ we say that $D_1\leq D_2$ if $\nu_1\leq \nu_2$, $\text{supp} D_1\subset \text{supp} D_2$ and the multiplicity at $x\in \text{supp} D_1$ is smaller or equal than the multiplicity of $x\in \text{supp } D_2$. 

For an effective divisor $D$ we define:
$$
    Q (D)=\{\a\in C_2(X): \text{div}(\a)\geq D\},
$$
with  the understanding that, for the trivial divisor $D=0,$ we have Q(0)=$C_2(X).$

Cleraly, \\
if $D_1\leq D_2$  then $Q(D_2) \subseteq Q(D_1)$.\\

It is a direct consequence of the Riemann-Roch theorem that,
\begin{eqnarray}\label{2.07*}
\begin{array}{l}
\text{ if }  1\leq \nu < 2(\gg-1), \\ \text{ then }    \dim_{\CC} Q(D)= 3(\gg-1)-\nu=\dim_\CC(C_2(X))-\nu. \quad \quad 
\end{array} 
\end{eqnarray}

Therefore, if $D_1\in X^{(\nu_1)}$ and $D_2\in X^{(\nu_2)}$  satisfy: $D_1\leq D_2$ and $1\leq \nu_1<\nu_2<2(\gg-1)$ then 
\begin{eqnarray}\label{2.06} 
\exists \, \a\in Q(D_1) \text{ but } \a\notin Q(D_2).
\end{eqnarray}

In particular for $q\in X$, letting $D=q$ and $Q(q):=Q(D)$  we see that,

\begin{eqnarray}\label{2.09}
Q(q) =\{ \a \in C_2(X) : \a(q)=0 \},
\end{eqnarray}
and
\begin{eqnarray}\label{2.07}
    \dim_\CC Q(q)= \dim_\CC C_2(X)-1.
\end{eqnarray}

Therefore, for all  $ q\in X \,\, \exists \,\alpha \in H^0(X,\otimes^2(K_X))$ such that $\alpha(q) \neq 0,$
namely:
$$
C_2(X) = H^0(X,\otimes^2(K_X)) \text{ is base point free.}  \quad \quad 
$$  
\vskip0.2cm
More generally, from \eqref{2.06} we can show that, if $1\leq \nu < 2(\gg-1)$ then for given distinct points $\{x_1,\ldots,x_\nu\}$ in $X$ we have:
\begin{eqnarray}\label{2.07bis}
\begin{array}{l}
\forall \, l\in \{1,...,\nu\} \,\, \exists \,\alpha_l \in C_2(X): \,\, \alpha_l(x_l) \neq 0 \,\, \text{ but }
 \alpha_l(x_j)=0 \quad \\
j \in \{1, ...,\nu \}\setminus\{ l \}.
\end{array}
\end{eqnarray}
By  \eqref{2.07bis}, we obtain in addition that, 
\begin{eqnarray}\label{2.07tris}
\begin{array}{l}
\text{if $1\leq \nu < 2(\gg-1)$ then for given distinct points} \, \{x_1,\ldots,x_\nu\} \subseteq X \\ \exists \,\alpha \in C_2(X): \,\,  \alpha(x_j) \neq 0,
\forall \,\,\,  j\in \{1, ...,\nu \}.
\end{array}
\end{eqnarray}
Indeed, in order to satisfy \eqref{2.07tris}, it suffices to take $\alpha = \sum_{l=1,...,\nu} \alpha_l$, with $\alpha_l$ in \eqref{2.07bis}.
\vskip0.4cm
By using \eqref{2.07}, we can define the \underline{Kodaira map}:
\begin{equation}\label{kodaira map} 
\tau:X\longrightarrow \mathbb{P}(\mathcal{H}^{0,1}(X,E))\simeq \mathbb{P}^{3(\mathfrak{g}-1)-1} = \mathbb{P}^{3\mathfrak{g}-4} 
\end{equation}  
 by associating to any $q \in X$ the element of
$\mathbb{P}(\mathcal{H}^{0,1}(X,E))$ identified by 
the ray of classes generated by 
$[\beta] \in \mathcal{H}^{0,1}(X,E)\setminus \{  0 \} $ with harmonic representative
$\beta_{0}\in [\beta]$ satisfying:
$$
\int_{X} \beta_{0}\wedge \alpha=0, \; \forall \; \alpha \in Q(q).
$$
Equivalently, by recalling \eqref{dual}, \eqref{2.07} and \eqref{2.09}, to any $q \in X$ the map $\tau$ will associate the ray of functionals in  
$(C_{2}(X))^{*}$ which admit the linear space $Q(q)$ as their kernel.
Such a map is holomorphic \cite{Donaldson_Book},\cite{Griffiths_Harris}, and we have:
\begin{lemma}[\cite{Tar_2}] 
The image $\tau(X)$ defines a complex curve into the projective space $\mathbb{P}(\mathcal{H}^{0,1}(X,E))$ of complex dimension
$3(\mathfrak{g}-1) -1 \geq 2$. 
\end{lemma}
Furthermore we point out the following useful "approximation" result:
\begin{lemma}\label{approssimazione}
Consider the effective divisors:  $D_k = \sum_{l=1}^{N} x_{k,l}$ and $D = \sum_{l=1}^{N} x_{l},$  with $x_{k,l} \to x_{l} $ in $ X, \,\, l =1,..., N$ and degree $N \in \{1, ... ,2 \gg -3 \}.$ 
Then for every $\alpha \in Q(D)$ there exists  $ \alpha_{k} \in Q(D_k)$ such that (up to subsequences) $ \alpha_{k} \to \alpha,$ as $k \to +  \infty. $

\end{lemma}
\begin{proof}
From \eqref{2.07*}  we know that, $dim(Q(D_k)) = dim(Q(D)) = 3(\gg-1) - N =: \nu.$ So, we let  $\{s_{1,k}, \ldots,s_{\nu,k}\}$ be an orthonormal basis for  $Q(D_k),$ i.e. 
$$ \int_X <s_{h,k},s_{j,k}>dA = \delta_{h,j} \ \  j,h = 1, \ldots,\nu.$$
Hence, along a subsequence, we can assume that,
$ s_{j,k} \to s_j,$ as $k \to + \infty$ and $$\int_X <s_h,s_j> = \delta_{h,j} \ \ j,h = 1,\ldots,\nu.$$
In particular, since $s_{j,k}(x_{l,k}) = 0,$ then by passing to the limit as $k \to + \infty,$  we find that, $s_{j}(x_l) = 0, \,\, \forall  \ l = 1 \ldots,N.$ That is, $s_j \in Q(D), \  \forall j=1, \ldots,\nu.$ 
As a consequence we derive that $\{s_1, \ldots,s_\nu\}$ defines an orthonormal basis for $Q(D).$ Hence for $\alpha \in Q(D),$ we can write: $ \alpha = \sum_{j=1}^{\nu} c_j s_j$ with $c_j \in \mathbb{C},$ and at this point it suffices to take $\alpha_k = \sum_{j=1}^{\nu} c_j s_{j,k} \in Q(D_k),$ to find  that, $\alpha_k \to \alpha,$ as $k \to + \infty.$ \end{proof}
   
\begin{remark}\label{differential}
For $ x_0 \in X$  let $ \alpha \in H^0(X, \otimes^2(K_X)))$ such that $\alpha(x_0) = 0.$ Then it is possible to define the differential $d_{x_0}(\alpha)$ of $\alpha $ in $x_0$ which is an element  of the fiber $(\otimes^3(K_X))_{x_0}.$ To this purpose, let $\alpha_0$ be a local holomorphic section of $K_X \otimes K_X$ which is nowhere zero on an open neighbourhood $U$ of $x_0.$ The function $ f := \frac{\alpha}{\alpha_0}$ is holomorphic in $U$ and we have $d_{x_0}(f)= \partial_{x_0}(f) + \overline{\partial}_{x_0}(f) = \partial_{x_0}(f) \in ({T_{x_0}^{1,0}})^* = (K_X)_{x_0}.$ 
Then we define the differential of $\alpha$ at $x_0$ as given by: $$d_{x_0}(\alpha):= \partial_{x_0}(f) \otimes \alpha_0(x_0) \in (\otimes^3(K_X))_{x_0}.$$  

Observe that this definition does not depend on the choice of the local nowhere vanishing section $\alpha_0.$ In fact let $\alpha_1(x) = e(x) \alpha_0(x)$   where $e$ is a nowhere zero holomorphic function in $U.$ Then $$d_{x_0} (\frac{\alpha}{\alpha_1}) \otimes \alpha_1(x_0) = \frac{1}{e(x_0)} d_{x_0}(\frac{\alpha}{\alpha_0}) \otimes (e(x_0) \alpha_0(x_0)) = d_{x_0} (\frac{\alpha}{\alpha_0}) \otimes \alpha_0(x_0), $$

If we have a local holomorphic chart centred at  $x_0$ then we can choose $\alpha_0 = dz^2$ and locally, by letting  $\alpha = a(z)dz^2$ with $a(0)=0,$ we find $d_{x_0} \alpha = \frac{\partial a}{\partial z}(0) dz^3_0.$

By using the Chern connection of the hyperbolic metric on $K_X^2$ it is possible to define the differential $d_{x_0} \alpha \in (K_X)_{x_0}^3$ even in case $\alpha(x_0) \neq 0$. 
In fact, in the chosen local holomorphic  coordinates  satisfying (\ref{hypcoord}) then the local expression of $d_{x_0} \alpha$ is given as above, for details see Section 12 in Chapter 5 of \cite{Demailly}). 
\end{remark}
\qquad

We have:
\begin{lemma}\label{zeros}
Let $D = \sum_{j=1}^{m} n_j x_j\,$ be an effective divisor (see \eqref{divis}) with $$0 < deg(D) =\nu < 2(\gg-1)$$ then:
$$ \begin{array}{l}
(i) \ \mbox{if $n_j = 1$ then there exists $\alpha_j^{(1)} \in Q(D - x_j)$ such that  $\alpha_j^{(1)}(x_j) \neq 0;$}  \\ 
(ii) \ \mbox{if $n_j\geq 2,$ then there exist $\alpha_j^{(1)} \in Q(D - n_j x_j)$ and  $\alpha_j^{(2)} \in Q(D - n_j x_j)$}\\    \mbox{\ such that \ $\alpha_j^{(1)}(x_j) \neq 0$  and $d_{x_j} \alpha_j^{(1)} = 0$} 
\mbox{ \ while \   $\alpha_j^{(2)}(x_j) =0$ and $d_{x_j} \alpha_j^{(2)} \neq 0.$}
\end{array}
$$
\end{lemma}
The definition of $d_{x_j}$  is given in Remark \ref{differential} above.\\ 
\begin{proof}
To prove  (i) we observe that since $deg(D) < 2( \gg-1),$ then we can use (\ref{2.06}) to find that, $ dim ( Q(D - x_j) ) > dim ( Q(D) )$. Consequently, there exists $\alpha_j \in Q(D - x_j) \setminus Q(D),$ and so necessarily: $\alpha_j(x_j) \neq 0.$ Hence, in this case we simply set: $\alpha_j^{(1)} = \alpha_j.$ Similarly to prove (ii) we observe that   
 $dim(Q(D - n_j x_j)) > dim(Q(D - (n_j-1) x_j)) > dim(Q(D -(n_j-2) x_j )),$ so there exists $\alpha_1 \in C_2(X)$ such that $\alpha_1 \in Q(D - n_j x_j) \setminus Q(D- (n_j-1) x_j) $ i.e. $\alpha_1(x_j) \neq 0;$ and
$\alpha_2 \in Q(D - (n_j-1) x_j) \setminus Q(D - (n_j-2) x_j)$ i.e. $\alpha_2(x_j) = 0$ and  $d_{x_j} \alpha_2\neq 0.$ 
So in this case  we let $\alpha_j^{(2)} = \alpha_2$ and 
$\alpha_j^{(1)} = \alpha_1 - \left(  \frac{d_{x_j} \alpha_1}{d_{x_j} \alpha_2} \right)  \alpha_2.$    
\end{proof}

Lemma \ref{zeros} will be very useful in the sequel, since it implies the following decomposition property: 
\begin {equation*}
\mbox{ if } D = \sum_{j=1}^{m} n_j x_j \mbox{with}  \,\,\, 1\leq n_j\leq 2, \,\,\,\,\, j=1,...,m; \,\, \,\,\mbox{and} \,\,0< degD < 2(\gg-1),
\end{equation*}
then
\begin {equation}\label{decomposition2}  
 C_2(X) =Q(D) \oplus Y, \,\,\mbox{with} \,  Y= span\{ \alpha_j^{(1)} \,\, \text{for} \,\,\,  j=1,...,m \} \oplus  \{ \alpha_j^{(2)} \,\,\text{for} \,\,\, j : n_j =2 \}
\end{equation}

Indeed, the holomorphic quadratic differentials of Lemma \ref{zeros} are linearly independent, so that dim Y= $\nu $. Furthermore, for any $\alpha \in C_2(X),$ by using the usual local holomorphic z-coordinates around the point $x_j$ (centered at the origin)  we may write:
\begin {equation}\label{local}
\alpha =a_j(z) dz^2, \,\,\, \alpha_j^{(1)}=a_{j}^{(1)}(z) dz^2, \,\,\,\, \alpha_j^{(2)}=a_{j}^{(2)}(z) dz^2;
\end{equation}
where $a_j(z), \,\,\, a_{j}^{(1)}(z) \,\,\, \text{ and } \,\, a_{j}^{(2)}(z)$ are holomorphic functions around the origin, with $a_{j}^{(1)}(0)\neq 0$ and $(a_{j}^{(1)})'(0)= 0,$ while $a_{j}^{(2)}(0)= 0$ and  $(a_{j}^{(2)})'(0)\neq 0.$
We let:
\begin {equation}\label{lambda}
\lambda^{(1)}_j= \frac{a_j(0)}{a_{j}^{(1)}(0)} \,\, \text{for} \,\,\,  j=1,...,m,\,\,\quad \lambda^{(2)}_j= \frac{a'_j(0)}{(a_{j}^{(2))})'(0)} \,\, \text{for} \,\,\, j: n_j=2.
\end {equation} 
and notice that such values are uniquely identified independently of the local chart considered around the point $x_j.$  Moreover:
$$\alpha - ( \sum_{j=1}^{m}\lambda^{(1)}_j \alpha_j^{(1)} + \sum_{j: n_j=2}\lambda^{(2)}_j \alpha_j^{(2)})\in Q(D).$$
Hence,
\begin {equation}\label{decom}
\alpha = \sum_{j=1}^{m}\lambda^{(1)}_j \alpha_j^{(1)} + \sum_{j: n_j=2}\lambda^{(2)}_j \alpha_j^{(2)} + \hat{\alpha}
\end {equation}
with unique $\hat{\alpha} \in Q(D).$
\vskip0.2cm

Finally,  in Appendix 3 we prove a nice geometric description about the projectivization of all  "dual" decompositions of the type  \eqref{decomposition2}, for any divisor $D.$ The following holds:

\begin{proposition}{\label{***}}
       For any $\nu\in \{1,\cdots, \gg-1\}$ let  $\tilde \S_\nu\subset \PP(\H^{0,1}(X,E))$ be the $\nu$-secant variety of $\tau(X)$ (cf. \cite{ACGH}) we have:
     \begin{eqnarray}\label{1}
      \tilde \S_1= \tau(X)\subset \tilde \S_2\subset \cdots \tilde\S_{\gg-1},~~ \text{dim}(\tilde \S_\nu) \leq 2\nu-1,
     \end{eqnarray}($\tau(X)$ the image of the Kodaira map in \eqref{kodaira map}) and 
     \begin{eqnarray}\label{2}
     [\b]_\PP \in \tilde \S_\nu \longleftrightarrow    \exists \text{ divisor } D\in X^{(\nu)}: \int_X \b\wedge \a =\int_X \b_0 \wedge \a =0 
     \end{eqnarray} $\forall \a\in Q(D),$  ($\b_0 \in [\b]$ the associated harmonic Beltrami differential). 
\end{proposition}

\section{The Donaldson functional and the statement of the Main Results.}\label{statements_main_results}
For now on we fix the 
pair: $(X,[\beta]) \in \mathcal{T}_{\mathfrak{g}}(S) \times \mathcal{H}^{0,1}(X,E),$ with 
$\beta_{0} \in [ \beta ]$ the harmonic representative of the class $[\beta].$\\ 

We shall be looking for a solution pairs $(u,\a)$ of \eqref{0.6}, \eqref{0.8} subject to the constraint:
$*_E^{-1}(e^{-u}\a) \in [\b].$
Equivalently, we ask the pull back metric $g$ on $X$ and the $(2,0)$-part $\a$ of the second fundamental form $II$ to satisfy:
$$
g=e^u g_X\quad \text{ and } \, \a=e^u*_E(\b_0+\bar\p \eta), \text{ with suitable }\eta \in A^0(E).
$$
Correspondingly, we can formulate the ``constraint" Gauss-Codazzi equations (with respect to the given pair $(X,[\beta])$)  in terms of $(u,\eta)$ as follows: 
\begin{equation}\label{system_of_equations_intro}
\left\{
\begin{matrix*}[l]
\Delta u +2 -2te^{u} -8e^{u}\Vert \beta_{0}+\overline{\partial}\eta \Vert^{2} =0  &  \;\text{ in }\;  &  X,  \\
\overline{\partial}(e^{u}*_{E}(\beta_{0}+\overline{\partial}\eta))=0, &  \;\text{}\;  &   \\ 
\end{matrix*}
\right.
\end{equation} with $t=1-c^2$.

Incidentally, observe that the Beltrami differential 
$\beta_{0}+\bar{\partial}\eta \in [\beta]$ satisfying the second equation in \eqref{system_of_equations_intro}  is harmonic with respect to the metric 
$h=e^{\frac{u}2}g_{X}$. 

With this point of view, it can been shown that the system \eqref{system_of_equations_intro}
can be formulated in terms of  Hitchin self-duality equations \cite{Hitchin}
for a suitable nilpotent $SL(2,\mathbb{C})$ Higgs bundle, 
we refer to \cite{Alessandrini_Li_Sanders}, \cite{Huang_Lucia_Tarantello_2}
for details and also we mention \cite{Li}  for related issue concerning  minimal immersions.

On the other hand, it is easy to check that solutions of \eqref{system_of_equations_intro} correspond to critical points of the following \underline{Donaldson functional} (in the terminology of \cite{Goncalves_Uhlenbeck})
\begin{equation}\label{F_t}
F_{t}(u,\eta)
= 
\int_{X}
\left(\frac{\vert \nabla u \vert^{2}}{4}
-
u
+
te^{u}
+
4e^{ u}\Vert \beta_{0} + \overline{\partial} \eta \Vert^{2}
\right)
\,dA
, 
\end{equation}
$t\in \R$, with ``natural" (convex) domain:
$$
\begin{array}{l}
\Lambda
=
\left\{  
(u,\eta)\in H^{1}(X) \times W^{1,2}(X,E)
\; : \; 
\int_{X}e^{u}
\Vert \beta_{0} + \bar{\partial}\eta \Vert^{2}\,dA < \infty
\right\}, 
\end{array}
$$
where $H^{1}(X)$ is the usual Sobolev spaces of function of $X$
and $W^{1,2}(X,E)$ is the Sobolev space of sections of $E$ (see \eqref{W_1_p}).  

We refer to \cite{Tar_2} for a detailed discussion about the Gateaux differentiability  of $F_t$ along "smooth" directions and the corresponding notion of "weak" critical point and relative regularity.

For $t>0$ the functional $F_{t}$ is clearly bounded from below in $\Lambda$, however this may no longer be the case for $t\leq0$.

Firstly anticipated in \cite{Goncalves_Uhlenbeck} and then established rigorously in \cite{Huang_Lucia_Tarantello_2},  we have that, for $t>0$ the Donaldson functional  $F_{t}$ admits a unique (smooth) critical point $(u_t, \eta_t)$ corresponding to its global minimum in $\Lambda$. 
Clearly, Theorem \ref{thm_A}  is a direct consequence of this fact. 
Moreover, as discussed in \cite{Huang_Lucia_Tarantello_2} and \cite{Tar_2}, from Theorem \ref{thm_A} we can deduce useful algebraic consequences about the parametrization of  all possible irreducible representations:
$$\rho:\pi_{1}(S)\longrightarrow PSL(2,\mathbb{C})$$ with
$PSL(2,\mathbb{C})$ the
(orientation preserving) isometry group of 
$\mathbb{H}^{3},$ 

see also \cite{Uhlenbeck}, \cite{Loftin_Macintosh_1}, \cite{Loftin_Macintosh_2} and \cite{Taubes} for more applications in this direction and related issues.

On the other hand, to find (CMC) c-immersions of $X$ with  
$|c| \geq 1$,
we need to investigate the existence of critical points for the functional $F_{t}$ when we take $t \leq 0$.  
This is a delicate task even for $t=0$, as the functional:
$$
F_{0}(u,\eta)
=
\int_{X}
\left(
\frac{1}{4}\vert \nabla u \vert^{2}
-
u
+
4e^{u}\Vert \beta_{0} + \overline{\partial} \eta \Vert^{2}
\right)
\,dA,
$$
may not admit any critical point in $\Lambda,$ in fact it can happen that the system 
\eqref{system_of_equations_intro} admits no solutions for $t=0.$

This is indeed the case, if we take $[\beta]=0$ (i.e. $\beta_{0}=0$) 
where we  find:
$u_{t}=\ln \frac{1}{t}\rightarrow +\infty$, $\eta_{t}=0$ and
$F_{t}(u_{t},\eta_{t})\rightarrow -\infty$, 
as $t\rightarrow 0^{+},$ 
and in fact \eqref{system_of_equations_intro} admits no solutions for $t=0$ and
$\beta_{0}=0$.

More generally, from \cite{Tar_2} we know about the continuous dependence of the pair $(u_t,\eta_t)$ with respect to the parameter $t\in(0,+\infty)$, and the existence and uniqueness of a (smooth) critical point for $F_0$ is actually equivalent to the continuous extension of $(u_t,\eta_t)$ at $t=0,$ as the following holds (see Theorem 8 in \cite{Tar_2}):
\begin{thmx}[Theorem 8 \cite{Tar_2}]\label{thmprimo}
If $(u_{0},\eta_{0})$ is a solution for the system
\eqref{system_of_equations_intro} with $t=0$, then
\begin{enumerate}[label=(\roman*)]
\item $(u_{t},\eta_{t})\rightarrow (u_{0},\eta_{0})$ uniformly in $C^{\infty}(X), $
as $t\rightarrow 0^{+}$;
\item
$F_{0}$ is bounded from below in $\Lambda$ and attains its global minimum at $(u_{0},\eta_{0})$ which defines its only critical point.\\ 
Hence,
$(u_{0},\eta_{0})$ is the \underline{only} solution of \eqref{system_of_equations_intro} with $t=0.$

\end{enumerate}
\end{thmx}

Therefore, to identify possible critical points for $F_{0}$, 
we must investigate  when the pair $(u_{t},\eta_{t})$ 
survives the passage to the limit, as $t\to 0^{+}$.  

To this purpose, for fixed $[\beta] \in \mathcal{H}^{0,1}(X,E) \setminus \{ 0 \} $
with harmonic representative $\beta_{0} \in [\beta]\neq 0$
we set:
$$
\beta_{t}=\beta_{0}+\overline{\partial}\eta_{t} \in A^{0,1}(X,E)
\; \text{ and } \; 
\alpha_{t}=e^{ u_{t}}*_{E}\beta_{t} \in C_{2}(X) \setminus \{ 0 \}.
$$

For $t>0$, we let $D_t=\text{div}(\a_t)$ be the zero divisor of $\a_t$ 
and let $Z_{t}$ be its support. Namely, $Z_t$ is the finite set of \underline{distinct} zeroes of $\alpha_{t}$, 
whose multiplicities adds up to $4(\mathfrak{g}-1)$ (see \eqref{2.03}). 
In terms of the fiberwise norm, for $\alpha_{t}$ we have:
$$\Vert \alpha_{t} \Vert(q)=\Vert \alpha_{t} \Vert_{E^{*}}(q)>0,
\; \forall \; q\in X\setminus Z_{t}.$$ 

Moreover we let,
$$
s_{t}\in \R \; : \; 
e^{ s_{t}}
=
\Vert \alpha_{t} \Vert_{L^{2}}^{2}
\; \text{ and } \; 
\hat{\alpha}_t
=
\frac{\alpha_{t}}{\Vert \alpha_{t} \Vert_{L^{2}}}
=
e^{-\frac{s_{t}}{2}}\alpha_{t},
$$
where $\Vert \alpha \Vert_{L^{2}}$ is the $L^{2}$-norm of $\alpha_t\in C_{2}(X)$ (see \eqref{norm}) and we have: $\text{div}(\widehat\a_t)=\text{div}(\a_t)=D_t$.

In order to attain an accurate asymptotic description 
about the behaviour of $(u_{t},\eta_{t})$, 
as $t\rightarrow 0^{+}$, 
we shall need to account for possible blow-up phenomena (cf. \cite{Brezis_Merle}) for
\begin{equation}\label{u_to_xi}
\xi_{t}:=-u_{t}+s_{t},
\end{equation} satisfying the Liouville-type equation:
\begin{eqnarray}\label{1.19a}
-\Delta \xi_t=8\|\widehat\a_t\|^2 e^{\xi_t} -f_t \text{ in } X,
\end{eqnarray}with $f_t= 2(1-t e^{u_t})$ satisfying: $0\leq f_t \leq 2$ in $X$.
\vskip0.5cm

We can combine the blow-up information in Theorem 3 of \cite{Tar_1} and Theorem \ref{thmprimo} above, to conclude the following facts about $\xi_t$ as $t\to 0^+$: 
\vskip 0.5cm
FACT 1. If $\xi_t$ does not blow up (i.e $\limsup_{t \to 0^+}max_{X} \ \xi_t < + \infty$)  then $(u_{t},\eta_{t})\rightarrow (u_{0},\eta_{0})$ uniformly in $C^{\infty}(X),$
as $t\rightarrow 0^{+}$; and $(u_{0},\eta_{0})$ is the unique critical point of $F_0$ in $\Lambda$ corresponding to its global minimum point.
\vskip 0.5cm
FACT 2. If $\xi_t$ does blow up (i.e. $\limsup_{t \to 0^+} max_{X} \ \xi_t = + \infty, $) then\\ $\liminf_{t \to 0^+} max_{X} \ \xi_t = + \infty, $ and along \underline{any} sequence $t_k \to 0^+,$ we have that    
 $ \xi_k := \xi_{t_k}$ admits a finite set $\mathcal{S}$ of blow up points, (depending possibly on the sequence  $t_k$) such that the following holds for any \underline{blow-up point} $x \in \mathcal{S}:$ 
 $$\lim_{k \to + \infty}   \ max_{B(x;r)} \ \xi_k = + \infty, \, \, \text{ \  for all small \ } r > 0;$$   
the \underline{ blow-up mass} at $x$ is defined by: 
\begin{equation}\label{sigma_q} 
\sigma(x)
:=
\lim_{r \to 0^{+} }
\left(
\lim_{k \to +\infty }8 \int_{B(x;r)} \|\widehat\a_{t_{k}}\|^2 e^{\xi_k}dA\right) 
\end{equation}
and it satisfies the following \underline {quantization} property: 
$$
\sigma(x)\in  8\pi \mathbb{N};
$$
see \cite{Tar_1}, \cite{Tar_2} and \cite{Tar_3} for details.

Thus, we need to understand when "blow-up" occurs and see when it is possible to rule it out. 

For this purpose we introduce the following notion of "blow-up" divisor:
\begin{definition} \label{blowupdiv}

If $\xi_{k}$ blows-up with (non-empty) blow up set $\mathcal{S}$ then  we define the \underline{blow-up divisor} of $\xi_{k}$ the formal sum:
$$ D=\sum_{ x \in \mathcal{S}} n_x x \in X^{(\nu)} \mbox{ \  with \ } n_x=\frac 1 {8\pi} \sigma(x)\in   \mathbb{N}, $$
with  supp $D=\mathcal{S}$ and 
$deg(D) =  \sum_{ x \in \mathcal{S}} n_x =\nu.$
\end{definition}
Since after integration of \eqref{1.19a} we have:
$$
 8\int_X \|\widehat\a_{t_{k}}\|^2 e^{\xi_k}dA
=   8 \int_X \e^{u_{t_{k}}}
\Vert \beta_{0} + \overline{\partial}\eta_{t_{k}} \Vert^{2}dA = 
 8\pi(\gg-1) - 2t_{k} \int_X \e^{u_{t_{k}}}dA
$$ 
we find:

\begin{equation}\label{degblowup}   
\nu=deg(D) =  \sum_{ x \in \mathcal{S}} n_x \in \{1,...,\gg-1\}
\end{equation} 
that is, any blow-up divisor belongs to $X^{\nu}$ with $\nu\in \{1,...,\gg-1\}.$

One of the main contribution of \cite{Tar_2} concerns the case of blow-up with minimal blow-up mass $8\pi.$ Namely when the corresponding blow-up divisor is formed by (blow-up) points with multiplicity one.
This is always the case for genus $\gg=2,$ where by (\ref{degblowup}), we know that necessarily the corresponding blow-up divisor is formed by a single point of multiplicity one.  

The following was established in \cite{Tar_2} : 
\begin{thmx} [\cite{Tar_2}]\label{thmB.3} 
Suppose that $\xi_{k}$ in \eqref{u_to_xi} admits a (non-empty) blow-up set $\mathcal S$
with the corresponding blow up divisor satisfying: $$D_1 = \sum_{x \in \mathcal{S}} x.$$ 

Then there exist a divisor $ 0 \neq \wt D \leq D_1 $ such that,  
\begin{equation} \int_{X} \beta_{0} \wedge \alpha = 0, 
\; \forall \;   \alpha \in Q(\wt D). \label{orthod2} \end{equation}
In particular if the genus $\gg=2$ then $ \mathcal{S} =\{ x \},$ $\wt D = D_1 = x$ and $[\b]_\PP=\tau(x)\, \in \tau(X).$
\end{thmx}

Recall that $[\b]_\PP$ is defined in \eqref{betap}, and indeed when $\wt D = D_1 = x$ then by \eqref{kodaira map}, the orthogonality condition \eqref{orthod2} just states that $[\b]_\PP$ belongs to the image of the Kodaira map at $x$.\\

Notice that Theorem \ref{thmB.3} is a direct consequence of Theorem 9  and Theorem 10 in \cite{Tar_2}.

Our main contribution will be to establish the following extension of  Theorem \ref{thmB.3}: 
\begin{thm}\label{thm3}
    Suppose that $\xi_{k}$ in \eqref{u_to_xi} admits (non-empty) blow-up set $\mathcal S$ with corresponding
 blow up divisor: $D_2=\sum_{x\in \mathcal{S}}  n_x x.$
If 
    \begin{eqnarray}
        \label{1.0*} 
 n_x\in \{1, 2\} \,\,\,\,\forall x\in\mathcal{S},
    \end{eqnarray}then there exists a divisor  $ 0 \neq \wt D \leq D_2$ such that,
    \begin{eqnarray}\label{1.0**} \int_X \b_0\wedge \a =0 ,\,\,\forall \a \in Q(\wt D).
    \end{eqnarray}
    
\end{thm}
It is reasonable to expect that the conclusion \eqref{1.0**} holds without the assumption \eqref{1.0*}. In fact, Theorem \ref{thm3} provides the first non-trivial step that could allow one (inductively and by further scaling) to attain \eqref{1.0**} in the general case.\\

Since the pairs $(u_t,\eta_t)$ minimize the functional $F_t$, one would expect also that the blow-up set should consist always of only one point possibly located away from the zero set of $\widehat\alpha_0$ := $\lim_{k \to +\infty } \widehat\a_{t_k}.$ 

This would provide a dramatic simplification, since when the blow-up set $\mathcal{S} $ does not contain any zeroes of $\widehat{\alpha_0},$ then the blow-up divisor of $\xi_k$ is given exactly as in Theorem \ref{thmB.3} and one could profit directly from the very accurate blow-up analysis available in \cite{Brezis_Merle}, \cite{Li_Harnack}, 
 \cite{Li_Shafrir}, \cite{Chen_Lin_1} and \cite{Chen_Lin_2}.\\
Even when the blow-up set $\mathcal{S}$ contains a zero of $\widehat\a_0$ which is not a  limit of  \underline{distinct} zeros  of  $\widehat\a_t,$ to proceed towards \eqref{1.0**}, one could still rely on the asymptotic analysis established in \cite{Chen_Lin_3}, \cite{Chen_Lin_4}, \cite{Bar_Tar_JDE},  \cite{Kuo_Lin}, and more recently in \cite{Wei_Zhang_1}, \cite{Wei_Zhang_2} and \cite{Wei_Zhang_3}.

However here we must deal with a new and 
most delicate  blow-up  situation which occurs when the blow-up set $\mathcal{S}$ contains a zero of $\widehat\a_0$ where several different zeroes of $\widehat\a_t$ accumulate, as $t \to 0^+$. We shall refer to this situation as the blow-up at "collapsing" zeroes. Below we describe the new analytical difficulties one needs to resolve in order to attain a grasp for the corresponding blow-up behaviour (see  \cite{Lin_Tarantello}, \cite{Lee_Lin_Tarantello_Yang} and \cite{Tar_1}) yielding to Theorem \ref{thm3}.

\bigskip

In view of Theorem 6 in \cite{Tar_2} the following holds:

\begin{thmx}[\cite{Tar_2}]\label{thmB.1}   
If $\gg=2,$ then for every $[\b]\in \H^{0,1}(X,E)\setminus\{0\}$ the functional $F_0$ is bounded from below in $\Lambda$. 
Moreover, if  $[\b]_\PP\notin \tau(X)$ then $F_0$ attains its infimum in $\Lambda$ at a (smooth) minimum point which is its only critical point.  
\end{thmx} 
\vskip0.3cm

By showing that, for $\gg=2$ the Donaldson functional is invariant under bi-holomorphisms of X (see Appendix 2) and by recalling the unique biholomorphic hyperelliptic involution (see \cite{Miranda}, \cite{Griffiths_Harris})
\begin{eqnarray*}
    j: X\to X
\end{eqnarray*} with exactly $2(\gg+1)=6$ distinct fixed points, 
we obtain the following improvement of Theorem \ref{thmB.1}:

\begin{thm}\label{thmg2} 
Assume that $\gg=2$ and let $[\b]\in \H^{0,1}(X,E)\setminus\{0\}.$ \\
$$ \begin{array}{l} 
\text{(i) Either: }
\, \lim_{t \to 0^+} max_{X} \xi_t = + \infty
\text{ (blow-up), and then there exist a \underline{unique} point} \\ q \in X: 

\mbox{$[\beta]_{\mathbb{P}} = \tau(q) \mbox{ \ with \ }  j(q) = q,$ }  and \, 
 
\lim_{t \to 0^+} max_{K} \xi_t = - \infty,
\mbox{ for any compact}\\ K \subset X \setminus \{q\}.\\
\text{(ii) Or:} \,\, \lim_{t \to 0} max_{X} \xi_t < + \infty, 

\text{\ then $F_0$ attains its infimum in $\Lambda,$ }\\
\text{and the (smooth)  minimum point is the only critical point of $F_0.$} \end{array} $$ 
In particular, if  $[\beta]_{\mathbb{P}}\not \in \{
\tau(q),  \  q \in X \mbox{ \ with \ } j(q)=q$\} then alternative (ii) holds.  
\end{thm}

Concerning the case of higher genus, we can use Proposition \ref{***}  to deduce that, if $[\b]_{\PP}$ satisfies \eqref{1.0**} for some $ \wt D\in X^{(\nu)}$ with $\nu\in\{1,\cdots, \gg-1\}$, then $[\b]_{\PP}$ belongs to the subvariety $\wt\S_\nu$ in $\PP(\mathcal{H}^{0,1}(X,E))$ with $\text{dim}(\wt \S_\nu) \leq 2\nu -1.$  Moreover, we have: $\wt\S_1=\tau(X)\leq \wt \S_2 \subseteq \cdots \subseteq \wt \S_{\gg-1}$, where $\tau(X)$ is the image of the Kodaira map in \eqref{kodaira map}. In particular  $\PP(\mathcal{H}^{0,1}(X,E)\setminus \wt\S_{\gg-1}$  is an open Zariski dense subset.

So by virtue of Theorem \ref{thm3}  we can derive the analogous of Theorem \ref{thmB.1} for the case of genus $\gg=3.$

 \begin{thm}\label{thm1.1} 
Assume $\gg= 3$ and let $[\b]\in \H^{0,1}(X,E)\setminus\{0\}$  satisfy  $[\b]_\PP\notin \wt\S_{\gg-1}.$ Then $F_0$ is bounded from below and attains its infimum in $\Lambda.$ Moreover the (smooth)  minimum point is the only critical point of $F_0$. 
\end{thm}
Clearly, Theorem \ref{thm1} follows from Theorem \ref{thmg2}
 and   Theorem \ref{thm2}  follow from Theorem \ref{thm1.1}.
 
Reasonably, we expect Theorem  
\ref{thm1.1} to remain valid for any genus $\gg\geq3.$

\section{Asymptotics}  \label{Asymptotics}

As already mentioned, from Theorem \ref{thmprimo}, we know that the functional $F_0$  admits a (weak) critical point if and only if as $t\to 0^+,$ the minimum pair $(u_t,\eta_t)$ of $F_t$ converges in $\Lambda.$ \\ 
So, the main purpose of this section is to describe  as accurately as possible the asymptotic behaviour of $(u_t,\eta_t)$ as $t\to 0^+.$  For this purpose we let:
\begin{equation*}
u_{t}
=
w_{t} 
+
d_{t}
,
\; \text{ with } \; 
\int_{X} w_{t}dA=0
\; \text{ and } \;  
d_{t}=\fint_{X}u_{t}dA
\end{equation*}
\begin{equation*}
\beta_{t}=\beta_{0}+\bar{\partial}\eta_{t} \in A^{0,1}(X,E)
\; \text{ and } \; 
\alpha_{t}=e^{u_{t}}*_{E}\beta_{t}\in C_{2}(X)
\end{equation*}
and recall that, 
\begin{equation*}
s_{t}\in \R \; : \; e^{ s_{t}}=\Vert \alpha_{t} \Vert_{L^{2}}^{2} \,\,\,\text{and} \,\,\, \hat{\alpha_t}=e^{-s_{t}/2}\alpha_{t}.
\end{equation*}
The following easy bounds where derived in \cite{Tar_2}:
\begin{lemma}
\begin{equation}\label{property_v} \begin{array}{l}
\; \forall \;   q  \in [ 1,2 )  \;
\; \exists \; 
C_{q}>0 \; : \; \Vert w_{t} \Vert_{W^{1,q}(X)}\leq C_{q},
\\
w_{t}\leq C \; \text{ in  } \; X \text{ and } \, \; te^{d_t}\leq 1, \\ \int_{X}e^{- u_{t}}dA \geq
C \fint_{X} \Vert \beta_{0} \Vert^{2}dA,\\
s_t \leq d_t + C,\\
\end{array} \end{equation}
for a suitable constant $C>0.$

\end{lemma}	
\begin{proof}
See Lemma 3.7 and Remark 3.1 of \cite{Tar_2}.    
\end{proof}

It was observed in \cite{Tar_2} that the map:
\begin{eqnarray*}
    t\to 4\int_X e^{u_t} \|\b_0+\bar \p \eta_t\|^2 dA=4\pi(\gg-1)-t\int_X e^{u_t} dA
\end{eqnarray*} is decreasing in $(0,+\infty)$ (see Lemma 3.6 of \cite{Tar_2}), and so it is well defined the value:
\begin{eqnarray}\label{rho}
    \rho([\b])=\rho([\b_0]):=4 \lim_{t\to0^+} \int_X e^{u_t} \|\b_0+\bar\p \eta_t\|^2 dA =4\lim_{t\to0^+} \int_X e^{\xi_t} \|\wh\a_t\|^2 dA.
\end{eqnarray}Notice in particular that,
\begin{eqnarray}\label{3.1*}
\rho([\b])\in [0, 4\pi(\gg-1)] \text{ and } \rho([\b])=0 \Longleftrightarrow [\b]=0.
\end{eqnarray}
Furthermore, it was shown in \cite{Tar_2} that in case $F_0$ is bounded from below then necessarily: 
\begin{eqnarray}\label{3.1**}
\rho ([\b])= 4\pi(\gg-1) \;
 \text{ i.e. } \, \lim_{t\to0^+} t\int_X e^{u_t}=0.
\end{eqnarray}

In view of the estimates in \eqref{property_v}, along a (positive) sequence
$t_{k}\longrightarrow 0^{+}$,  
for $$d_{k}:=d_{t_{k}}, \,\; u_{k}=u_{t_{k}}, \,\; w_k:=w_{t_k},$$  
we may assume that,
\begin{equation*}
w_{k} \longrightarrow w_{0}
\; \text{ and } \; 
e^{w_{k}}\longrightarrow e^{w_{0}} 
\; \text{ pointwise and in   } \; 
L^{p }(X),\; 
\end{equation*}
\begin{equation}\label{6.20_prime}
t_{k}e^{d_{k}} \longrightarrow \mu \geq 0
\; \text{ and so } \;
t_{k}e^{u_{k}}
\longrightarrow 
\mu e^{w_{0}}
\; \text{ pointwise and in  } \; 
L^{p}(X), 
\end{equation}
for any $p>1,$ and as $k\longrightarrow +\infty$.

In addition, it follows from the discussion in Section 1.1, that 
for suitable $1\leq N \leq 4(\mathfrak{g}-1)$ and $k$ large, we may assume that the zero divisor of $\wh\a_k:=\wh \alpha_{t_{k}}\in C_{2}(X) \setminus \{ 0 \}$ satisfies:

\begin{eqnarray*}
    \text{div}(\wh \a_k)=\sum_{j=1}^N n_j z_{j,k}\;\, \text{and} \,\, \sum_{j=1}^{N} n_{j}=4(\mathfrak{g}-1).
\end{eqnarray*}
with $z_{j,k}\neq z_{l,k} \; \mbox{ for  $j\neq l\in \{1,\ldots,N \}$}$ the (distinct) zeroes of $\wh \a_k$ and $n_{j}\in \mathbb{N}$ the  multiplicity relative to $z_{j,k}, \; j\in \{1,\ldots,N \}$. 

Moreover, $\text{ as } k\to +\infty,$ we may let,
$$
\wh\a_k\to \wh\a_0, \;\; z_{j,k}\longrightarrow z_{j},
\; \text{ with } \; \wh\a_0(z_{j})=0, \;\, j \in \{ 1,\ldots,N \}.
$$

Since the total multiplicity of each $z_{j}$ adds up to
the value: $4(\mathfrak{g}-1)$,
we see that $\wh\a_0$ cannot vanish anywhere else.
However, while the points in $Z^{(k)}:=\{ z_{1,k},\ldots,z_{N,k} \},$ are distinct (and  correspond to the  distinct  zeroes  of $\wh\alpha_k$,) their limit points $\{z_1,\ldots, z_N\}$ may \underline{not} be distinct.

Therefore, we let: $D_0= div(\wh \a_0),$ be the zero divisor of $\wh \a_0$ and define its support:
 $$Z^{(0)}= supp\, D_0$$ 
 which collects the distinct points in $\{z_1,\ldots, z_N\}$. To be more precise we set, 
$$
    D_0=\text{div}(\wh\a_0)=\sum_{p\in Z^{(0)}} n_p p \,\ \text{with} \sum_{p\in Z^{(0)}} n_p=4(\gg-1)
$$
Hence, 
we can identify the set $Z_0$ of points in $Z^{(0)}$ (possibly empty) which correspond to the limit of "collapsing"  zeroes in $Z^{(k)}$ as given by:
\begin{eqnarray*}
    Z_0:=\{p\in Z^{(0)}: n_p>n_j \text{ with } p=z_j \text{ for some } 1\leq j\leq N \},
\end{eqnarray*}
which is indeed well defined (independently on the choice of the index $j$) since if $z_j=p=z_k$ for $j\neq k\in \{1,\ldots, N\}$, then $n_p\geq n_j+n_k> \max\{n_j,n_k\}$.

We define:
$$
\xi_{k}=- (u_{t_{k}}-s_{t_{k}})
$$
and let,
\begin{equation}\label{6.22}
R_{k} = 8 \Vert \wh{\alpha}_{k} \Vert^{2}
\end{equation}
so that $R_k$ and $|\nabla R_k|$ are uniformly bounded in $X$. Moreover, we have:
\begin{equation}\label{6.23}
-\Delta \xi_{k}
=
R_{k}e^{\xi_{k}}-f_{k}
\; \text{ in  } \; 
X \qquad \text{ and } \,\,
\int_{X}R_{k}e^{\xi_{k}}\, \leq C
\end{equation}
with 
$f_{k}
: =
2 (1-t_{k}e^{u_{t_{k}}})
$ 
satisfying:
$\Vert f_{k} \Vert_{L^{\infty}(X)}\leq 2$ and in view of \eqref{6.20_prime},
\begin{equation}\label{6.24}
\begin{split}
&
f_{k}
\rightarrow
f_{0}
=:
2 (1-\mu e^{w_{0}})
\, \text{ in } \, 
L^{p}(X),\; p>1;
\\
& 
\int_{X}f_{0}=2 \rho([\beta])>0, 
\; \text{ for } \; [\beta]\neq 0.
\end{split} 
\end{equation} 
 
Also notice that,
\begin{equation}\label{form_of_R_k}
R_{k}(z)
=
8 
\prod_{j=1}^{N}(d_{g_{X}}(z,z_{j,k}))^{2 n_{j}}G_{k}(z)
,\;
z \in X,
 \end{equation}
where $d_{g_{X}}$ defines the distance relative to the metric $g_X$. From (\ref{6.22}) we have:
$$G_{k}\in C^{1}(X) \;\; 0 < a \leq G_{k} \leq b	 \mbox{ and  \ }  \vert \nabla G_{k} \vert \leq A \mbox{ \  in   }  X, $$
 with suitable positive constants $a,b$ and $A$.
Hence (by taking a subsequence if necessary) we may assume that, 
\begin{equation}\label{6.27}
G_{k}\rightarrow G_{0}
\; \text{ in   } \; 
C^{0}(X) 
\; \text{ and so } \; 
R_{k}\rightarrow R_{0}
\; \text{ in } \; 
C^{0}(X),
\; \text{ as } \; 
k\rightarrow +\infty,
\end{equation}
with
\begin{equation}\label{3.57a}
R_{0}(z)
=
8 
\prod_{p\in Z^{(0)}}  (d_{g_{X}}(z,p))^{2n_{p}}G_{0}(z)
=
8 \Vert \hat{\alpha}_0 \Vert^{2}.
\end{equation}
 
With the information above, we can apply  Theorem 3 of \cite{Tar_1} 
and deduce the following alternatives
about the asymptotic behaviour of $\xi_{k}$:
\begin{thmx}[Theorem 3 \cite{Tar_1}]\label{thm_blow_up_global_from_part_1} 
Let $\xi_{k}$ satisfy \eqref{6.23} and  assume
\eqref{6.24}-\eqref{3.57a}. Then (with the above notation)
one of the following alternatives holds (along a subsequence):

\

\noindent
(i) \quad (compactness)\; : \; 
$\xi_{k}\longrightarrow \xi_{0}$ in $C^{2}(X)$ with
\begin{equation}\label{thm_referenced_compactness}
-\Delta \xi_{0}
=
R_{0}e^{\xi_{0}}-f_{0}
,\;
\; \text{ in   } \; 
X
\end{equation}

\noindent
(ii)  \quad (blow-up)\; : \;There exists a \underline{finite} blow-up set 
$$
\mathcal{S}
=
\{ x\in X
\; : \; 
\; \exists \; 
x_{k}\rightarrow x
\; \text{ and } \; 
\xi_{k}(x_{k})
\rightarrow 
+ \infty,
\; \text{ as } \; 
k\rightarrow +\infty
 \} 
$$

\quad \,
such that, 
$\xi_{k}$ 
is uniformly bounded from above on compact sets of $X\setminus \mathcal{S}$

\quad \;
and,
as $k\rightarrow +\infty$, 
\begin{enumerate}[label=(\roman*)]
\item[a)] either (blow-up with concentration)\;:\;
\end{enumerate}
$$
\begin{array}{l}
\xi_{k}\longrightarrow -\infty
\; \text{ uniformly on compact sets of  } \;
X\setminus \mathcal{S},
\\
R_{k}e^{\xi_{k}}
\rightharpoonup
\sum_{x\in \mathcal{S}}\sigma(x)\delta_{x}
\; \text{ weakly in the sense of measures, and }
\; 
\sigma(x)\in 8\pi \mathbb{N}.
\end{array}
$$  
\text{In particular,} $ \int_{X} f_{0}\,dA\in 8\pi \mathbb{N},$ 
\begin{eqnarray} \label{3.58a'}
\sigma(x)=8\pi \; \text{ if } \; x\not \in Z^{(0)} 
\; \text{ and } \; 
\sigma(x) = 8\pi (1+n_{i}) \; \text{ if } \; x=  z_{i} \in Z^{(0)} \setminus Z_{0}. \end{eqnarray}

Such an alternative always holds
when $\mathcal{S} \setminus Z_{0} \neq \emptyset$. 

\item[b)] or (blow-up without concentration)\;:\; 
\begin{align}
&
\xi_{k}\rightarrow \xi_{0}
\; \text{ in } \; C^{2}_{loc}(X\setminus \mathcal{S}),
\label{3.58c}
\\
&
R_{k}e^{\xi_{k}}
\rightharpoonup
R_{0}e^{\xi_{0}}
+
\sum_{x\in \mathcal{S}}\sigma(x)\delta_{x}
\; \text{ weakly in the sense of measures, } \;
\notag
\\
&
\sigma(x)\in 8\pi \mathbb{N} \;\;\;\text{ and } \; \mathcal{S} \subset Z_{0}, 
\notag
\end{align}	
with $\xi_{0}$ satisfying:
\begin{equation*}
\quad\quad\quad
-\Delta \xi_{0}
=
R_{0}e^{\xi_{0}}
+
\sum_{x\in \mathcal{S}}\sigma(x)\delta_{x}-f_{0}
\; \text{ in  } \; 
X.
\end{equation*}

\end{thmx}

It is useful to emphasise that if alternative (ii)-b) holds then blow-up occurs at points of "collapsing" zeroes . Actually, in this case the (CMC) immersion of $X$ corresponding to the Cauchy data $(u_t-s_t,\hat{\alpha}_t)$ can be taken to the limit, as $t \to 0^+$ along a sequence, to give a (CMC) immersion of $X$  into a hyperbolic cone-manifold of dimension 3 (characterized by the presence of conical singularities along lines). See (\cite{KS}) for details about hyperbolic cone-manifolds. In particular, at the limit, the induced metric on $X$ admits conical singularities at the blow-up points with conical angles an integral multiple of $8\pi,$ (and not the usual $4\pi$ due to our normalization of the conformal factor) see e.g. \cite{Mazzeo_Zhu_1, Mazzeo_Zhu_2}, \cite{Mondello_Panov_1, Mondello_Panov_2}).

\begin{remark} \label{rho1}
\medskip

1) By virtue of Theorem \ref{thm_blow_up_global_from_part_1} and \eqref{3.1**}, we record that:
$$
    \text{if } [\b]\neq 0, \text{ then } \rho([\b]) \geq 4\pi \text{ and in particular  if } \gg=2 \text{ then } \rho([\b])=4\pi. $$

2) If alternative (i) holds then we can use Theorem \ref{thmprimo} to conclude that $F_{0}$ is bounded from below in $\Lambda$ and
$(u_{t},\eta_{t})\rightarrow (u_{0},\eta_{0})$ in $\Lambda$,
as $t\rightarrow 0^{+}$,
with $(u_{0},\eta_{0})$ the global minimum and only critical point of $F_{0}$, and  
$\rho([\beta])=4\pi(\mathfrak{g}-1),$  see \cite{Tar_2} for details.\\

3) If alternative (ii) holds with $D$ the blow-up divisor of $\xi_k,$ then $ \text{degree} \, D \leq \frac{\rho([\b])}{4\pi} $ and ``blow-up with concentration" (i.e. (ii)-a) 
occurs if and only if 
degree $D=\frac{\rho([\b])}{4\pi}.$
\end{remark}

In the following, we shall investigate the sequence $\xi_k$ in case of blow-up (in the sense of alternative (ii) of Theorem \ref{thm_blow_up_global_from_part_1}) with the purpose to obtain the orthogonality relation (\ref{1.0**}) for the given class $[\b]\in \H^{0,1}(X,E)\setminus\{0\}$.  \\
We let the corresponding blow-up set of $\xi_k$ be given by:  
$$
\mathcal{S} = \{x_1,\ldots,x_m\}.
$$ 
For $r>0$ sufficiently small and $l\in\{1,...,m\}$ we define: 
$$
x_{k,l} \in \overline{B(x_l;r)} : \xi_k(x_{k,l}) = max_{\overline{B}(x_l; r)} \xi_k \to + \infty \ \mbox{and} \ x_{k,l} \to x_l \ \mbox{as $k \to + \infty$}. $$

As already observed in \cite{Tar_2}, and in view of \eqref{antiso}, for every $\alpha \in C_2(X)$ we have:
\begin{lemma} For any $r>0$ sufficiently small we have:
\begin{equation}\label{betawedge} 
\begin{array}{l}
\int_X \beta \wedge \alpha = \int_X \beta_0 \wedge \alpha = \\ e^{\frac{-s_k}{2}}\left( \sum_{l=1}^{m} \int_{B(x_{l}; r)} e^{\xi_k}<*^{-1} \hat{\alpha_k},*^{-1} \alpha> d A  \right) +o(1) \\ =  e^{\frac{-s_k}{2}}\left( \sum_{l=1}^{m} \int_{B(x_{l}; r)} e^{\xi_k}<\alpha, \hat{\alpha_k}> d A  \right) +o(1) \\ \mbox{as \, $k \to + \infty.$ } \end{array} \end{equation}
\end{lemma}
\begin{proof} By formula (3.75) in \cite{Tar_2} and by using  \eqref{antiso} we find:
\begin{eqnarray*}
&& \int_X \beta_0 \wedge \alpha = e^{\frac{-s_k}{2}} \int_X e^{\xi_k} <*^{-1} \hat{\alpha_k}, *^{-1} \alpha> d A =\\
&&\; = e^{\frac{-s_k}{2}} \int_X e^{\xi_k} <\alpha, \hat{\alpha_k}> d A =\\
&&\; = e^{\frac{-s_k}{2}} \left( \sum_{l=1}^{m} \int_{B(x_{l}; r)}  e^{\xi_k} <\alpha, \hat{\alpha_k}> dA \right.\\
&& \; + \left. \sum_{l=1}^{m} \int_{X \setminus \bigcup_{l=1}^{m}  B(x_{l}; r)}  e^{\xi_k} <\alpha, \hat{\alpha_k}> dA     \right).
\end{eqnarray*}
Since
$$
c_{k}=F_{t_{k}}(u_{k},\eta_{k})
=
\frac{1}{4}\int_{X}\vert \nabla w_{k} \vert^{2}dA
-
4\pi(\mathfrak{g}-1)d_{k}+O(1),
$$
we see that, in case of blow up, necessarily:
$d_{k}\rightarrow +\infty \, \, \text{as} \,\,  k\to +\infty. $

Moreover,  $\Vert w_{k} \Vert_{L^{2}(X)}\leq C$ and 
we can use elliptic estimates to derive that the sequence $|w_{k}|$  is uniformly bounded away from the blow-up set $\mathcal{S}$, and therefore, 
\begin{equation}\label{xi_k_versus_d_k_minus_s_k}
\xi_{k}
=-  (d_{k}-s_{k})+O(1)
\; \text{ on compact sets of  } \;
X\setminus \mathcal{S}.
\end{equation}
We can use the last estimate in \eqref{property_v} together with \eqref{xi_k_versus_d_k_minus_s_k} and find a suitable constant $C=C_r > 0 )$ to obtain:
$$  e^{\frac{-s_k}{2}}\left\vert \sum_{l=1}^{m} \int_{X \setminus \bigcup_{l=1}^{m}  B(x_{l}; r)}  \!\!\!\!\!\!\!\!\!\!\!\!\!\!\!\! e^{\xi_k} <\alpha, \hat{\alpha}_k> dA    \right\vert   \leq C_r e^{\frac{-s_k}{2} - (d_k- s_k)} \leq C_r e^{\frac{-d_k}{2}} \to 0 $$ as $k \to + \infty,$
and \eqref{betawedge} is established.
\end{proof}

In particular, from \eqref{xi_k_versus_d_k_minus_s_k} we deduce that:\\
``blow-up with concentration"  
occurs if and only if $d_{k}-s_{k}\longrightarrow +\infty$.\\

So our effort in the following will be to estimate each of the integral terms in \eqref{betawedge}. \\ 

To this purpose we point out that, in local holomorphic z-coordinates, by means of formula \eqref{norme}, we have:  
\begin{equation}\label{localcoordqdiff}  
\begin{array}{l} 
\mbox{if } \hat{\alpha}_k = \hat{a}_k(z)(dz)^2 \,\,\mbox{ and } \alpha = a(z)(dz)^2, \,\,\mbox{ then} \\
\\
<\alpha \,,\, \hat{\alpha}_k>dA = a (\overline{\hat{a_k}}) |dz^2|^2 e^{2u_X} \frac{i}{2} dz \wedge d\bar{z}= 4 a (\overline{\hat{a_k}}) e^{-2u_X} \frac{i}{2} dz \wedge d\bar{z}. \\
 \end{array} \end{equation}  

\vspace{0.3cm}
We start our ``local" analysis around a given blow-up point, say $x_0\in \mathcal{S},$ and for small $r>0$ we let,

$$
  x_k\in X: \xi_k(x_k):=\max\limits_{B(x_0; r)} \xi_k\to +\infty \, \text{and} \, x_k\to x_0, \, \text{as} \, k\to +\infty.
$$

In $B(x_0; r)$ we introduce local "normal" holomorphic $z-$coordinate at $x_0$ centred at the origin  (as specified in \eqref{coord}
 \eqref{hypcoord})  and we write:

\begin{eqnarray}\label{alpha}
\hat{\alpha}_{k} = \hat{a}_{k}(z) (dz)^2 \,\,\,\, \mbox { and } \,\,\, \hat{\alpha}_0 = \hat{a}_{0}(z) (dz)^2
\end{eqnarray}
with $\hat{a}_{k}$ and $\hat{a}$ holomorphic functions in $\Omega_r$, and
\begin{eqnarray}\label{conv_alpha}
 \hat{a}_{k} \to \hat{a}_0 \,\, \mbox{ uniformly in $\Omega_r$ } \,\,\, \text{as} \, k\to +\infty.
 \end{eqnarray}

Moreover, we let: 
\begin {eqnarray}\label{zetakappa} 
z_k \,\, \mbox{ the local expression of $x_k$ in the given z-coordinates (at $x_0$)} \\ 
\mbox{ so that: }\,\,\,  
z_k \to 0 \,\,\mbox { as } k \to \infty.
\end{eqnarray}

As usual, we shall not distinguish between a function and its expression defined in $\Omega_r$ in terms of the $z-$coordinates. 

Therefore, by using a translation and  by replacing: 
\begin{eqnarray}\label{translation}
 \xi_k(z) \to \xi_k(z+z_k) \,\,\, \mbox{ defined in } \Omega_r - z_k ,  
\end{eqnarray} 
for $\delta >0$ sufficiently small: $\bar{B}_{\delta} \subset  (\Omega_r - z_k ),$ we are reduced to analyse the local problem:
\begin{eqnarray}\label{equation_xi}
-\Delta \xi_k = W_ke^{\xi_k} - g_k 
    \, \text{ in $B_\delta,$  } \,
\int_{B_\delta} W_k e^{\xi_k} \frac{i}{2} dz \wedge d\bar z \leq C,
\end{eqnarray} 
where $\Delta
:=4\p_z\p_{\bar z}$ is the flat Laplacian in $\CC$ (or $\R^2$), and we have:  
\begin{eqnarray}\label{equation_W}
W_{k}(z):= R_{k}(z+z_k)e^{2u_X(z+z_k)}\,\,\mbox { and }  \,\, g_{k}(z):= e^{2u_X(z+z_k)}f_k (z + z_k).
\end{eqnarray} 
Thus, in view of \eqref{translation}, there holds:
\begin{eqnarray}\label{blow up seq}
    \xi_k(0)=\max_{B_{\delta}} \xi_k \to +\infty, \quad \text{ as } k\to +\infty,
\end{eqnarray}
and we may let the origin be the only blow-up point of $\xi_k$ in $\bar B_{\delta}$,   namely:
\begin{eqnarray}\label{blow-up seq bounds cpt set}
\forall    K \Subset \bar B_{\delta} \setminus \{0\} ~~ \max_K \xi_k \leq C ~~ \text{ with suitable } C=C(K)>0.
\end{eqnarray} 
By well known potential estimates (see \cite{Li} and \cite{BCLT}) we know also that,
$$
    \max_{\p B_{\delta}}\xi_k-\min_{\p B_{\delta}} \xi_k \leq C \quad \mbox{ \  \ } 
$$ for suitable $C=C(\delta)>0$. 

By the convergence properties in \eqref{6.24} and \eqref{6.27} and by recalling \eqref{alpha1}, \eqref{alpha} and \eqref{conv_alpha}, as $ k\to +\infty,$  we have: 
\begin{eqnarray}\label{convergenceW}
W_k \to W_0 \,\,\mbox{uniformly in } \bar B_{\delta}\,\, \mbox{ with } \,\,\, W_0(z):=32 |\hat{a}_0|^{2}e^{-2u_{X}}, 
\end {eqnarray} 
and for any $p\geq1$ we have:
\begin{eqnarray}\label{g_k}
g_k  \to e^{2u_X} f_0:=g_0,  \quad \text{pointwise and in }  L^p(B_{\delta}). 
\end{eqnarray}

In order to obtain Theorem \ref{thm3} and in view of the results in \cite{Tar_2} concerning blow-up points with minimal mass $8\pi,$ we need to concern ourselves only with the case where,
\begin{equation}\label{blowup mass2}  \sigma_0 := \sigma(x_0):=\lim_{r \to 0}\lim_{k\to +\infty} \int_{B(x_{0}; r)} W_ke^{\xi_k} dA = 16 \pi. \end{equation}

In view of \eqref{3.58a'} we know that, if \eqref{blowup mass2} holds then necessarily: $ x_0 \in Z^{(0)},$ that is $\hat{\alpha_0}(x_0)=0$ ( i.e. $W_{0}(0)=0$).

Hence in \eqref{alpha} we have:

\begin{eqnarray}\label{psi}
\hat{a}_{k}(z) = \prod_{j=1}^s (z-\hat{p}_{j,k})^{n_j}\psi_{k} (z)  \,\,\,\, \mbox { and } \,\,\, \hat{a}_{0}(z) = z^{n}\psi_{0}(z)
\end{eqnarray}
with a suitable integer $s\geq1, \,\, n_j\in\NN :  n=\sum_{j=1}^s n_j,$ and also $ \psi_k$ and  $\psi_0$ holomorphic function never vanishing in $\bar B_{\delta}$.
Furthermore, by taking a subsequence if necessary, we have:
\begin{eqnarray}\label{conv_psi}
\hat{p}_{j,k} \to 0 \,\,\, \mbox{ and }
 \psi_{k} \to \psi_0 \,\, \mbox{ uniformly in $\bar B_{\delta}$ } \,\,\, \text{ as } \,\, k\to +\infty.
 \end{eqnarray}
Therefore, for
\begin{eqnarray}\label{conv_pj}
p_{j,k}:= \hat{p}_{j,k} - z_k \to 0 \,\,\,\text{ as } \,\, k\to +\infty,
\end{eqnarray} 
we find:
\begin{eqnarray}\label{W}
W_{k}(z)= (\prod_{j=1}^s |z-p_{j,k}|^{2n_j})h_k(z) e^{-2u_X(z+z_k)},\,\, h_k(z)=32 |\psi_{k} (z+z_k)|^2
\end{eqnarray} 
in $ \bar B_{\delta}.$
In particular, we have: 
$$
    0<b_1\leq h_k(z) \leq b_2, ~ |\n h_k|\leq A \text{ and } h_k\to h_0 := 32|\psi_0|^2  ~ \text{ uniformly in } \bar B_{\delta},
$$ 
with suitable constants $0<b_1\leq b_2$ and $A>0.$
Without loss of generality, we can assume that,
$$
    0\leq |p_{1,k}|\leq |p_{2,k}|\leq \cdots \leq |p_{s,k}| \to 0, \quad \text{ as } k\to +\infty,
$$
and (to simplify notations) we may set,
 \begin{eqnarray}\label{h_0}
    h_0(0)=1.
\end{eqnarray}

\begin{remark}\label{massquantization_local} In view of the above properties we can apply Proposition 2.1 and Theorem 1 of \cite{Tar_1} to the "local" problem \eqref{equation_xi} and conclude the analogous blow-up alternatives and mass "quantization" property  as  stated in Theorem \ref{thm_blow_up_global_from_part_1} for the "global" problem \eqref{6.23}.
\end{remark}

\subsection{The case $s=1$.}

In this case, we know that the blow-up point $x_0$ of $\xi_k$ is not in $Z_0$, and hence blow-up must occur with the "concentration" property (see part (ii)-a of Theorem \ref{thm_blow_up_global_from_part_1}).
In other words, we have:
$$ \max_{K} \xi_k \to -\infty \ \mbox{for any compact } K \subseteq  X \setminus \mathcal{S}. $$
In turn,  by recalling \eqref{translation} and in view of \eqref{blowup mass2},  for any $\delta >0 $ sufficiently small, for the "local" problem \eqref{equation_xi}, we have:
$$
\int_{\Omega_r - z_k}W_k e^{\xi_k} = \int_{B_{\delta}}W_k e^{\xi_k} + o(1) \,\,\, \mbox { as } k \to + \infty,
$$
\begin{eqnarray}\label{W_k}
    W_k e^{\xi_k} \rightharpoonup 16\pi \delta_0, ~~ \text{ weakly in the sense of measures;}
    \end{eqnarray} 
and in particular,
$$\int_{B_{\delta}}W_k e^{\xi_k} \to 16\pi,\,\,\mbox{ as } k \to + \infty.$$ 

As $s=1,$ we may let $p_k:=p_{1,k},$ and (in account of \eqref{blowup mass2}) we see that  necessarily: $n_1=1=n$  (see \cite{Bartolucci_Tarantello}).

Hence, in the given local holomorphic coordinates, we have:
$$
    \widehat{\a}_k=(z-p_k) \psi_k(z) (dz)^2,  \text{ with } p_k\to 0 \text{ as } k\to +\infty,
$$
and
\begin{eqnarray}\label{a_0}
    \widehat \a_0=z \psi_0(z) (dz)^2. 
\end{eqnarray}

Analogously, for $\a\in C_2(X)$, we set: 
\begin{eqnarray}\label{eta}
    \a= a(z)(dz)^2
\end{eqnarray} 
with $a(z)$  holomorphic in $\bar{B_\delta}$.

If  $\alpha \in Q(x_0)$  (i.e. $\a(x_0)=0$), then by Lemma \ref{approssimazione}, it follows (up to subsequences) that, 
\begin{eqnarray}\label{app_alpha1}
\mbox{ there exist} \,\, \alpha_{k} \in Q(x_k) :\,\,\,  \alpha_{k} \to \alpha, \,\,\mbox{ as } k \to \infty.
\end{eqnarray}
In particular, in the given  local holomorphic z-coordinates we have:
\begin{eqnarray}\label{app_alpha2} 
\alpha_{k} = a_{k}(dz)^2, \,\, a_k(z)=(z - z_{k})C_{k}(z) \,\,\mbox{ and }  a(z)=zC(z), \,\,C(0)=a'(0),
\end{eqnarray} 
with $C_{k}(z)\,\,\mbox {and } C(z)$ holomorphic in  $\bar{B_\delta},$  and $C_k \to C$ uniformly in $\bar {B}_\delta,$ as $k \to \infty.$

There holds:
\begin{lemma}\label{lem1}
Assume $s=1$ in \eqref{W}. For  $\a\in C_2(X)$: $\a(x_0)=0,$ let  $\alpha_{k} \in Q(x_k)$ be such that \eqref{app_alpha1} and \eqref{app_alpha2} hold. Then for $r>0$ sufficiently small we have:
\begin{eqnarray}\label{total asymp behaviour}
    \int_{B(x_0; r)} e^{\xi_k} <\a_k\, \, , \,\widehat \a_k > dA= 64\pi a'(0) \overline{ \psi_0(0)} +o(1),~ \text{ as } k\to +\infty.
\end{eqnarray}
\end{lemma} 

 \begin{proof}
     In order to obtain \eqref{total asymp behaviour}, we need to describe the asymptotic profile for $\xi_k$ around the blow-up point $x_0$. For this purpose, in the given local $z$-coordinate, we need to distinguish between the case
     \begin{eqnarray}\label{simple blowup}
        \ov{ \lim}_{k\to +\infty} e^{\xi_k(0)} |p_k|^4<+\infty,\quad \text{ (simple blow-up) }
     \end{eqnarray}
     or
     $$
             \ov{ \lim}_{k\to +\infty} e^{\xi_k(0)} |p_k|^4=+\infty,\quad \text{ (non-simple blow-up) }
     $$ see Kuo-Lin \cite{Kuo_Lin}, Bartolucci-Tarantello \cite{Bartolucci_Tarantello} and Wei-Zhang \cite{Wei_Zhang_2},\cite{Wei_Zhang_1}, \cite{Wei_Zhang_3}.
    Suppose first that \eqref{simple blowup} holds. Then let $\varepsilon_k:=e^{-\frac{\xi_k(0)}{4}} \to 0$ as $k\to +\infty,$ and consider the  scaled sequence: 
$$
\varphi_k(z):=\xi_k( \varepsilon_k z)-\xi_k(0) 
$$ 
satisfying:

$$
\left\{ \begin{array}{ll}
   -\Delta \varphi_k=|z-\frac{p_k}{\varepsilon_k} |^2 h_k(\ve_k z)  e^{-2u_X(\varepsilon_k z + z_k) } e^{\varphi_k} -\varepsilon_k^2 g_k (\varepsilon_k z) & \text{ in } D_{k,\delta}:= B_{\frac{\delta}{\varepsilon_k}} 
       \\ \varphi_k(0)=0 =\max_{D_{k,\delta}} \varphi_k&
\\ \int_{D_{k,\delta}} \left|z-\frac{p_k}{\varepsilon_k}\right|^2  h_k(\varepsilon_k z) e^{\varphi_k}  e^{-2u_X(\ve_kz + z_k) } \frac{i}{2} dz \wedge d\bar z\to 16\pi & \text{  as } k\to +\infty,
\end{array}\right.
$$
In view of \eqref{simple blowup} we know that, $\left| \frac{p_k} {\varepsilon_k} \right| =e^{\frac{\xi_k (0)}{4}}  |p_k| \leq C,$    
         and so we may suppose that, as $k\to +\infty,$ (along a subsequence) we have: 
         \begin{eqnarray*}
  \frac{p_k}{\varepsilon_k} \to p_0, \quad  
 \varphi_k \to \varphi \text{ uniformly in } C^2_{loc}(\RR^2),
         \end{eqnarray*}
         
         (see e.g.  \cite{Brezis_Merle}) with $\varphi$ satisfying: 
         \begin{eqnarray*}
             -\Delta \varphi =|z-p_0|^2 e^\varphi \text{ in } \RR^2 \text{ and } \int_{\RR^2} |z-p_0|^2 e^\varphi  \frac{i}{2}dz \wedge d\bar z \leq 16\pi.             
         \end{eqnarray*}
         We can apply the classification result of \cite{Prajapat_Tarantello} to $\varphi(z+p_0)$ and by using the information that $\varphi(0)=\max\limits_{\RR^2} \varphi=0$, we arrive at the following expression for $\varphi$,
         \begin{eqnarray*}
             \varphi(z)=\ln \frac{1}{\left(1+\frac 1 {32}|z(z-2p_0)|^2\right)^2} \text{ with } \int_{\R^2} |z-p_0|^2 e^\v \frac{i}{2} dz \wedge d\bar z=16\pi.
         \end{eqnarray*}
         Actually when \eqref{simple blowup} holds, we can apply Theorem 5.6.51 of \cite{Tarantello_Book} and obtain (as in \cite{BCLT})  
         \begin{eqnarray}\label{est of varphi_k with model}
             \left| \varphi_k(z) -\ln \frac 1 { \left( 1+\frac 1 {32} |z(z-2p_0)|^2\right)^2} \right| \leq C, \text{ for } \,\, z\in D_{k,\delta}.
     \end{eqnarray}
             As a consequence, by using \eqref{localcoordqdiff}, we have:
            \begin{eqnarray*}
               &&  \int_{B(x_0; r)} e^{\xi_k} <\a_k\, \, , \,\widehat \a_k > dA=
               \\&&=4 \int_{B_\delta} e^{\xi_k} a_{k}(z + z_k) (\bar{z}- \overline{p_k}) \overline{\psi}_k(z+z_k) e^{-2u_X(z+z_k)} \frac{i}{2} dz \wedge d\bar z +o(1)
                \\&& =\;4\int_{B_\delta} e^{\xi_k} z C_{k}(z +z_k)  (\bar{z}-\overline{p_k}) \overline{\psi}_k(z + z_k) e^{-2u_X(z + z_k)} \frac{i}{2} dz \wedge d\bar z + o(1)
                 \\&&=\;4\int_{D_{k,\delta}} e^{\varphi_k(z)} z C_{k}(\varepsilon_k z + z_k)  \left(\bar{z}-\frac{\overline{p_k}}{\varepsilon_k} \right) \overline{\psi_k}(\varepsilon_k z + z_k)  e^{-2u_X(\varepsilon_k z + z_k)} \frac{i}{2} dz \wedge 
                d\bar z + o(1).
             \end{eqnarray*}
             
             In view of \eqref{est of varphi_k with model}, we can use dominated convergence to pass to the limit into the integral sign and obtain:
             {\small \begin{eqnarray*}
           \int_{B(x_0; r)} e^{\xi_k} <\a_k\, \, , \,\widehat \a_k > dA =\left(4 \int_{\RR^2} \frac{z \overline{(z-p_0)} }{\left( 1+\frac 1 {32} |z(z-2p_0)|^2\right)^2}\frac{i}{2} dz \wedge d \bar z\right) a'(0) \overline{\psi_0(0)}  +o(1), 
             \end{eqnarray*}} as $ k\to +\infty$.
           At this point, we can use the change of variable $z=z-p_0$ to find:
             \begin{eqnarray*}
            &&    \int_{\RR^2} \frac{z (\bar{z}-\overline{p_0})} {\left( 1+\frac 1 {32} |z(z-2p_0)|^2\right)^2} \frac{i}{2}dz \wedge d \bar z=\int_{\RR^2} \frac{\bar{z}( z+ p_0) }{\left( 1+\frac 1 {32} |z^2-p_0^2|^2\right)^2} \frac{i}{2} dz \wedge d \bar z  
               \\&=&\int_{\RR^2} \frac{|z|^2}{\left( 1+\frac 1 {32} |z^2-p_0^2|^2\right)^2} \frac{i}{2} dz \wedge d \bar z + p_0 \int_{\RR^2} \frac{\bar{z} }{\left( 1+\frac 1 {32} |z^2-p_0^2|^2\right)^2} \frac{i}{2}dz \wedge d \bar z\\&=&\int_{\R^2}  \frac{|z|^2}{\left(1+\frac 1 {32} |z^2-p_0|^2\right)^2} \frac{i}{2} dz \wedge d \bar{z}, 
             \end{eqnarray*} 
           since by symmetry, the latter integral in the line above vanishes (being the function into the integral sign is odd). Moreover, from \cite{Prajapat_Tarantello} we know that,  
             \begin{eqnarray*}
                     \int_{\R^2}  \frac{|z|^2}{\left(1+\frac 1 {32} |z^2-p_0|^2\right)^2} \frac{i}{2} dz \wedge d \bar{z}  =16\pi,
             \end{eqnarray*} 
             and \eqref{total asymp behaviour} is established in this case.\\
Next assume that, along a subsequence, there holds
             \begin{eqnarray*}
                 \xi_k(0)+ 4\log |p_k| \to +\infty, \text{ as } k\to +\infty,
             \end{eqnarray*}and let,
             \begin{eqnarray}\label{seq pt limit}
                 \frac{p_k}{|p_k|} \to p_0, \text{  as  } k\to +\infty, \text{ with } |p_0|=1.
             \end{eqnarray}
In this situation, we consider the scaled sequence:
             \begin{eqnarray*}
                 \wt{\v}_k(z):=\xi_k(p_k+|p_k| z)+4 \log |p_k|
             \end{eqnarray*} and, by recalling \eqref{conv_pj}, we see that it satisfies:
             \begin{eqnarray*}
\left\{
\begin{array}{l}
{\small -\Delta \wt\v_k(z)=|z|^2 e^{\wt \v_k(z)} h_k(p_k+|p_k|z) e^{-2u_X(\hat{p}_k+|p_k|z)}-|p_k| g_k(p_k+|p_k|z) \; \text{ in } \wt{D_{k,\d}},} \\
\int_{\wt{D_k}} |z|^2 e^{\wt\v_k(z)} h_k(p_k+|p_k|z)e^{-2u_X(\hat{p}_k+|p_k|z)} \frac{i}{2} dz \wedge d\bar z\to 16\pi \qquad \text{ as } k \to+\infty,\\ 
\wt\v_k \left(-\frac{p_k}{|p_k|}\right) \to +\infty 
\qquad\text{ as } k \to+\infty,
\end{array} \right.
\end{eqnarray*}
where $\wt{D_{k,\d}}:=\{z: p_k+|p_k|z \in B_{\d} \},$ for $\d\in(0,r)$.\\
Hence, we can apply the ``non-simple" blow-up analysis of \cite{Kuo_Lin} and \cite{Bar_Tar_JDE} to $\wt\v_k$, and in the account of \eqref{seq pt limit}, we conclude that,
\begin{eqnarray*}
    e^{\wt\v_k(z)} \to 8\pi \d_{p_0}+8\pi\d_{-p_0} \text{ as } k\to +\infty, \text{   weakly in the sense of measure (locally in  } \RR^2).
\end{eqnarray*}Consequently, for $R>>1$ large, we derive that: 
\begin{eqnarray*}
    \int_{\wt{D_{k,\d}} \setminus \{ |z| \leq R \}} |z|^2 e^{\wt\v_k} \frac{i}{2}dz \wedge d\bar z \to 0, \text{ as } k\to +\infty.
\end{eqnarray*}
Thus, still recalling \eqref{conv_pj},  we have:
\begin{eqnarray*}
&& \int_{B(x_0; r)} e^{\xi_k} <\a_k\, \, , \,\widehat \a_k > dA=4\int_{B_\delta} e^{\xi_k} (\bar{z}-\bar{p}_k) \overline{ \psi_k}(z + z_k)zC_{k}(z + z_k) e^{-2u_X(z + z_k)} \frac{i}{2}dz \wedge d\bar z +o(1)
\\&=& 4 \int_{\wt{D_{k,\d}}} e^{\wt\v_k(z)} \bar{z}\left( z+\frac{ p_k}{|p_k|}  \right)  C_{k}( \hat{p}_k+|p_k|z) \overline{\psi_k}(\hat{p}_k+|p_k|z) e^{-2u_X(\hat{p}_k+|p_k|z) }\frac{i}{2} d z \wedge d\bar z + o(1)
\\&=& {\small 4\int_{\{ |z|\leq  R\} } \!\!\!\!\!\!\!\!\!\!\!\!e^{\wt\v_k(z)} \bar{z}\left( z+\frac{ p_k}{|p_k|}  \right)  C(\hat{p}_k+|p_k|z) \overline{\psi_k}(\hat{p}_k+|p_k|z) e^{-2u_X(\hat{p}_k+|p_k|z) } \frac{i}{2} d z \wedge d\bar z+o(1)}
\\&=& 4 (\overline{p_0} (2 p_0) 8\pi) a'(0) \overline{\psi_0}(0)+o(1), \text{ as } k\to +\infty.
    \end{eqnarray*}
    Since $|p_0|=1$, also in this case we conclude that, 
    \begin{eqnarray*}
       \int_{B(x_0; r)} e^{\xi_k} <\a_k\, \, , \,\widehat \a_k > dA=64\pi a'(0) \overline{\psi_0}(0) +o(1), \text{ as } k\to +\infty
    \end{eqnarray*}as claimed.
    
 \end{proof}

\subsection{The case $s\geq 2$.}
In this case, in local $z$-coordinates, $\xi_k$ satisfies:

\begin{eqnarray}\label{2.1}
\left\{
\begin{array}{llll}
    -\Delta \xi_k= \prod_{j=1}^s |z-p_{j,k}|^{2n_j} h_k(z) e^{-2u_X(z + z_k)} e^{\xi_k} -g_k(z)  \text{ in } B_\delta,\\
\\
\int_{B_\delta}  \prod_{j=1}^s |z-p_{j,k}|^{2n_j} h_k(z) e^{-2u_X(z + z_k)} e^{\xi_k} \frac{i}{2} dz \wedge d\bar z \leq C,
\end{array} \right. \end{eqnarray}
where all properties stated above hold. 
Notice now that the ``concentration" property for $\xi_k$ is no longer ensured. But yet, for $\delta \in (0, r)$ sufficiently small there holds:
\begin{eqnarray*}\label{2.1*}
\int_{(\Omega_{r} - z_k)}  \prod_{j=1}^s |z-p_{j,k}|^{2n_j} h_k(z) e^{-2u_X(z + z_k)} e^{\xi_k} \frac{i}{2} dz \wedge d\bar z \\ =  \int_{B_\delta}  \prod_{j=1}^s |z-p_{j,k}|^{2n_j} h_k(z) e^{-2u_X(z + z_k)} e^{\xi_k} \frac{i}{2} dz \wedge d\bar z  +  o_{r}(1)
\end{eqnarray*}
where:
$$ o_{r}(1) \to 0 \,\, \mbox{ as} \,\, r\to 0^+, \,\, \mbox{ uniformly in $k.$} $$ 

Let,
$$
    n:=\sum_{j=1}^s n_j \geq 2 ~~\text{ and } \tau_k:=|p_{s,k}|\to 0^+         ~~\text{ as } k\to +\infty,
$$
and consider the scaled sequence:
$$
    \v_k(z):=\xi_k(\tau_k z) +2(n+1) \ln \tau_k.
$$
By setting,
$$
    q_{j,k}:=\frac{p_{j,k}}{\tau_k},\quad j=1,\cdots, s;
$$ 
we have that, 
\begin{eqnarray}\label{seq convengence}
    0\leq |q_{1,k}|\leq |q_{2,k}|\leq \cdots \leq |q_{s,k}| =1,\text{ with } q_{j,k}\to q_j, \text{ as } k\to +\infty.
\end{eqnarray}for suitable $q_j,j=1,\ldots, s;$ and  the convergence in \eqref{seq convengence} is attained possibly along a subsequence. Thus for,
$$
 h_{1,k}(z):=h_k(\tau_k z) e^{-2u_X(\tau_k z + z_k)}  ~~\text{  and }  g_{1,k}(z):=\tau_k^2  g_k( \tau_k z),
$$
we  have:
\begin{eqnarray}\label{eq of varphi-k}
\left\{
\begin{array}{llll}
-\Delta \v_k =\prod\limits_{j=1}^s |z-q_{j,k}|^{2n_j} h_{1,k}(z) e^{\v_k} -g_{1,k}(z)  & \text{ in } D_{k, \delta}:=B_{\frac{\delta}{\tau_k}}   \\
 \int_{D_{k,\delta}} \prod\limits_{j=1}^s |z-q_{j,k}|^{2n_j}  h_{1,k}( z)e^{\v_k} \frac{i}{2} dz \wedge d\bar z\leq C    & 
\end{array} \right.
\end{eqnarray}
with
$$
 h_{1,k}(z) \to h_0(0)=1\text{ and } g_{1,k} \to 0, \text{ uniformly in } C_{loc}(\RR^2), \text{ as } k \to +\infty.
$$

Let 
\begin{eqnarray}\label{lambda-varphi}
 \lambda_{\varphi}:=\lim_{R\to +\infty} \lim_{k\to +\infty} \int_{B_R} \prod_{j=1}^s |z-q_{j,k}|^{2n_j} h_k(\tau_k z)e^{\v_k} \frac{i}{2} dz \wedge d\bar z \leq \s_0=16\pi
\end{eqnarray}
from \cite{Lee_Lin_Wei_Yang} we know that the following relation holds: 
\begin{eqnarray}\label{s0 identity}
 \s_0^2 -\lambda_\v^2 =8\pi (n+1) (\s_0-\lambda_\v), 
\end{eqnarray}
see also the Appendix in \cite{Tar_1} for a detailed proof.

Hence, from  \eqref{s0 identity} we obtain that,
\begin{eqnarray*}
  &&  \text{either } \lambda_\v=\s_0 \text{ or } \lambda_\v=8\pi (n+1)- \s_0 \in 8\pi \NN\cup \{0\} \\&& \text{ and (since } \l_\v<\s_0) \text{ \  we have: \ } 0\leq \l_\v<4\pi (n+1).
\end{eqnarray*}

In particular, in our case where \eqref{blowup mass2} holds, we find that necessarily,
\begin{eqnarray} \label{two cases} 
&& \text{either }  \lambda_\v=\s_0=16\pi \text{ or } \s_0=16\pi ,\lambda_\v=8\pi  \text{ and } n=s=2 \\ \nonumber && \text{ with }  n_1=n_2=1.
 \end{eqnarray}
It is well known, that (possibly along a subsequence)  $\v_k$ must satisfy one of the following three alternatives (see \cite{Brezis_Merle}): 
\begin{enumerate}
\item[(i)] $\v_k\to -\infty$ uniformly on compact sets.

\item[(ii)] $\v_k\to \v$ uniformly on compact sets with $\v$ satisfying: 

\begin{eqnarray}\label{eq of varphi-limit}
\left\{
\begin{array}{llll}
-\Delta \v =\prod\limits_{j=1}^s |z-q_{j}|^{2n_j}  e^{\v}    ~~~\text{ in } \RR^2 \\
 \int_{\RR^2 } \prod\limits_{j=1}^s |z-q_{j}|^{2n_j}   e^{\v} \frac{i}{2} dz \wedge d\bar z=\l_\v   
\end{array} \right.
\end{eqnarray}

\item[(iii)] $\v_k(0)=\max\limits_{D_{k,\delta}} \v_k \to +\infty.$\\  Since: $ \l_\v \in \{8\pi,16\pi\}$ then the blow-up set $\mathcal{S}_\v$ of $ \v_k$ contains the origin and at most one other point.
\end{enumerate}

We can readily rule out the alternative (i), as it would imply $\l_\v=0$, in contradiction with \eqref{two cases}. 

Next we analyse the case where 
alternative (ii) holds. To this purpose we  recall some relevant properties about solutions of \eqref{eq of varphi-limit}. We know that they play a relevant role in the search of spherical metrics, namely metrics of constant curvature on the Riemann sphere with punctures where we have conical singularities with given angles, see e.g. \cite{Mondello_Panov_1}, \cite{Mondello_Panov_2}, \cite{Mazzeo_Zhu_1} and \cite{Mazzeo_Zhu_2} and references therein.

By Theorem  2 in \cite{Chen_Li_2}, we know that, 
\begin{eqnarray*}
\v(z)=-\frac{\l_\v}{2\pi} \ln |z|+O(1), \text{ as }  |z|\to +\infty,
\end{eqnarray*} and since $\lambda_{\varphi} \in \{8\pi,16\pi\}$ then by  the integrability condition in  \eqref{eq of varphi-limit} we conclude  that necessarily:
\begin{eqnarray}\label{alternative ii}
 \l_\v=16\pi, \  s=n=2 \text{ with } n_1=n_2=1 \text{ and } q_1\neq q_2.
\end{eqnarray}

The latter condition, $q_1\neq q_2$ is ensured by \cite{Prajapat_Tarantello}

\begin{lemma}\label{lem2}
 In case $\varphi_k$ satisfies alternative (ii), then 
 \eqref{alternative ii} holds and for any $\a\in C_2(X)$ with local expression  
\eqref {eta} and any $r>0 $ sufficiently small, we have:
 \begin{eqnarray}\label{alternaive ii asy}
  && \int_{B(x_0; r)} e^{\xi_k} \<\a\, , \wh{\a}_k \> dA = \\
  && \;=\;\frac{1}{\left( W_k\left(\frac{p_{1,k}+p_{2,k}}{2}\right)\right)^{\frac{1}{2}}} \left(-32\pi \frac{(\bar{q}_2- \bar{q}_1)^2}{|q_2-q_1|^2}a(0) \overline{\psi_0}(0) +o(1)\right) + o_{r}(1)  \nonumber
 \end{eqnarray}
as $k\to +\infty.$
\end{lemma}

\begin{proof}
  In view of \eqref{alternative ii}, we compute:
\begin{eqnarray}\label{dominated convergence 2}
&& \int_{B_\d(x_k)} e^{\xi_k} \<\a\,, \wh{\a}_k\> dA =\\ 
&& \;=\;  \notag 4 \int_{B_\d } e^{\xi_k(z)}  (\bar{z}-\bar{p}_{1,k})(\bar{z}-\bar{p}_{2,k}) \overline{\psi_k}(z + z_k)a(z + z_k) e^{-2u_X(z + z_k)} \frac{i}{2} dz \wedge d \bar z  + o_{r}(1) \nonumber\\
&& \;=\; \frac {4}{\tau_k^2 }\int_{D_k} e^{\v_k} (\bar{z}-\bar{q}_{1,k})(\bar{z}-\bar{q}_{2,k}) \overline{\psi}_{k}(\tau_k z + z_k) a(\tau_k z + z_k) e^{-2u_X(\tau_k z + z_k)}  \frac{i}{2}  dz \wedge d \bar z  + o_{r}(1). \nonumber
\end{eqnarray} 
 Moreover, by well-known potential estimates,  (as in Lemma 5.6.52 of \cite{Tarantello_Book}) we  know that,
 \begin{eqnarray*}
  \forall \varepsilon>0, \exists \, C_\ve>0 \text{ and } k_\ve\in \NN: ~~ \varphi_k(z)\leq -\left( \frac{\l_\v}{2\pi}-\ve\right) \ln |z|+C_\ve, 
\end{eqnarray*} 
$\text{ for }  |z|\geq 1\,\text{ and } \, k\geq k_\ve,$
by which we can use dominated convergence in \eqref{dominated convergence 2} and derive, as $k\to +\infty$ 
\begin{eqnarray}\label{R^2}
    && \int_{B_\d } e^{\xi_k} \<\a\, , \wh{\a}_k\> dA =
    \\
    && 
    \;=\;\frac {4}{\tau_k^2 } \left[ \left( \int_{\RR^2} e^\v (\bar{z}-\bar{q}_1)(\bar{z}-\bar{q}_2) \frac{i}{2} dz \wedge d\bar z \right) \overline{\psi}_0(0)a(0)  +o(1) \right] + o_{r}(1).\nonumber
 \end{eqnarray}
 Thus, to conclude \eqref{alternaive ii asy} we prove: 
 
\underline{\textbf{Claim}:} 
 \begin{eqnarray}\label{claim R2}
  \int_{\RR^2} e^\v (z-q_1)(z-q_2)\frac{i}{2} dz \wedge d\bar z =-32\pi \frac{(q_2-q_1)^2}{|q_2-q_1|^4} .
 \end{eqnarray}
To establish the claim, we let $q=(q_1-q_2)/2 \neq 0,$ ( recall \eqref{alternative ii})
, and by using  complex variable, we define:
\begin{eqnarray*}
 \psi(z):=\v\left(qz+\frac{q_1+q_2} 2\right) +6 \ln|q|,
\end{eqnarray*} satisfying:
\begin{equation}\label{R^2-1} 
\left\{
\begin{array}{llll}
-\Delta \psi =|z^2-1|^2 e^\psi ~~~   \text{ in } \RR^2 \\
 \int_{\RR^2 } |z^2-1|^2 e^\psi \frac{i}{2} dz \wedge d\bar z=16\pi.   
\end{array} \right.
\end{equation}
The solutions of (\ref{R^2-1}) have been completely characterized in Appendix 1, by which we can deduce that,

\begin{eqnarray*}
    \int_{\R^2} e^{\psi(z)}(z^2-1) \frac{i}{2}dz \wedge d\bar z=-8\pi.
\end{eqnarray*}
see Lemma \ref{lemmaA.2} in Appendix 1 for details.
Consequently,
\begin{eqnarray*}
        \int_{\R^2} e^{\v(z) }(z-q_1)(z-q_2) \frac{i}{2}dz \wedge d\bar z=\frac{q^2}{|q|^4} \int_{\R^2} e^{\psi} (z^2-1)\frac{i}{2}dz \wedge d\bar z=-8\pi \frac{q^2}{|q|^4},
\end{eqnarray*} 
and by recalling that $q=\frac{q_2-q_1}{2},$ we see that \eqref{claim R2} is established. On the other hand, we find:  $$ \left( W_k\left(\frac{p_{1,k}+p_{2,k}}{2} \right) \right)^{\frac{1}{2}} = \frac{\tau_k^2}{4} | q_{1,k}-q_{2,k} |^2 (1 +o(1))= \frac{\tau_k^2}{4} | q_1 - q_2 |^2 (1 +o(1)),\,\, k\to +\infty  $$ and from 
(\ref{R^2}), (\ref{claim R2}) we readily get (\ref{alternaive ii asy}). 
\end{proof}
\vspace{0.6cm}
Next assume that $\v_k$ satisfies alternative (iii), namely:
\begin{equation}\label{alterniii}
    \v_k(0)=\max_{D_{k,r}} \v_k\to +\infty, \text{ as } k\to +\infty,
\end{equation}
so $0\in \mathcal{S}_\v$ (the blow-up set of $\v_k$), and by Theorem 1 of \cite{Tar_1} (see also Remark \ref{massquantization_local}), the blow-up mass at the origin satisfies the \underline{quantization} property: $\s_\v(0) \in 8\pi\mathbb N.$ Since, $\s_\v(0) \leq \s_0 = 16\pi$, we deduce that,
\begin{eqnarray}\label{8pim}
    \s_\v(0)=8\pi m \text{ with } m\in\{1,2\}.
\end{eqnarray}
\begin{proposition}\label{prop1}
If $\v_k$ blows-up (i.e. (\ref{alterniii}) holds) then:
\begin{enumerate}
\item[(i)]   blow-up occurs with the "concentration" property, namely: 
\begin{equation}\label{concentration} \forall K \Subset \mathbb{R}^2 \setminus \mathcal{S}_{\varphi} \quad  \ max_K \varphi_k \to - \infty \mbox{ \ as } k \to + \infty. \end{equation}

\item[(ii)] If $\lambda_{\varphi} = 16 \pi$  then for $R_0>0$ sufficiently large, there holds: 
\begin{equation}\label{suplog}
sup_{D_{k,\d} \setminus \{ |z|\geq R_0\}} \left(\varphi_k(z) +  2(n+1)\log|z| \right) \leq C_0, \end{equation}
for suitable $C_0 >0.$
\end{enumerate}
\end{proposition}
\begin{proof}
In case we have $\sigma_{\varphi}(0) = 16 \pi$ (recall (\ref{8pim})), then necessarily: $\lambda_\varphi = 16 \pi = \sigma_\varphi(0)$ and so $\mathcal{S}_\varphi = \{0\}.$ Furthermore for given $0 < \epsilon < R$ there holds: 
$$ \int_{\epsilon < |z| < R}  \prod\limits_{j=1}^s |z-q_{j,k}|^{2 n_j} h_{1,k}(z)e^{\varphi_k} \frac{i}{2}dz \wedge  d \bar{z} \to 0  $$
as $k \to + \infty.$ 
Consequently, by a well known Harnack type inequality (see \cite {Brezis_Merle}, \cite{Li_Shafrir}) it follows that, $ max_{2\epsilon \leq  |z| \leq R/2} \, \varphi_k \to -\infty $ as $k \to + \infty$ and (\ref{concentration}) is established in this case. 

Next assume that $\sigma_\varphi(0) = 8 \pi.$ In case $\lambda_\varphi = 8 \pi$, then the "concentration" property would follow exactly as above. Hence assume that $\lambda_\varphi = 16 \pi,$ and argue by contradiction, by assuming that blow-up for the sequence $\v_k$ occurs without the  "concentration" property. Since, $\lambda_\varphi = 16 \pi,$ we see that $\v_k$ cannot admit any other blow-up point beside the origin, i.e. $\mathcal{S}_\v=\{0\}$  and $\s_\v(0)=8\pi.$ Furthermore, (along a subsequence) 
    \begin{eqnarray*}
        \v_k\to \v \text{ uniformly in } C^2_{loc}(\RR^2\setminus\{0\}), 
    \end{eqnarray*}with $\v$ satisfying:
    
\begin{eqnarray*} 
\left\{
\begin{array}{llll}
-\Delta \v =\prod\limits_{j=1}^s |z-q_{j}|^{2n_j}  e^{\v}  +8\pi\d_0 ~~~~ \text{ in } \RR^2 \\
 \int_{\RR^2 } \prod\limits_{j=1}^s |z-q_{j}|^{2n_j}   e^{\v} \frac{i}{2}dz \wedge d\bar z=8\pi.
\end{array} \right.
\end{eqnarray*}
In addition, in this situation the origin must be of "collapsing" type (as otherwise the "concentration" property would hold), and so the origin is the limit point of different $q_{j,k},\,\, j=1,\ldots, s$ as $k \to + \infty.$ But $|q_{s,k}|=1$, and so we find that $s\geq 3$ (and so $n\geq 3$) and there exists $s_1\in \{2,\ldots, s-1\}$ such that $q_1=\cdots=q_{s_1}=0$ while $q_j\neq 0$ for $j=s_1+1,\ldots, s$. So by letting $\bar n:=\sum\limits_{j=1}^{s_1} n_j\geq 2$, we see that the function: 
\begin{eqnarray*}
    \psi(z):=\v(z)+4\ln |z|, 
\end{eqnarray*}extends smoothly at the origin, and it satisfies:
\begin{eqnarray}\label{eqs of psi}
\left\{
\begin{array}{llll}
-\Delta \psi =|z|^{2(\bar n-2)} \left( \prod\limits_{j=s_1+1}^s |z-q_{j}|^{2n_j}  e^{\psi}     \right) ~~~  \text{ in } \RR^2 \\
 \int_{\RR^2 } |z|^{2(\bar n-2)} \left(\prod\limits_{j=s_1+1}^s |z-q_{j}|^{2n_j}   \right) e^\psi \frac{i}{2}dz \wedge d\bar z=8\pi.
\end{array} \right.
\end{eqnarray}
Again, by using the Theorem 2 of \cite{Chen_Li_2}, we obtain that
\begin{eqnarray*}
    \psi(z)=-4\ln |z|+O(1), \text{ as } |z|\to +\infty,
\end{eqnarray*} that is,
\begin{eqnarray*}
    |z|^{2(\bar n-2)}\prod\limits_{j=s_1+1}^s|z-q_j|^{2n_j} e^{\psi(z)} =O( |z|^{2(n-4)}) , \text{ as } |z|\to +\infty,
\end{eqnarray*}with $n\geq 3.$ This is impossible since it violates the integrability property in  \eqref{eqs of psi}. So \eqref{concentration} is established in all cases.
To establish (\ref{suplog}) notice that in the given assumption we have: 
$\lambda_\varphi= \sigma_0 = 16 \pi.$
Consequently for every $\ve>0$, there exists $k_\ve\in \NN$, $R_\ve>1$ and $\d_\ve\in(0,r)$ such that, for every $R\geq R_\ve$, $\d\in (0, \d_\ve)$ and $k\geq k_\ve$ there holds:
\begin{eqnarray}\label{tar_11}
    \int_{D_{k,\d} \setminus \{ |z|\geq R\}}  e^{\v_k} \left(\prod_{j=1}^s |z-q_{j,k}|^{2n_j} \right) h_{1,k}(z) e^{\v_k}\frac{i}{2} dz \wedge d\bar z <\ve.
\end{eqnarray}.

It is a well-known fact, (see e.g. Lemma 3.1 of \cite{Tar_1}) that from \eqref{tar_11}, we readily get (\ref{suplog}).  \end{proof}

\begin{proposition}\label{prop2}
Suppose that $\v_k$ blows-up (i.e. (\ref{alterniii}) holds)  with  $\s_\v(0)=8\pi$, 
\begin{enumerate}
    \item[(i)] If $\l_\v=16\pi$, then $\v_k$ admits exactly one other blow-up point $y_0\neq 0,$ i.e. $\mathcal{S}_\v=\{0,y_0\}$, with corresponding blow-up mass $\s_\v(y_0)=8\pi =\s_\v(0)$. 

Moreover (\ref{suplog}) holds with suitable $\d >0$ sufficiently small and $R>1$ sufficiently large.
\vspace{0.6cm}
    \item[(ii)] If $\l_\v=8\pi$ (i.e. the second alternative in  \eqref{two cases} holds), then $s=n=2$, $n_1=n_2=1$ \mbox{ \ with } $q_1\neq 0$ and  $q_2\neq 0$ . Furthermore  $\mathcal{S}_\v=\{0\},$ and  either \eqref{suplog} holds  or  there exists a sequence $\ve_k\to 0^+$ such that
    $\frac{\tau_k} {\ve_k} \to 0$  as $k\to+\infty,$ and the properties of part (i) hold for the sequence 
    \begin{eqnarray}\label{sequence of tilde varphi}
        \wt\v_k(z)=\xi_k(\ve_kz)+2(n+1) \ln \ve_k.
    \end{eqnarray}
\end{enumerate}
\end{proposition}

\begin{proof}

To establish (i), we need to show that $\v_k$ admits exactly another blow-up point beside the origin. To this purpose, we argue by contradiction 
and  
assume that $\mathcal{S}_\v=\{0\}$. Then, for $\d>0$ sufficiently small and for 
any $R>>1$ large, we would have:
\begin{eqnarray*}
   && \int_{B_R} \left( \prod_{j=1}^s |z-q_{j,k}|^{2n_j} \right) h_{1,k}(z) e^{\v_k} \frac{i}{2} dz\wedge d\bar z \\&= & \int_{B_\d} \left(\prod_{j=1}^s |z-q_{j,k}|^{2n_j} \right) h_{1,k}(z) e^{\v_k} \frac{i}{2} dz\wedge d\bar z+o(1)
    \\&=&  \s_\v(0)+o(1)=8\pi+o(1) ~~~\text{ as } k\to +\infty \text{ and as } \d\to0^+.
\end{eqnarray*}
As a consequence, (by recalling \eqref{lambda-varphi}) we would have that $\l_\v=8\pi,$ in contradiction with the assumption that $\l_\v=16\pi$. So $\mathcal{S}_\v$ must contain exactly one more blow-up point $y_0\neq 0$, with blow up mass $\s_\v(y_0)=8\pi$ as claimed in (i). 

In order to establish (ii), recall that by assumption:  $\l_\v=\s_\v(0)=8\pi,$ so, by \eqref{two cases}, we have: $s=2=n$ and $n_1=n_2=1,$ and so $|q_2|=1.$ Therefore, for any given $0<\d<R$, there holds:
\begin{eqnarray*}
    \forall \ve>0 \, \exists \, k_\ve\in \NN: \,\int_{B_R\setminus B_\d} \left( \prod_{j=1}^s |z-q_{j,k}|^{2n_j} \right) h_{1,k} (z) e^{\v_k} \frac{i}{2} dz \wedge  d\bar z <\ve, \, \forall k\leq k_\ve  
\end{eqnarray*} and we conclude that, $\mathcal{S}_\v=\{0\}$. 
To show that also $q_1\neq 0$ we argue by contradiction. Indeed, if $q_1=0$ then the origin would be a blow-up point for $\v_k$ corresponding to the limit of an isolated zero of the weight function in equation \eqref{eq of varphi-k}, (a situation analogous to the case $s=1$ discussed above). But from \cite{Bartolucci_Tarantello}, this fact would imply that: $\s_\v(0)=8\pi(n_1+1)=16\pi,$ in contradiction to our assumption. Hence necessarily: $q_1\neq 0$ as claimed.

Next, assume that \eqref{suplog} fails for any  $R>1$ large. By taking into account that 
$S_\varphi = \{0\},$ then we find $y_k \in D_{k,\d}$ satisfying:

\begin{eqnarray*}
    |y_k|\to +\infty,   \  \xi_k(\tau_k y_k) +2(n+1) \ln(\tau_k |y_k| ) \to +\infty, \text{ as } k\to +\infty,
\end{eqnarray*} and in addition,
\begin{eqnarray*}
    \ve_k:=\tau_k|y_k| \to 0 \,\,  \text{ and } \frac{\tau_k}{\ve_k} =\frac 1{|y_k|} \to 0, \text{ as } k\to +\infty.
\end{eqnarray*}
Furthermore, 
$$
    \wt\v_k(z)= \xi_k(\ve_k z) +2(n+1) \ln \ve_k =\v_k(|y_k| z) +2(n+1) \ln |y_k|, 
$$
satisfies: 
$$
\begin{array}{l}
 -\De \wt\v_k= \left( \prod_{j=1}^s \left| z-\frac{q_{j,k}}{|y_k|} \right|^{2n_j}\right) h_k(\ve_k z) e^{-2u_X(\ve_k z)} e^{\wt\v_k}-\ve_k g_k(\ve_k z)   \text{ in } \wt D_{k,\delta} := B_{\frac{\delta}{\ve_k }}\\
\int_{\wt D_{k,\delta}} \prod\limits_{j=1}^s \left( \left| 1-\frac{q_{j,k}}{|y_k|} \right|^2 \right) h_k(\ve_k z) e^{-2u_X(\ve_k z)} e^{\wt\v_k} \frac{i}{2}dz \wedge d\bar z < C.  \\ 
\end{array} $$
We let, 
\begin{eqnarray*}
    \l_{\wt \v} &:=&\lim_{R\to +\infty} \lim_{k\to +\infty} \int_{B_R} \prod\limits_{j=1}^s \left( \left| 1-\frac{q_{j,k}}{|y_k|} \right|^2 \right) h_k(\ve_k z) e^{-2u_X(\ve_k z)} e^{\wt\v_k} \frac{i}{2}dz \wedge d\bar z \\
    &\leq& \s_0=16\pi. 
\end{eqnarray*}
Moreover,
$$\wt\v_k(0) =\v_k(0)+ 2(n+1) \ln |y_k| \to +\infty, \text{ as } k\to +\infty.$$
and so the origin is a blow-up point for $\tilde{\varphi_k}$ of ``collapsing" zeroes, as we have:
$ \wt{q}_{j,k}:=  \frac{q_{j,k}}{|y_k|} \to 0,\text{ as } k\to +\infty, ~~ \forall j=1,\cdots, s.$
 Moreover, 
\begin{eqnarray*}
    \wt\v_k(\frac {y_k}{|y_k|} ) =\v_k(y_k)+ 2(n+1) \log |y_k| \to +\infty ,
\end{eqnarray*}and so, by letting $\wt y_0:= \lim\limits_{k\to +\infty} \frac {y_k}{|y_k|} $ (possibly along a subsequence), we see that  
$\tilde{y_0}$ is  another blow up point for $\wt \v_k$ with $|\tilde{y_0}| =1.$ 
At this point, we can apply Theorem 1 of \cite{Tar_1} to the sequence $\wt{\v}_k$ and obtain the "quantization" property for the blow-up masses relative to those blow-up points. But, recalling that
$\l_{\wt \v} \leq 16\pi,$ 
we deduce that necessarily: $   \s_{\wt \v} (0)=8\pi =\s_{\wt \v}(\wt y_0),$ and no other blow-up point is admissible. 
In conclusion, there holds:
\begin{eqnarray*}
\mathcal{S}_{\wt \v}=\{0, \wt y_0\} \quad \l_{\wt \v}=16\pi \,\, \text {and} \,\,
    \forall \, K\Subset \RR^2\setminus \mathcal{S}_{\wt\v} \implies \max_K \wt\v_k \to -\infty, \text{ as } k\to +\infty.
\end{eqnarray*}

In this way, we have shown that the sequence $\wt{\v}_k$ satisfies  the analogous properties as those of $\v_k$ described in (i), and so for it the same conclusion holds.
\end{proof}

By virtue of Proposition \ref{prop1} and Proposition \ref{prop2}, we can apply to the sequence $\v_k$ and $\wt\v_k$ the blow-up analysis detailed in \cite{Tar_1} concerning the phenomenon of blow-up at "collapsing" zeroes, which was initially investigated in \cite{Lin_Tarantello} and \cite{Lee_Lin_Tarantello_Yang} (see also \cite {Suzuki_Ohtusuka}, \cite{Lee_Lin_Wei_Yang} and \cite{Lee_Lin_Yang_Zhang} for related issues). More importantly, in \cite{Tar_1} we find the point-wise estimates used below, which extend to the "collapsing" case the well known estimates  established in  \cite{Li_Harnack} for the "regular" (non-vanishing) case.

\vspace{.6cm}

We start to discuss the case where:  $\l_\v=16\pi$ and $\s_\v(0)=8\pi$. We set
\begin{eqnarray}\label{W_1 k z}
    W_{1,k}(z):=\left( \prod_{j=1}^s|z-q_{j,k}|^{2n_j} \right) h_{1,k}(z) \to \prod_{j=1}^s  |z-q_j|^{2n_j} =W_1(z),
\end{eqnarray} uniformly as $k\to +\infty$.

From Proposition \ref{prop2}, we have that,
$\mathcal{S}_\v=\{0,y_0\}$ with $y_0\neq  0$. Thus, by taking $\d>0$ sufficiently small we find:
\begin{eqnarray*}
      z_{k}^{(1)}\in B_{\d}(y_0): ~\v_k(z_{k}^{(1)})=\max_{B_\d} \v_k \to +\infty\ ~ \text{ and } \,z_{k}^{(1)}\to y_0  \neq 0, \text{ as }k \to+\infty.
\end{eqnarray*}Furthermore, since both blow-up points of $\v_k$ admit  blow-up mass $8\pi,$ then from Theorem 3 in \cite{Tar_1}, for suitable $\delta_1 > 0$ sufficiently small, we obtain:
\begin{enumerate}
    \item [(a)] $W_{1,k}(0)>0$ (i.e. $q_{j,k}\neq 0$ or 
     $p_{j,k} \neq 0$) and 
    \begin{eqnarray}\label{(a)}
        \v_k(z) =\ln \left( \frac{ e^{\v_k(0)}}{ \left( 1+\frac 1 8 W_{1,k}(0)e^{\v_k(0)}|z|^2\right)^2 }\right) +O(1), \text{ for } |z|<\d_1.
    \end{eqnarray}
    \item [(b)] $W_{1,k}(z_{k}^{(1)})>0$ (i.e. $z_{k}^{(1)}-q_{j,k} \neq 0$ or equivalently $\tau_k z_{k}^{(1)}-p_{j,k} \neq 0$) and 
     \begin{eqnarray}\label{(b)}
        \v_k(z) =\ln \left( \frac{ e^{\v_k(z_{k}^{(1)})}}{ \left( 1+\frac 1 8 W_{1,k}(z_{k}^{(1)})e^{\v_k(z_{k}^{(1)})}|z-z_{k}^{(1)}|^2\right)^2 }\right) +O(1), \text{ for } |z-z_{k}^{(1)}|<\d_1.~~~
    \end{eqnarray} 
    \end{enumerate}

By the ``concentration" property \eqref{concentration}, we see that,
\begin{eqnarray}\label{lambda_k}
    \mu_k:=(W_{1,k}(0))^2 e^{\v_k(0)} \to +\infty, \text{ as } k\to +\infty, 
\end{eqnarray}and by a well-known Harnack type inequality, we know that,
\begin{eqnarray*}
    (W_{1,k}(z_{k}^{(1)}))^2 e^{\v_k(z_{k}^{(1)})} =\mu_k+O(1), \text{ as } k\to +\infty.
\end{eqnarray*}In particular, since $\v_k(0)=\max_{D_{k,\delta}}\v_k,$ we find:
\begin{equation}\label{case i W_k est}
    0< \frac{W_k(0)}{W_k(\tau_k z_{k}^{(1)})} =\frac{W_{1,k}(0)}{W_{1,k}(z_{k}^{(1)})} \leq C \left( \frac{e^{\v_k(z_{k}^{(1)})}}{e^{\v_k(0)}} \right)^{1/2} \leq C.
\end{equation}

In addition we have: $$
    \max_K \v_k =-\ln \mu_k +O(1), \text{ for any compact } K\subset \RR^2 \setminus \mathcal{S}_\v,
$$ and 
\begin{eqnarray*}
    W_{1,k} e^{\v_k} \rightharpoonup 8\pi \d_0+8\pi \d_{y_0}, \text{ weakly in the sense of measures, (locally in } \RR^2)
\end{eqnarray*}
as $k \to + \infty.$

In case $\l_\v=8\pi =\s_\v(0)$, and the second alternative holds in (ii) of Proposition \ref{prop2},  then by letting:
\begin{eqnarray*}
    \wt W_{1,k}(z) =\prod_{j=1}^ s\left| z-\frac{p_{j,k}}{ \ve_k}\right|^{2n_j} h_k( \ve_k z) e^{-2u_X( \ve_k z + z_k)} ,
\end{eqnarray*} we have the analogous of the expression  \eqref{(a)} and \eqref{(b)}  for $\tilde{\varphi_k},$ with $\mathcal{S}_{\tilde{\varphi}} = \{0,\tilde{y_0},\}$ and $\tilde{y_0} \neq 0.$ Hence, $\wt W_{1,k}(0) >0$ (for $k$ large), that is: $p_{j,k}\neq 0$ for all $j=1,\cdots, s,$ and there exists $\wt z_k\in B_\d(\wt y_0)$ satisfying: 
\begin{eqnarray*}
     \wt \v_k(\wt z_k) =\max_{ B_\d(\wt y_0)} \wt \v_k \to +\infty, \,\,  \wt z_k\to \wt y_0 \text{ as } k\to +\infty, 
\end{eqnarray*} with $\wt W_{1,k}(\wt z_k)>0$ (i.e. $\ve_k \wt z_k -p_{j,k} \neq 0$ $\forall j=1,\cdots, s$). In addition,
\begin{eqnarray}\label{case ii W_k est}
    0<\frac{ W_k(0)}{ W_k( \ve_k \wt z_k)} =\frac{ \wt W_{1,k}(0)}{\wt W_{1,k}(  \wt z_k)}\leq C,
\end{eqnarray}for suitable $C>0$. 

Thus, to treat in a unified way the two situations described by (i) or by the second alternative in part  (ii) of Proposition \ref{prop2}, we introduce the following notations:
\begin{eqnarray}\label{Z0}
    Z_k:=\left\{ \begin{array}{rcl}
       \tau_kz_{k}^{(1)}  &  \mbox{ if } \l_\v=16\pi\\
       \ve_k \wt z_k  & \mbox{ \ if \ } \l_\v=8\pi
    \end{array}\right. \to 0 \text{ as } k\to +\infty;   \ \   Z_0:=\left\{ \begin{array}{rcl}
      y_0 \neq 0 &  \mbox{ if } \l_\v=16\pi\\
       \wt y_0  \neq 0  & \mbox{if } \l_\v=8\pi.
    \end{array}\right. 
\end{eqnarray}
In particular, notice that $Z_k \neq 0$ and $\frac{Z_k}{|Z_k|} \to \frac{Z_0}{|Z_0|}$ 
 as $k \to + \infty$ (possibly along a subsequence).  Still, by passing to a subsequence if necessary, we let
\begin{eqnarray}\label{Q_j and l}
    Q_j:=\lim_{k\to +\infty} \frac{\overline{p_{j,k}}}{|p_{j,k}|}, \quad Q_j^* =\lim_{k\to +\infty} \frac{\overline{Z_k -p_{j,k}}}{|Z_k-p_{j,k}|}, 
\end{eqnarray} 
 so that $|Q_j|=1=|Q_j^*|,\quad \forall j=1,\cdots, s$. \\ We set,
$$
    A_0=\prod_{j=1}^s (-Q_j)^{n_j}\quad A_0^*= \prod_{j=1}^s (Q_j^*)^{n_j}\quad A_1^* =\frac{Z_0}{|Z_0|}   \prod_{j=1}^s (Q_j^*)^{n_j}\quad A_2^* =\frac{\overline{Z_0}}{|Z_0|}   \prod_{j=1}^s (Q_j^*)^{n_j},
$$
Since,
$\vert A_0 \vert =\vert A_0^* \vert = \vert A_1^* \vert = \vert A_2^* \vert = 1$ then for any $l\geq 0$ we have: $\vert A_0+lA_0^* \vert^2+\vert lA_1^* \vert^2 \geq \frac{1}{2}$, and consequently: $$( A_0+lA_0^* , 
 lA_1^* )\neq ( 0 , 0 ),\quad \forall l\geq 0.$$  

The following holds:
\begin{lemma}\label{lem3}
    Assume that alternative (iii) holds for $\v_k$ with $\s_\v(0)=8\pi$. Then  there exists $l \geq 0$  such that, for $r>0$ sufficiently small and $\a\in C_2(X)$ with local expression (\ref{eta}) we have: 
    \begin{eqnarray}\label{iii asymptotic behaviour}
        \int_{B(x_0; r))} e^{\xi_k} \<\a\, , \wh \a_k\> dA=
        \end{eqnarray} $$ \frac{32\pi}{(W_k(0))^{1/2}} a(z_k) \overline{\psi_k}(z_k) \left(A_0+l A_0^* + o(1) \right) +\frac{32 \pi |Z_k|}{(W_k(0)^{\frac{1}{2}}} l \left(
\overline{\psi_0}(0) a'(0)  A_1^* + 
\overline{\psi_0'}(0)a(0) A_2^* +o(1) \right)+ o_{r}(1), $$  as $k \to + \infty,$ and $(A_0 + lA_0^*, lA_1^*) \neq (0,0).$   
\end{lemma} 

\begin{proof}
     We set, 
\begin{equation}\label{expansion}
a(z) = a(z_k) + (z - z_k)C_{k}(z), \quad   \psi_k(z) = \psi_k(z_k) +(z - z_k)H_k(z),     
\end{equation} with  $C_{k}(z)$ and $H_k(z)$ holomorphic functions in $\bar{B_\delta}.$ Moreover,\\ $C_{k} \to C$ and  $H_{k} \to H$ uniformly in $\bar{B_\delta}$ as $k\to +\infty,$  and there holds:\\
$a(z) = zC(z), \,\, \psi_0=zH(z), \,\,\mbox{ so that} \,\, C(0) = a'(0),  \, H(0) = \psi_0'(0). $

Assume first that $\l_\v=16\pi$, so (i) of Proposition \ref{prop1} holds. 
In this case, for fixed $R_0 >0 $ sufficiently large, we have: 
$$ \int_{B_{\d}\setminus \{ z: |z|\geq R_0 \tau_k \}} 
e^{\xi_k}  |z|^{2n} \frac{i}{2}dz \wedge d\bar{z}  = \int_{D_{k,\delta}\setminus \{ z: |z|\geq R_0 \}}e^{\varphi_k}  |z|^{2n} \frac{i}{2} dz \wedge d\bar{z}\to 0 $$ as $k \to + \infty.$ Consequently, for given
$0 \leq b \leq n,$ we obtain:
$$
\int_{B_{\d}\setminus \{ z: |z|\geq R_0 \tau_k \}} e^{\xi_k}  (\prod\limits_{j=1}^s |z-p_{j,k}|^{n_j}) |z|^b \frac{i}{2} dz \wedge  d \bar{z} = o(\frac{1}{\tau_k^{n-b}}) $$
as $k \to + \infty.$ 
Therefore, by recalling \eqref{expansion}, we can estimate:
{\footnotesize \begin{eqnarray}\label{palletta}
&& \int_{B(x_0 ; r))} e^{\xi_k} \<\a\, , \wh \a_k\> dA = 4 \int_{B_\delta} e^{\xi_k}\hat{a_k}(z + z_k)a(z + z_k)e^{-2u_X( z + z_k)} \frac{i}{2}dz \wedge d\bar z + o_{r}(1)
= \nonumber\\
&& \;=\; \frac{4}{\tau_k^n}  \left[ \overline{\psi}_k(z_k) a(z_k) \left( \int_{|z| \leq R_0} e^{\varphi_k(z)} \prod\limits_{j=1}^s (\bar{z}-\overline{q_{j,k}} )^{n_j}   e^{-2u_X(\tau_k z + z_k)} \frac{i}{2} dz  \wedge d \bar{z} + o(1) \right) \right]+ \nonumber \\
&& +\; \frac{4}{\tau_k^{n-1}} \left[ \int_{|z| \leq R_0} e^{\varphi_k(z)} \prod\limits_{j=1}^s (\bar{z}-\overline{q_{j,k}} )^{n_j}  \left(z C_k(\tau_kz + z_k) \overline{\psi_k(z_k)} + \bar{z} \overline{H_k}(\tau_kz + z_k)a(z_k)\right) 
 e^{-2u_X(\tau_kz +z_k)}  \frac{i}{2} dz \wedge d \bar{z} + \right.   \nonumber \\
&&\left.+ \,\, o(\frac{1}{\tau_k^{n-1}}) \right] \, + \, o_{r}(1).
 \end{eqnarray}}

To illustrate the origin of \eqref{iii asymptotic behaviour}, let us treat the (simpler) case of blow-up without ``collapsing". As the blow-up mass is $8\pi,$ this requires the "non-vanishing" condition: $W_1(0)=\prod\limits_{j=1}^s |q_j|^{2n_j}>0$ (i.e. $q_j\neq 0,\quad \forall j=1,\ldots, s$) and $W_1(y_0) =\prod\limits_{j=1}^s |y_0-q_j|^{2n_j}>0$ (i.e. $q_j\neq y_0, \quad \forall j=1,\ldots, s$.). Then in \eqref{Q_j and l}, we have: $$Q_j=\frac{\overline{q_j}}{|q_j|} \quad Q_j^* =\frac{\overline{y_0-q_j}}{|y_0-q_j|} \quad j=1,\ldots, s, \,\, \text{ and } \quad e^{\v_k} \to \frac{8\pi }{W_1(0)} \d_0+\frac{8\pi}{W_1(y_0)} \d_{y_0},$$ weakly in the sense of measures (locally in $\RR^2$), as $k \to + \infty.$
Thus, in the "non-collapsing" case, we obtain:
$$
\int_{B(x_0; r))} e^{\xi_k} \<\a\, , \wh \a_k\> dA 
$$
$$  = \frac{32 \pi}{\tau_k^n}  \left[ \overline{\psi_k}(z_k)a(z_k) \left(\frac{A_0}{{W_1(0)}^{\frac{1}{2}}} + \frac{A^*_0}{{W_1(y_0)}^{\frac{1}{2}}} + o(1) \right) \right]  $$
 $$ + \frac{32\pi}{\tau_k^{n-1}} \left[ \overline{\psi_0}(0)a'(0) \frac{A^*_0}{{W_1(y_0)}^{\frac{1}{2}}}y_0 + \overline{\psi_0'}(0)a(0) \frac{A^*_0}{{W_1(y_0)}^{\frac{1}{2}}}\bar{y_0} + o(1) \right]\,\, + \, o_{r}(1), \quad k \to + \infty. 
$$
We define (possibly along a subsequence): 
 \begin{equation}\label{l} l:= lim_{k \to + \infty} \frac{{W_{1,k}(0)}^{\frac{1}{2}}}{{W_{1,k}(z_k)}^{\frac{1}{2}}} \\  = lim_{k \to + \infty} \frac{{W_k(0)}^{\frac{1}{2}}}{{W_k(Z_k)}^{\frac{1}{2}}} \geq 0, \end{equation}
so that, in this case, we have: $ l = \frac{{W_1(0)}^{\frac{1}{2}}}{{W_1(y_0)}^{\frac{1}{2}}}.$

Hence, by observing that,
${{W_k(0)}^{\frac{1}{2}}} = \tau_k^n ({{W_{1,k}(0)}^{\frac{1}{2}}}),$ we derive:

$$
\int_{B(x_0; r)} e^{\xi_k} \<\a\, , \wh \a_k\> dA 
= \frac{32 \pi}{{W_k(0)}^{\frac{1}{2}}}  \left[ \overline{\psi_k}(z_k)a(z_k) \left(A_0 + l A_0^*  +o(1) \right) \right]  
$$
$$ + \frac{32\pi}{{W_k(0)}^{\frac{1}{2}}} l|Z_k| \left[ \overline{\psi_0}(0) a'(0)  A_1^*  + \overline{\psi_0'}(0)a(0)  A_2^* + o(1)  \right] \,\, +\, o_{r}(1), 
$$ as $k \to + \infty,$
and (\ref{iii asymptotic behaviour}) is established in this case.

    Next, let us consider the "collapsing"  case, where  blow-up occurs at an accumulation point of several different zeroes in $\{q_{1,k},\ldots, q_{s,k}\}.$ 
    
    If this occurs at the origin, then by recalling that $| q_{s,k}|=1,$ necessarily:  $s\geq 3,$ and we find $s_1\in \{2, \cdots, s-1\}$ such that (along a subsequence) we have: 
    $q_{j,k} \to 0 \text{ for } j=1,\ldots, s_1,$ while $q_{j,k} \to q_j\neq 0 \text{ for } j=s_{1}+1\ldots, s, \, \text{ as } k\to +\infty.$
    Moreover, we find $s_2\in \{1,\ldots, s_1\}$ such that, as $k\to +\infty,$ we have:
\begin{eqnarray*}
&& \tau_{1,k}:=|q_{1,k}|\to 0, \frac{q_{j,k}}{\tau_{1,k}} \to z_j\neq 0, \text{ for } j=1,\ldots, s_2, \\ &&  \frac{ \tau_{1,k}}{ |q_{j,k}|} \to 0, \text{ for } j=s_2+1,\ldots,s.
\end{eqnarray*}  

Thus, by setting: $ \ve_{j,k}=\frac{\tau_{1,k}}{|q_{j,k}|}$, we find  $\ve_{j,k} \to \frac 1{|z_j|} $ for $j=1,\ldots, s_1, $ while $\ve_{j,k}\to 0$ for $j=s_1+1,\ldots, s,$ as $k\to +\infty$. We define: 
$$
    \phi_k(z) := \v_k(\tau_{1,k} z)+2 \ln \tau_{1,k} +\ln W_{1,k}(0) \text{ and } \d_k:= \tau_k \tau_{1,k} \to 0,
$$ satisfying: 
\begin{eqnarray*}
   &&  -\Delta \phi_k = \prod_{j=1}^s \left|\ve_{j,k} z-\frac{q_{j,k}}{|q_{j,k}|} \right|^{2n_j}h_{1,k} (\tau_{1,k} z) e^{\phi_k}, \text{ for } |z|< \frac{\d_1}{\tau_{1,k}},
    \\&&  \phi_k(0)=\v_k(0)+ 2\ln |q_{1,k}| +\ln W_{1,k}(0) \geq \v_k(0) +2\ln W_{1,k}(0) =\ln\mu_k \to +\infty,
\end{eqnarray*} as $k\to +\infty$ (recall \eqref{lambda_k}), for $\delta_1>0$ sufficiently small. 
Hence, the origin is a ``non-collapsing" blow-up point of $\phi_k$, and  by \cite{Li_Harnack} (or \eqref{(a)}), there holds:
$$
        e^{\phi_k} \to 8\pi \d_0 \text{ weakly in the sense of measures locally in } \RR^2.
$$
and
    $$
        \phi_k(z) =\ln \left( \frac{ e^{\phi_k(0)}}{ \left( 1+\frac {h_{1,k}(0)}{8} e^{\phi_k(0)}|z|^2\right)^2 }\right) +O(1), \text{ for } |z|<\frac{\d_1}{\tau_{1,k}}.
        $$

As a consequence, for any $R >0,$ we have:
    $$
     \int_{R\leq |z| \leq \frac{\d_1}{\tau_{1,k}}} e^{\phi_k}  \prod_{j=1}^s \left| \ve_{j,k} z-\frac{q_{j,k}}{|q_{j,k}|} \right|^{2n_j} h_{1,k}(\tau_{1,k} z) \frac{i}{2} dz  \wedge d\bar z \to 0, 
    $$
 as  $k\to +\infty;$ and we deduce,   
{\footnotesize
     \begin{eqnarray*} && \left| \int_{R\leq |z|\leq \frac{\d_1}{\tau_{1,k}}} e^{\phi_k} \prod_{j=1}^s \left( \ve_{j,k} z - \frac{q_{j,k}}{|q_{j,k}|} \right)^{n_j}  \frac{i}{2}dz \wedge d\bar z \right|
    \\
    && \;\leq\; C  \left( \int_{R\leq |z|\leq \frac{\d_1}{\tau_{1,k}}}  e^{\phi_k}  \prod_{j=1}^s \left| \ve_{j,k} z-\frac{q_{j,k}}{|q_{j,k}|^{2n_j}} \right|^{2n_j} \, \!\!\!\!\! \frac{i}{2} dz \wedge  d \bar{z} \right)^{1/2}   \left( \int_{R\leq |z|\leq \frac{\d_1}{\tau_{1,k}}} e^{\phi_k} \!\frac{i}{2} dz \wedge d \bar{z}  \right) ^{1/2}  \to 0, 
    \end{eqnarray*}}as $k\to +\infty$. 
    
    With this information, we can control the first integral term in \eqref{palletta} as follows: 
    $$ \begin{array}{l}
        \int_{B_{\d_1}} e^{\v_k(z)} \prod_{j=1}^s \left( \bar{z}-\overline{q_{j,k}}\right) ^{n_j}   e^{-2u_X(\tau_k z + z_k)} \frac{i}{2} dz \wedge  d\bar z
        \\ =  \frac 1{ ( W_{1,k}(0))^{1/2}} \int_{|z|\leq \frac{\d_1}{ \tau_{1,k}} }  e^{\phi_k} \prod_{j=1}^s \left( \ve_{j,k} \bar{z}- \overline{\frac{q_{j,k}}{|q_{j,k}|}} \right)^{n_j}  e^{-2u_X(\d_k z + z_k)} \frac{i}{2}  dz  \wedge d\bar z \notag
        \\ =  \frac 1{ ( W_{1,k}(0))^{1/2}}\left( \int_{|z|\leq R }  e^{\phi_k} \prod_{j=1}^s \left( \ve_{j,k} \bar{z}- \overline{\frac{q_{j,k}}{|q_{j,k}|}} \right)^{n_j} 
        e^{-2u_X(\d_k z + z_k)} \frac{i}{2} dz \wedge  d\bar z+o(1)\notag \right)
        \\ =  \frac 1{ ( W_{1,k}(0))^{1/2}} \left( 8\pi \prod 
        (-Q_j)^{n_j}  +o(1) \right)  , 
        \text{ as } k\to 
        +\infty.
  \end{array} $$
  Similarly, concerning the second integral terms in \eqref{palletta}, we get:
{\footnotesize \begin{eqnarray*} 
&&\int_{B_{\delta_1}} e^{\varphi_k(z)} \left( \prod_{j=1}^s \left(  \bar{z}-\overline{q_{j,k}}) \right)^{n_j}   \right) z C_k(\tau_k z + z_k) e^{-2u_X(\tau_k z + z_k)} \frac{i}{2}dz \wedge  d \bar{z} =\\
&& \; =\; \frac{1}{\left(W_{1,k}(0) \right)^{\frac{1}{2}}} \left( \int_{|z| \leq R} e^{\phi_k} \left( \prod_{j=1}^s \left(\varepsilon_{j,k} \bar{z} -\frac{\overline{q_{j,k}}}{|q_{j,k}|} \right)^{n_j} \right)\tau_{1,k}zC_k(\d_k z + z_k)e^{-2u_{X}(\d_k z + z_k)} \frac{i}{2}dz \wedge d \bar{z} +o(1) \right)   \\
&& \;=\; 
o \left( \frac{1}{({W_{1,k}(0))}^{\frac{1}{2}}}  \right), 
\mbox{ \ as $k \to + \infty,$ \ } 
\end{eqnarray*}}
and analogously, 
$$ \int_{B_{\delta_1}} e^{\varphi_k(z)} \left(\prod_{j=1}^s \left(  \bar{z}-\overline{q_{j,k}}) \right)^{n_j}   \right) \bar{z} \overline{H_k}(\tau_k z + z_k) e^{-2u_{X}(\tau_k z + z_k)} \frac{i}{2}dz \wedge  d \bar{z} = o \left( \frac{1}{{(W_{1,k}(0))}^{\frac{1}{2}}}  \right), 
 $$
as  $k \to + \infty.$
\medskip

Next, by assuming that also the blow-up point $y_0$ is of "collapsing" type, then using similar arguments (and notations) for the translated sequence $\varphi_k(z + z_k^{(1)})$ we obtain:
{\footnotesize \begin{eqnarray*} 
&&\int_{|z-z_k^{(1)}| < \delta_1} e^{\varphi_k(z)} \left( \prod_{j=1}^{s} (\bar{z}-\overline{q_{j,k}})^{n_j} \right)  e^{-2u_X(\tau_k z +z_k)}\frac{i}{2}dz \wedge  d \bar{z} =\\
&& \;=\; \int_{|z| < \delta_1} e^{\varphi_k(z + z_k^{(1)})} \left( \prod_{j=1}^{s} (\bar{z} + (\overline{z_k}^{(1)} - \overline{q_{j,k}}))^{n_j} \right)e^{-2u_X( Z_k + z_k + \tau_k z)}\frac{i}{2}dz \wedge d \bar{z} =\\
&& \;=\; \frac{1}{ \left(W_{1,k}(z_k^{(1)})\right)^{\frac{1}{2}} } \left( \int_{|z| \leq R} e^{\phi_k(z)}  \left(\prod_{j=1}^{s} \left(\varepsilon_{j,k} \bar{z} + \frac{\overline{z_k}^{(1)} - \overline{q_{j,k}}}{|z_k^{(1)}-q_{j,k}|} \right)^{n_j} \right)e^{-2u_X(Z_k + z_k + \d_kz)} \frac{i}{2} dz \wedge  d \bar{z} + o(1)  \right) =  
\\
&& \; =\; \frac{8\pi}{ \left(W_{1,k}(z_k)\right)^{\frac{1}{2}} } \left(A_0^* + o(1)\right) =  
 \frac{8\pi}{ \left(W_{1,k}(0)\right)^{\frac{1}{2}} }l \left(A_0^* + o(1)\right), \mbox{as $k \to +\infty$  } 
 \end{eqnarray*}}
(recall that $l$ is defined in ( \ref{l})).
Analogously we have: 
{\footnotesize \begin{eqnarray*}
 && \int_{|z-z_k^{(1)}| < \delta_1} e^{\varphi_k(z)} \left( \prod_{j=1}^{s} (\bar{z}-\overline{q_{j,k}})^{n_j} \right)zC_k(\tau_kz + z_k))  e^{-2u_X(\tau_k z + z_k)}\frac{i}{2}dz \wedge  d \bar{z} =
\\
&& \;=\;
 \int_{|z| < \delta_1} e^{\varphi_k(z+z_k^{(1)})} \left( \prod_{j=1}^{s} (\bar{z}+(\overline{z_k}^{(1)} -\overline{q_{j,k}}))^{n_j} \right) (z_k^{(1)}+z)C(Z_k+z_k+ \tau_k z)e^{-2u_X( Z_k + \tau_k z+z_k)}\frac{i}{2}dz \wedge d \bar{z} =
\\
&&= \frac{1}{ \left(W_{1,k}(z_k^{(1)})\right)^{\frac{1}{2}}} \int_{|z| \leq R}e^{\phi_k(z)} \prod_{j=1}^{s}\left(\varepsilon_{j,k}\bar{z}+\frac{\overline{z_k}^{(1)}-\overline{q_{j,k}}}{|z_k^{(1)}-q_{j,k}|} \right)^{n_j}(z_k^{(1)}+\tau_k z)C_k(Z_k+z_k+\d_k z)e^{-2u_X(Z_k+\d_kz+z_k)}\frac{i}{2} dz \wedge d \bar{z} \\  
 && \, \,  +o(1) \,   \;=\;
\frac{8\pi|z_k^{(1)}|}{ \left(W_{1,k}(z_k^{(1)})\right)^{\frac{1}{2}}} \left( \prod_{j=1}^{s} (Q_j^*)^{n_j} \frac{y_0}{|y_0|}a'(0) +o(1) \right) =  
 \frac{8\pi|z_k^{(1)}|}{ \left(W_{1,k}(0)\right)^{\frac{1}{2}}}l(A_1^*a'(0) + o(1)),
 \end{eqnarray*}}
as $ k \to +\infty.$
\\

Similarly we find:
{\footnotesize $$ 
\int_{|z-z_k^{(1)}| < \delta_1} e^{\varphi_k(z)} \prod_{j=1}^{s}(\bar{z}-\overline{q_{j,k}})^{n_j}  \bar{z} \overline{H_k}(z_k + \tau_k z) e^{-2u_X(z_k + \tau_k z)}\frac{i}{2}dz \wedge d \bar{z} =
\frac{8\pi|z_k^{(1)}|}{ \left(W_{1,k}(0)\right)^{\frac{1}{2}}}l \left(  A_2^* \overline{\psi_0'}(0) + o(1) \right),  $$}
as $k \to + \infty.$
\vskip0.5cm
In case the blow-up point $y_0$ is \underline{not} of "collapsing" type, hence $W_1(y_0) \neq 0,$ then we can use the previous analysis to attain the same conclusion for the above integral terms. However, we see  from  (\ref{l}) that in this case we have: $l = 0.$
Similar considerations apply in case it is the origin to be of non-collapsing type while $y_0$ is a "collapsing" blow-up point.

In either cases, by recalling the concentration property (\ref{concentration}) and by using the "asymptotic" behaviour derived above for the integral terms involved in  (\ref{palletta}), 
we obtain that (\ref{iii asymptotic behaviour}) holds with suitable $l \geq 0,$ when $\lambda_{\varphi} = 16 \pi.$ \vskip0.5cm

Next, assume that $\l_\v=8\pi$ so that the blow-up situation is described by part (ii) of  Proposition \ref{prop2}. In case the second alternative in (ii) holds,  then we can argue exactly as above for the scaled sequence $\wt \v_k$ in \eqref{sequence of tilde varphi}. Therefore on the basis of the notation introduced in \eqref{Z0}, we arrive at the desired conclusion \eqref{iii asymptotic behaviour}, in this case as well. 
\medskip
    Finally, in case the first alternative holds in (ii), then $\varphi_k$ satisfies  the estimate \eqref{(a)} in $B_R$ and we have : $n_1= n_2 = 1 \,\, \text{(i.e. n=2 ) and}\,\, W_1(0) =|q_1|^2|q_2|^2 >0$. In particular, 
    \begin{eqnarray*}
        e^{\v_k} \rightharpoonup \frac{8\pi}{W_1(0)} \d_0, \text{ weakly in the sense of measures (locally in }\R^2),
    \end{eqnarray*}
and there exist suitable constants $C_0 >0$ and $R_0 > 0 $ sufficiently large, such that, 
$$ \frac{1}{C_0}|z|^2 \leq |z-q_{1,k}||z-q_{2,k}| \leq C_0 |z|^2 \quad \text{for} \,\, |z| \geq R_0.$$
Thus, for all $R \geq R_0,$ we can estimate:
$$| \int_{D_{k,\d}\setminus \{ z: |z|\geq R \}} e^{\varphi_k(z)}(z- q_1)(z-q_2)e^{-2u_k(\tau_kz+ z_k)}\frac{i}{2}dz \ \wedge d \bar{z} | \leq $$
$$C \int_{D_{k,\d}\setminus \{ z: |z|\geq R \}}e^{\varphi_k(z)}|z|^2e^{-2u_k(\tau_kz + z_k)}\frac{i}{2}dz \wedge  d \bar{z} \leq  $$
$$\frac{C}{R^2} \int_{D_{k,\d}\setminus \{ z: |z|\geq R \}}e^{\varphi_k(z)} |z|^4 \frac{i}{2}dz \wedge d \bar{z}  \leq  \frac{C}{R^2}.$$

Therefore, for any $R > 0$ sufficiently large, we  find:
{\footnotesize \begin{eqnarray*}
&& \int_{B(x_0; r)}e^{\xi_k} \<\a\, , \widehat{\a_k}\> dA \\&=& \frac {4} {\tau_k^2}  \left( \int_{|z|\leq R} e^{\varphi_k(z)}(\bar{z}-\overline{q_{1,k}})(\bar{z}-\overline{q_{2,k}}) a(\tau_k z + z_k) \overline{\psi_k}(\tau_k z + z_k) e^{-2u_X(\tau_k z + z_k)} \frac{i}{2} dz \wedge d \bar z+O \left(\frac 1 {R^2} \right)\right) + o_{r}(1)
\\&=& \frac{4}{(W_{k}(0))^{\frac 1 2}} \left( 8\pi \frac{\bar{q_1}}{|q_1|}\frac{\bar{q_2}}{|q_2|}a(0)\overline{\psi_k}(0)+o(1)+O \left( \frac {1} {R^2} \right) \right) + o_{r}(1),
    \end{eqnarray*}}  as $k \to + \infty.$ Thus, by letting $ R \to + \infty,$  we may conclude:
    \begin{eqnarray*}
&& \int_{B(x_0; r)} e^{\xi_k} \<\a\, , \widehat{\a_k} \> dA=  \frac{32\pi}{(W_{k}(0))^{\frac 1 2}} \left( A_0 a(0)\overline{\psi_k}(0)
+o(1)\right) + o_{r}(1),
    \end{eqnarray*} as $k\to +\infty$, and in this case, \eqref{iii asymptotic behaviour} holds with $l=0$. \end{proof} 
    \begin{remark}\label{summarize}
    By taking $\theta_{1,k} =\left( W_k(0)\right)^{1/2} = \vert|\hat\alpha_k(x_k)\vert|$  and $\theta_{2,k}=|Z_k| $, we see that
    \begin{eqnarray*}
        0< \theta_{1,k} \leq C\theta_{2,k} \to 0, \text{ as } k\to +\infty ~~( C>0),
    \end{eqnarray*}
    and we can summarize (\ref{iii asymptotic behaviour}) as follows:
    \begin{eqnarray}\label{4.48bis}
        && \int_{B(x_0; r)} e^{\xi_k}  \<\a\,, \wh{\a}_k\> dA  =\\
        && \;=\; \frac 1 {\theta_{1,k}} \left[ a(z_k)(\Lambda_1+ o(1))
       + \theta_{2,k}\Lambda_2( a'(0) +o(1))\right]\,+\, o_r(1) \nonumber
    \end{eqnarray}  as $k\to +\infty,$  with suitable $\Lambda_1$ and $\Lambda_2$ complex numbers  such that $(\Lambda_1,\Lambda_2)\neq (0,0)$.  

Notice in particular that, when $\Lambda_1 \neq 0,$ then the first term in \eqref{4.48bis} becomes the leading term. So in this situation, we can replace (up to infinitesimal order)  $a(z_k)$ with $a(0)$ ( as $z_k\to 0 $)  and englobe also the second term into the infinitesimal term in \eqref{4.48bis}. In other words,
we have:
\begin{equation} \label{alfazero1}
\,\mbox{ if } \Lambda_1 \neq 0 \,\, \mbox{then } \,\, \int_{B(x_0; r)} e^{\xi_k}  \<\a\,, \wh{\a}_k\> dA  = \frac { \Lambda_1}{ \theta_{1,k}}( a(0) +o(1))\,+\, o_r(1). 
\end{equation}
In this way, we see that the expression \eqref{4.48bis}  contains the statement of Lemma \ref{lem2}  once we take:\\
$\theta_{1,k} = \frac {\tau_k^2,}{4} \,\, \theta_{2,k} = \tau_k, \,\, \Lambda_1 = \frac{-32 \pi {(\bar{q_1}- \bar{q_2})}^2}{|q_1-q_2|^4}\overline{\psi}_0(0) \neq 0$ (and $\Lambda_2=0).$

While, if $\Lambda_1 = 0,$ ( and so $\Lambda_2 \neq 0,$), then we take $\alpha (x_0)=0$ and by using the approximation $\alpha_k$ relative to the sequence $x_k \to x_0$ as given in Lemma \ref{approssimazione}  (so that $a_{k}(z_k)=0$) we find:
\begin{equation} \label{alfazero}
\int_{B(x_0; r)} e^{\xi_k}  \<\a_k\,, \wh{\a}_k\> dA  = \frac {\theta_{2,k}}{ \theta_{1,k}} \Lambda_2( a'(0) +o(1))\,+\, o_r(1), \,\,\mbox{ with } \Lambda_2 \neq 0.
\end{equation}
Hence, we see that  \eqref{4.48bis} contains also the statement of Lemma \ref{lem1}, where by hypothesis $\alpha (x_0) =0$ and simply one takes:\\
$\theta_{1,k} = \theta_{2,k} = max\{e^{-\xi_k(0)},|p_k|^4 \},$  $\Lambda_1 =0$ and $\Lambda_2 = 64 \pi \overline{\psi}_0(0) \neq 0.$ 

In other words (\ref{4.48bis}) expresses the asymptotic behaviour of the given integral terms either when $s=1$ or when $s\geq 2$ and either $\varphi_k(0) = O(1)$ or $\varphi_k$ blows-up at the origin with blow-up mass $8\pi.$ \end{remark}

Next by taking further scaling and by means of an induction argument  we show that also in the remaining case:
\begin{equation}\label{sigma_varphi(0)}
\sigma_{\varphi}(0) = 16 \pi = \lambda_{\varphi}
\end{equation}
the general expression (\ref{4.48bis})
remains valid.
Indeed, in case (\ref{sigma_varphi(0)}) holds,  the sequence $\varphi_k$ admits the exact same blow-up character of $\xi_k,$ the sequence we started with. In particular the origin must be a zero of $W_1$ in \eqref{W_1 k z}. 
 
 Since we have $|q_s|=1$, we have indeed improved our situation, since the number of (possible) ``collapsing" zeros for $\varphi_k$  has decreased by at least one unit, and so we can proceed with an induction argument. To this purpose we note that  for suitable $s_1\in \{1,\ldots, s-1\}$ and $\d_1>0$ sufficiently small, we have 
\begin{eqnarray*}
    -\De \v_k= \prod_{j=1}^{s_1} |z-q_{j,k}|^{2n_j} h_k^{(1)}(z) e^{ \v_k} -g_{1,k} (z) \text{ in } \overline{B}_{\d_1},
\end{eqnarray*}
where\begin{eqnarray*}
    0\leq |q_{1,k}| \leq \cdots \leq |q_{s_1,k}| =\tau_k^{(1)} \to 0, \mbox{ \ while $q_{j,k} \to q_j \neq 0, \, \text{ for } j= s_1+1, \ldots s,$} 
\end{eqnarray*} as   $k\to +\infty$.

Thus:
\begin{eqnarray*}
    h_k^{(1)}(z)=\prod_{j=s_1+ 1}^{s} |z-q_{j,k}|^{2n_j}  h_{1,k}(z) \to h_0^{(1)}(z) \quad \text{uniformly in } \,\, \overline{B}_{\d_1},   
\end{eqnarray*}

$0<a\leq h_k^{(1)}(z) \leq b_1$ and $|\n h_k^{(1)} |\leq A_1$ in $\overline{B}_{\d_1}.$
\vskip0.3cm
Furthermore for $N_1:=\sum_{j=1}^{s_1}n_j$, we have:
\begin{eqnarray*}
    W_{1,k}(z):=\prod_{j=1}^{s_1} |z-q_{j,k}|^{2n_j} h_k^{(1)} (z) \to |z|^{2N_1} h_0^{(1)}(z) \text{ as } k\to +\infty, \text{ uniformly in } \overline{B}_{\d_1}.
\end{eqnarray*}

More importantly, by \eqref{sigma_varphi(0)} we get,
\begin{eqnarray*}
&& \lim_{\d\to 0} \lim_{k\to +\infty}    \int_{B_{\d}} W_{1,k}(z) e^{\v_k} \frac{i}{2}dz \wedge d\bar z 
 =\\
 && \;=\;
 \lim_{\d\to 0} \lim_{k\to +\infty} \int_{B_\d} e^{\v_k}\prod_{j=1}^{s_1} |z-q_{j,k}|^{2n_j} h_k^{(1)}(z) e^{\v_k} \frac{i}{2} dz \wedge d\bar z=16\pi.
\end{eqnarray*}

In particular,  \begin{eqnarray*}
        \psi_{1,k}(z):=\prod_{j=s_1+1}^s (z-q_{j,k})^{n_j} \psi_k(\tau_k z) 
        \to \prod_{j=s_1+1}^s  (z-q_j) ^{n_j} \psi_0(0) :=\psi_0^{(1)} (z), 
\end{eqnarray*}
uniformly in $\overline{B}_{\d_1}$  as $k\to +\infty$; with $\psi^{(1)}_0(z)$  holomorphic and never vanishing in  $\overline{B}_{\d_1},$ provided $\delta_1>0$ is sufficiently small.

In case $s_1=1$, then $n=n_1= 1$ and by setting $q_k := q_{1,k},$ we can argue for $\varphi_k$ exactly as in Lemma \ref{lem1} to conclude that, for $\alpha \in C_2(X)$ with $\alpha(x_0) = 0$ we can appeal to the apprioximation \eqref{app_alpha1} and \eqref{app_alpha2} to find,
$$ \int_{|z| \leq R_0} e^{\varphi_k}(\bar{z}-\overline{q_k})\overline{\psi_{1,k}}(z) zC_k(\tau_kz + z_k)e^{-2u_X(\tau_kz + z_k)}\frac{i}{2}dz \wedge d \overline{z} = 16 \pi \overline{\psi_0^{(1)}(0)}a'(0) + o(1)$$
as $k \to +\infty.$
Hence, in view of (\ref{palletta}), in this case we obtain the equivalent of (\ref{total asymp behaviour}) with $\psi_0(0)$ replaced by $\psi_0^{(1)}(0) \neq 0.$

Similarly if $s_1 \geq 2,$ then we use the scaled sequence:
$$
        \v_k^{(1)}(z)=\v_k( \tau^{(1)}_k z) + 2(N_1+1) \ln(\tau_k^{(1)}), $$
and arrive at the appropriate asymptotic behaviour for the integral terms in the left hand side of (\ref{palletta}) either when $\varphi_k^{(1)}(0) = O(1)$ or when $\varphi_k^{(1)}$ blows up at the origin with blow-up mass $8 \pi.$ Thus we may obtain for $\varphi_k^{(1)}$ the analogous statement of Lemma \ref{lem2} and of Lemma \ref{lem3}, respectively.
In the remaining case, where we have blow-up at the origin with blow-up mass $16 \pi,$ we can iterate the arguments above for $\varphi_k^{(1)}$ and continue in this way. Since at each iteration step we decrease by at least one unit the number of "collapsing" zeroes, after finitely many steps,  only one of the three situations described above must occur. 

Thus by combining the results of Lemma \ref{lem1}, Lemma \ref{lem2} and Lemma \ref{lem3} (as we did in Remark \ref{summarize}), by the above iteration argument we arrive at the following conclusion:   

\begin{thm}\label{mainthmasymp}
Let $x_0 \in X$ be a blow-up point for $\xi_k$ with blow-up mass $16 \pi.$ 
There  exist  constants $\Lambda_1$ and $\Lambda_2$ in $\mathbb{C}$ with $(\Lambda_1,\Lambda_2) \neq (0,0)$ and sequences $0 < \theta_{1,k} \leq C \theta_{2,k}$ with $C > 0$ and $\theta_{2,k} \to 0$ as $k \to + \infty $,  such that, for $r> 0$ sufficiently small, the following holds: 

$$
\begin{array}{l}
(i) \ \mbox{If \ $\Lambda_1 \neq 0$ then for every $\alpha \in C_2(X)$ with local expression (\ref{eta}) around $x_0$} \\ \text{we have:}\\
\bigskip
\int_{B(x_0; r)} e^{\xi_k}\<\alpha, , \hat{\alpha_k}\> d A  = \frac{ \Lambda_1}{\theta_{1,k}} \left( a(0)  + o(1) \right) + o_r(1), \,\, k \to + \infty;  \\

(ii) \ \mbox{If \ $\Lambda_1=0$ \,\, (and so $\Lambda_2 \neq 0$) then for every $\alpha \in C_2(X)$ with $\alpha (x_0)=0$}\\ 
\mbox{ and local expression (\ref{eta}) around $x_0$ we have:} \\
\int_{B(x_0; r)} e^{\xi_k}\<\alpha_k ,\hat{\alpha_k}\> d A  = \frac{\theta_{2,k}}{\theta_{1,k}} \Lambda_2 \left( a'(0) + o(1) \right) + o_r(1), \,\, k \to + \infty;
\end{array}
$$
where $\alpha_k $ is given by Lemma \ref{approssimazione} with respect to the sequence $x_k \to x_0,$ and $\alpha_k \to \alpha$ as $k \to + \infty,$ and $ o_r(1)\to 0$ as $ r \to 0^+,$ uniformly in $k.$  
\end{thm}
\begin{remark}\label{molt1}
Theorem \ref{mainthmasymp} extends the asymptotic analysis carried out in \cite{Tar_2}, around a blow-up point  $x_0$ with blow-up mass $8\pi.$ In this situation, one can always ensure that $\Lambda_1 \neq 0$ in the corresponding expansion and the following holds: 
\begin{equation}\label{molt1bis}
\int_{B(x_0; r)} e^{\xi_k}\<\alpha,\hat{\alpha_k}\> d A  = \frac{1}{\theta_{1,k}} \left( a(0) \overline{\Lambda_1} + o(1) \right) + o(1) + o_r(1), \quad k \to + \infty,  
    \end{equation}  
with $\theta_{1,k}=||\hat\alpha_k||(x_k)$
and $\alpha \in C_2(X)$ 
with local expression (\ref{a_0}) around $x_0.$
Note that in this case the sequence $\theta_{1,k}$ does not necessarily converges to zero, as $k \to + \infty$. 
For details see \cite{Tar_2}, where (\ref{molt1bis}) is derived by similar, yet simpler, arguments as those presented above.   
\end{remark}

\section{The proof of Theorem  \ref{thm3}, Theorem  \ref{thmg2} and Theorem \ref{thm1.1}}\label{mainTheorems} 
To prove Theorem \ref{thmg2} 
we recall that a Riemann surface of genus $\gg =2$ is \underline{hyperelliptic}, i.e.
it admits a unique biholomorphic involution:
\begin{eqnarray*}
    j: X\to X
\end{eqnarray*}(cf. \cite{Miranda}, \cite{Griffiths_Harris}),  with
  exactly \underline{six} distinct fixed points, i.e.
    $\{q\in X: j(q)=q\}.$ 

We recall from  \cite{Tar_2} that, for genus $\gg=2$ and $[\b]\neq 0$, the functional $F_t$ is always bounded from below, but possibly not coercive. 
    
    Moreover, the infimum of $F_t$ is $\underline{not}$ attained iff  $\xi_{t_k}$ blows-up, for any given sequence $t_k\to 0^+$, and it admits exactly one blow-up point and the "concentration" property holds. Thus, the blow-up divisor takes the form: $D = q,$ for suitable $q\in X,$ and by Theorem \ref{thmB.3} we have that: $\tau(q) =[\b]_\PP.$ 
    
By using the equivariance of $F_t$ as establish in Appendix 2,  we shall show that the blow-up point $q$ must be one of the fixed points of $j,$ 
i.e. $j(q)=q$, and moreover it is uniquely identified, independently on the choice of the given sequence $t_k\to 0^+$. 

To this purpose, we consider the pullback map:
    \begin{eqnarray*}
        j^*: A^{0,1}(X,E)\to A^{0,1}(X,E),
    \end{eqnarray*} defined in Appendix 2 and used in Proposition \ref{propA.2} to  show the following:  
    
    if $\b_0\in A^{0,1}(X,E)$ is harmonic (with respect to $g_X$) then $j^*(\b_0)=\b_0$.

    Consequently, by using this fact together with the invariance  of the Donaldson functional, as established in Proposition \ref{propA.1} of the Appendix 2, 
we deduce that,
    \begin{eqnarray*}
        F_t(u_t,\eta_t)=F_t(u_t\circ j,j^*(\eta_t)).
    \end{eqnarray*} In other words, $(u_t\circ j,j^*(\eta_t))$ is also a minimizer for $F_t$, and due to the uniqueness of the minimizer we conclude that,
    \begin{eqnarray*}
        u_t=u_t\circ j \text{ and } \eta_t= j^*(\eta_t).
    \end{eqnarray*}Consequently,
    \begin{eqnarray*}
        \a_t=e^{u_t}*_E(\b_0+\bar \p \eta_t)=e^{u_t\circ j}*_E (j^* (\b_0)+\bar\p j^*(\eta_t)) \quad \text{ and } \quad \xi_t=\xi_t\circ j.
    \end{eqnarray*} 
    
    Therefore, if along a sequence $t_k\to 0^+$, we have that $\xi_{t_k}$ blows-up at the point $q$, then by the above identity, we see that $\xi_{t_k}$ blows-up also at $j(q),$ and by the uniqueness of the blow-up point, we conclude that $j(q)=q,$ as claimed. 
    On the other hand, for genus $\gg=2$ the Kodaira map is two-to-one (see \cite{Tar_2}) in the sense that,  $\tau(q_1) = \tau(q_2)$
    iff $q_1 =q_2$ or $q_1 = j(q_2).$ Therefore $\tau^{-1}([\beta]_{\PP})=\{q \},$ hence the blow-up point is uniquely identified independently on the chosen sequence $t_k\to 0^+.$
   
With this information, part (i) of Theorem   \ref{thmg2}  follow immediately. While part (ii) is a direct consequence of Theorem \ref{thmB.1}.

\vskip0.5cm 
At this point we focus on the proof of Theorem \ref{thm3} and Theorem \ref{thm1.1}.

Thus, from now on we assume that $\xi_k$ blows up with corresponding blow-up divisor: 
\begin{equation}\label{divisor}
D = \sum_{j=1}^{m} n_j x_j, \mbox{ \ and } n_j \in \{1,2\},
\end{equation}
so that, 
$$
\mathcal{S} = \{x_1,\ldots,x_m\} \quad \mbox{and} \quad   0 < deg(D) \leq \gg -1.
$$
We wish to prove the orthogonality relation (\ref{ortog}) for a suitable effective divisor $0<\tilde{D}\leq D.$ 

To this purpose, we let $r>0$ sufficiently small, and for each point $x_l \in \mathcal{S}$
we let: 
$$
x_{k,l} \in \overline{B_{\delta}(x_l)} : \xi_k(x_{k,l}) = \text{max}_{\overline{B_{\delta}(x_l)}} \xi_k \to + \infty \ \mbox{and} \ x_{k,l} \to x_l \ \mbox{as $k \to + \infty$}.  $$

We introduce local holomorphic coordinates at $x_{l}$ centred at the origin, defined in $B(x_l; r)$, and denote by $z_{k,l} \to 0 $ as $k \to + \infty$, the corresponding local expression of $x_{k,l}.$ \\
As already observed in \eqref{betawedge}, for every $ \alpha \in C_2(X)$ we have:
$$
\int_X \beta \wedge \alpha = e^{\frac{-s_k}{2}}\left( \sum_{l=1}^{m} \int_{B(x_l, r)} e^{\xi_k}\<\alpha , \hat{\alpha_k}\> d A  \right) +o(1), $$
as $k \to + \infty.$

\vskip0.5cm
Under the assumption (\ref{divisor}), we can use Theorem \ref{mainthmasymp} and Remark \ref{molt1} around each blow-up point $x_l\ \in \mathcal{S}.$     
More precisely, with the notation of Theorem \ref{mainthmasymp}, we  define: 
$$ \mathcal{S}_1 :=  \{ x \in \mathcal{S} : \mbox {$D$ at $x$ has multiplicity $1$ or multiplicity $2$ and $\Lambda_1 \neq 0$} \}, $$ $\mathcal{S}_2 := \mathcal{S} \setminus \mathcal{S}_1.$ \\
Without loss of generality, we shall assume that both set $\mathcal{S}_1 \,\,\text{and} \,\, \mathcal{S}_2$ are \underline{not empty,} since for $\mathcal{S}_1 \,\,\text{or} \,\, \mathcal{S}_2$ empty, then the desired conclusion would follow by analogous yet simpler arguments.\\

Thus, for every $l\in \{1,...,m\}$ we let,
$$ \Gamma_l^{(1)} := \Lambda_1 \neq 0 \mbox{ \ if $x_l \in \mathcal{S}_1, \ \Gamma_l^{(2)} := \Lambda_2 \neq 0$ if $x_l \in \mathcal{S}_2,$   } $$
and 
$$
m^{(1)}_{j,k} := \frac{e^{\frac{-s_k}{2}}}{\theta_{1,k}} \ \mbox{if} \,\, x_l \in \mathcal{S}_1, 
\,\,  \,\,    
m^{(2)}_{j,k} := \frac{e^{\frac{-s_k}{2}} \theta_{2,k}}{\theta_{1,k}} \ \mbox{if} \, \,  x_l \in \mathcal{S}_2.
$$

We consider the divisor $\hat{D}_k =  \sum_{l=1}^{m} x_{k,l} \,\, \mbox { and } \,\, \hat{D} =  \sum_{l=1}^{m} x_{l}$ and in view of Lemma \ref{approssimazione}, for every $\alpha \in Q (\hat{D})$ we find a sequence $\alpha_k \in Q (\hat{D}_k)$ such that $\alpha_k \to \alpha$ as $k \to + \infty.$

Around  each blow-up point $x_l \in \mathcal{S} , \,\, l=1, ..., m,$ we consider the local expression (\ref{eta}) for $\alpha$ and $\alpha_k$ and set: \\
$\a= a_l(z)(dz)^2$ and $\a_k= a_{k,l}(z)(dz)^2$ with $a_l(z)$ and $a_{k,l}(z)$ holomorphic around the origin. 
Moreover, if $x_l \in \mathcal{S}_2 $ then we have: $a_{k,l} (z_{k,l} )=0$ and $a_l(0)=0.$ 
  
Thus, according to (\ref{betawedge}) and by using Theorem \ref{mainthmasymp},  Remark \ref{summarize}, Remark \ref{molt1} and once we let $r \to 0^+,$   we derive as $k \to + \infty$, 
\begin{eqnarray}\label{sum}
\int_X \beta \wedge \alpha = \int_X \beta \wedge \alpha_k  \,+  o(1) \, = \\
= \left( \sum_{x_l \in \mathcal{S}_1}  m^{(1)}_{l,k} \Gamma_l^{(1)}  \left( a_{l}(0) + o(1) \right) + \sum_{x_l \in \mathcal{S}_2}  m^{(2)}_{l,k} \Gamma_l^{(2)}\left( a'_{l}(0) + o(1) \right) \right) + o(1), 
\end{eqnarray}

We let $j_0 \in \{1,...,m\}$ such that (up to subsequences):
$|m^{(1)}_{j,k}| \leq |m^{(1)}_{j_0,k}|$ for all $ j\in \{1,...,m\} \text{ with } x_{j} \in \mathcal{S}_1 $
and $ k \in \mathbb{N}$ sufficiently large. 
Therefore, $x_{j_0} \in \mathcal{S}_1$ and if we let $\alpha_{j_0}^{(1)}$ the holomorphic quadratic differential given by Lemma \ref{zeros} (with $D$ in \eqref{divisor}) we find: 
\begin{equation}\label{bound1}
\int_X \beta \wedge \alpha_{j_0}^{(1)} =  m^{(1)}_{j_0,k}  \left(a_{j_0}(0) \Gamma_{j_0}^{(1)} + o(1) \right)  + o(1). \end{equation}
Since the coefficient $a_{j_0}(0) \Gamma_{j_0}^{(1)} \neq 0,$ then  from \eqref{bound1} we find that the sequence $m^{(1)}_{j_0,k}$ is uniformly bounded in $k$.  As a consequence, all sequences $m^{(1)}_{j,k}$, are uniformly bounded.

Similarly, we let $j_1 \in \{1,...,m\}$ such that (up to subsequences), we have: 
$|m^{(2)}_{j,k}| \leq |m^{(2)}_{j_1,k}|$ for all $ j\in \{1,...,m\}  : x_{j} \in \mathcal{S}_2 $
and $ k \in \mathbb{N}$ sufficiently large.

Then $x_{j_1} \in \mathcal{S}_2,$ and we can consider $\alpha_{j_1}^{(2)}$ the holomorphic quadratic differential as given in part (ii) of  Lemma \ref{zeros} and  obtain:
$$
\int_X \beta \wedge \alpha_{j_0}^{(2)} =  m^{(2)}_{j_0,k} \left(a_{j_0}'(0) \Gamma_{j_0}^{(2)} + o(1) \right)  + o(1). $$
and as above we conclude that also in this case  all sequences $m^{(2)}_{j,k}$, are uniformly bounded.

With those uniform bounds available, the orthogonality formula (\ref{ortog}) readily follows from \eqref{sum} with $\tilde{D} =( \sum_{x_j \in \mathcal{S}_1}  x_j + \sum_{x_j \in \mathcal{S}_2} 2x_j) \leq  D.$ 
\vskip0.5cm

Clearly, from the above results we derive Theorem  \ref{thm3} and Theorem \ref{thm1.1} easily. 
\vskip0.5cm

In concluding, we point out that the orthogonality condition enjoyed by the blow-up divisor D can be used to provide a simple expression for the linear form: $\alpha \rightarrow \int_X \beta \wedge \alpha.$

For this purpose, we let  $\alpha_{j}^{(1)}$
and $\alpha_j^{(2)}$ the differentials given by Lemma \ref{zeros} with respect to the blow-up divisor $D.$
Given $x_j \in \mathcal{S}_1$ we let: $c_j^{(1)} :=  \int_X \beta \wedge \alpha_{j}^{(1)}$
and $c_j^{(2)} := 0,$ while 
 for  $x_j \in \mathcal{S}_2$ we let $c_j^{(1)} := \int_X \beta \wedge \alpha_{j}^{(1)}$
and $c_j^{(2)} :=  \int_X \beta \wedge \alpha_{j}^{(2)}.$ 
The orthogonality relation (\ref{ortog}) (valid  for  $D$) implies that the complex numbers $c_j^{(1)} \, \text{and} \, c_j^{(2)}$ are well defined independently of the chosen differentials $\alpha_j^{(1)} \, \text{and} \,\alpha_j^{(2)}.$

Furthermore, for all $ \alpha \in C_2(X),$  we can use the expansion (\ref{decom}), so that  by virtue of the orthogonality relation (\ref{ortog}) we find: 
\begin{equation}\label{sommapesata}
\int_X \beta \wedge \alpha = \sum_{j=1}^{m}  \left(c_j^{(1)}\frac{ a_{j}(0)}{a_{j}^{(1)}(0)} +  c_j^{(2)}\frac{a_{j}'(0)}{(a^{(2)}_j)'(0)}  \right), 
\end{equation}
in terms of  the local expression specified in (\ref{local}) around each point $x_j.$ \\
Observe that, since $\beta \neq 0,$
the numbers $c_j^{(1)}\, \text{and} \,$ $c_j^{(2)}$ are not all zero.\\

We can also express formula (\ref{sommapesata}) in a coordinate free way as follows:
\begin{proposition}
Under the assumptions above, there exist unique $b_j^{(1)} \in (K_X^2)_{x_j},$ and $b_j^{(2)} \in (K_X^3)_{x_j}$ not all zero such that:
$$
\begin{array}{l}
\int_X \beta \wedge \alpha = \sum_{j=1}^{m} \left(<\alpha(x_j),b_j^{(1)}>_{x_j} + <d_{x_j} \alpha, b_j^{(2)}>_{x_j} \right) \\  
\mbox{\ where  $< , >_{x_j}$ is the hermitian product induced by the hyperbolic metric $g_X$} \\
\mbox{on $(K_X^2)_{x_j}$ and $(K_X^3)_{x_j}$ respectively}. \end{array}  
$$
\end{proposition}
\begin{proof}
It suffices to observe that, in view of formula \eqref{sommapesata}  where $\alpha \in C_2(X)$ admits the local expression specified above around $x_j$ for $1 \leq j \leq m,$ we have: 
$$ c_j^{(1)}\frac{ a_{j}(0)}{a_{j}^{(1)}(0)}  = <\alpha(x_j),b_j^{(1)}>_{x_j}\quad \text{with $b_j^{(1)} = \overline{ \left(\frac{c_j^{(1)}}{a_{j}^{(1)}(0)} \right)}dz^2_{0},$}
$$
 and 
$$ c_j^{(2)}\frac{a_{j}'(0)}{(a^{(2)}_j)'(0)}  = <d_{x_j} \alpha(x_j),b_j^{(2)}>_{x_j} \quad \text{with $b_j^{(2)} = \overline{\left(\frac{c_j^{(2)}}{(a^{(2)}_j)'(0)}\right )}dz^3_{0}.$}
$$
The uniqueness follows by using the formulae above for   $\alpha_{j}^{(1)}$ and $\alpha_j^{(2)}.$  
\end{proof}

\section{APPENDIX }\label{Append1}
In this section we prove some useful facts needed above and of independent interest.

\subsection{Appendix 1}\label{appendix 1}

Throughout this section we use complex coordinates.
\begin{lemma}
    Let $\psi:=\psi(z)$ satisfy:
    \begin{eqnarray}\label{A.1eq}
\left\{
\begin{array}{llll}
-\De \psi =|z^2-1|^2 e^\psi ~~~   \text{ in } \RR^2 \\
 \int_{\RR^2 } |z^2-1|^2 e^\psi \frac{i}{2}dz \wedge d \bar{z} =16\pi.   
\end{array} \right.
\end{eqnarray}then,
\begin{eqnarray}\label{A.2}
    \psi(z)=\ln \frac{8|A|^2}{(|z|^2+|Az^2+ Bz+A|^2)^2},
\end{eqnarray}for any $A,B\in \CC$ with $A\neq 0$.  In particular,  \begin{eqnarray}
    \label{A.3}\psi(\frac 1 z) = \psi(z)+ 8\ln |z|.
\end{eqnarray}
\end{lemma}

\begin{proof}
    We notice that the function  $\v(z):=\psi(z)+2\ln|z^2-1|$ satisfies:
     $$
\left\{
\begin{array}{llll}
-\De \v =  e^\v-4\pi \d_{p=1}-4\pi\d_{p=-1}  ~~~   \text{ in } \RR^2 \\
 \int_{\RR^2 }   e^\v \frac{i}{2}dz \wedge d \bar{z}=16\pi.   
\end{array} \right.
$$ Hence, by Liouville formula \cite{Liouville}, we know that
\begin{eqnarray}\label{A.4}
    \v(z)=\ln \frac{8|f'(z)|^2}{(1+|f(z)|^2)^2}
\end{eqnarray}with $f$ a rational function on $\CC$ with critical points exactly $\pm 1$ of multiplicity $1$. Moreover, the integral condition in \eqref{A.1eq} implies that  $f$ must have degree 2. Moreover, recall that \eqref{A.4} is invariant if we replace  $f$ with the function: $ \frac {\gamma f + e}{ -\overline {e} f + \overline {\gamma}} \,\,$ for given $\gamma, e \in \mathbb{C}$ satisfying: $|\gamma|^{2} + |e|^{2} = 1.$ Therefore, without loss of generality we can take: $f(z)=\frac{az^2+ bz+c}{\a z+\beta}, $ with $a\neq 0$ and $\a \neq 0$.

By direct calculations, we see that $f'$ admits simple zeroes at $\pm 1$, if and only if $\beta=0$ and $a=c$. Namely, by setting $A:=\frac a {\alpha} \neq 0$, $B:=\frac b {\alpha} \in \CC$. We find:  $f(z)=A(z+\frac 1 z) +B$. Clearly, $f(z)=f(\frac 1 z)$, which should not be surprising, since all the critical points of $f$ lie on the unit circle, and so (up to equivalence) $f$ must be symmetric with respect to such circle.

Consequently, $f'(z)= A\left(1-\frac 1 {z^2}\right)$ and therefore:
\begin{eqnarray*}
   \psi(z)+ 2\ln|1-z^2|=\v(z) &=&\ln \left( \frac{ 8|A|^2 |z^2-1|^2}{|z|^4 (1+ |A(z+\frac 1 z)+B|^2)^2 } \right) \\& =& \ln \left( \frac{8|A|^2 |z^2-1|^2}{|z|^4 (1+ |A(z+\frac 1 z)+B|^2)^2 }\right) ,
\end{eqnarray*}and we deduce that $\psi$ is given by \eqref{A.2}, from which \eqref{A.3} can be easily checked.

\end{proof}

\begin{lemma}\label{lemmaA.2}
    Let $\psi$ satisfy \eqref{A.1eq}, then
    $$
    \int_{\RR^2} e^{\psi(z)} (z^2-1) \frac{i}{2}dz \wedge d \bar{z}=-8\pi.
            $$
\end{lemma}
\begin{proof}
    Let us prove first that the given integral is real, in the sense that it admits a vanishing imaginary part.

    By virtue of \eqref{A.3}, we use the change of variable $z=\frac 1 w$ to find:
     \begin{eqnarray*}
       \int_{\R^2} e^{\psi(z)} z^2 \frac{i}{2}dz \wedge d \bar{z}&=&\int_{\R^2} e^{\psi(\frac 1 w)} \frac 1 {w^2} \frac{1}{|w|^4} \frac{i}{2}dw \wedge d \bar{w}=  \int_{\R^2} e^{\psi(w)} \frac 1 {w^2} |w|^4 \frac{i}{2}dw \wedge d \bar{w}=\\
       &=&\int_{\R^2} e^{\psi(w)} \bar w^2 \frac{i}{2}dw \wedge d \bar{w}.
    \end{eqnarray*}  That is, $\int_{\R^2} e^{\psi(z)} z^2 \frac{i}{2}dz \wedge d \bar{z}$ coincides with its complex conjugate, and so, 
    \begin{eqnarray*}
        \text{Im}\left(\int_{\R^2} e^{\psi(z)} z^2 \frac{i}{2}dz \wedge d \bar{z}\right) =0.
    \end{eqnarray*} As a consequence we also deduce that: 
     $\text{Im}\left(\int_{\R^2} e^{\psi(z)} (z^2-1) \frac{i}{2}dz \wedge d \bar{z}\right) =0.$ 
    
    Next, by setting $z:=x+i y$, in cartesian coordinate we find:
    \begin{eqnarray*}
        \text{Re} \left(\int_{\R^2} e^{\psi(z)} (z^2-1) \frac{i}{2}dz \wedge d \bar{z} \right) = \int_{\R^2} e^{\psi}(x^2-1-y^2) dxdy.
    \end{eqnarray*}
So by setting:  $R(x,y)=|(x+iy)^2-1|^2 =(x^2-y^2 -1)^2 +(2xy)^2$ and $\beta:=\frac{1}{2\pi} \int_{\R^2} R(x,y) e^\psi dxdy$, we obtain that :
\begin{eqnarray*}
    -\De \psi =R(x,y) e^\psi\text{ in } \R^2 \text{ and } \beta= \frac 1 {2\pi} \int_{\R^2} R(x,y) e^\psi dxdy =8.
\end{eqnarray*}
Therefore, we can apply Theorem 2 of \cite{Chen_Li_1} and obtain the following useful identity:

\begin{eqnarray*}  \int_{\RR^2} \left( x\frac{\p R}{\p x}+y\frac{\p R}{\p y}\right) e^\psi dxdy =\pi \beta (\beta-4). \end{eqnarray*}

By straightforward calculations we find:  \begin{eqnarray*}  x\frac{\p R}{\p x}+y\frac{\p R}{\p y}=4(R(x,y) +x^2-y^2 -1)=4(R(x,y)+\text{ Re}(z^2-1)), \end{eqnarray*}and consequently, \begin{eqnarray*} \pi\b(\b-4)&=& \int_{\RR^2} \left( x\frac{\p R}{\p x}+y\frac{\p R}{\p y}\right) e^\psi  =4\int_{\RR^2} R(x,y)e^\psi +4\int_{\RR^2} e^\psi \text{Re}(z^2-1)\frac{i}{2}dz \wedge d \bar{z} \\&=&8\pi\b+4\int_{\RR^2} e^\psi \text{Re}(z^2-1) \frac{i}{2}dz \wedge d \bar{z}. \end{eqnarray*}

Thus, we obtain: 
$$
4\int_{\RR^2} e^\psi \text{Re} (z^2-1)\frac{i}{2}dz \wedge d \bar{z}=\pi\b(\b-4)-8\pi\b=\pi \b(\b-12), $$  with $\b=8$. In other words, $$     \int_{\RR^2} e^{\psi(z)}  \text{Re}(z^2-1)dzd\bar z=-8\pi $$ as claimed.

\end{proof}

\

\
\subsection{Appendix 2: Invariance of Donaldson functional}\label{append2}
Let $X$ and $Y$ be two Riemann surfaces with genus $\gg\geq 2$, and suppose there exists:
$$
    f:X\to Y \text{ a biholomorphic map.}
$$
We denote by $g_X$ and $g_Y$ the unique hyperbolic metric (compatible with the complex structure) in $X$ and $Y$ respectively.
Hence

\begin{eqnarray*}
    f:(X,g_X)\to (Y,g_Y)
\end{eqnarray*}defines an orientation-preserving isometry.

In holomorphic $z$-coordinates on $X$ and $w$-coordinates on $Y$ as in (\ref{hypcoord}), we have:
\begin{eqnarray*}
    g_X= e^{2u_X} |dz|^2 \text{ and } g_Y=e^{2u_Y} |dw|^2,
\end{eqnarray*}with $u_X$ and $u_Y$ smooth functions on $X$ and $Y$ respectively. So that the isometry condition states that,
\begin{eqnarray}\label{(2)}
    e^{2u_X(z)}=e^{2u_Y(f(z))} |f'(z)|^2,
\end{eqnarray}and the following holds for the area elements:
\begin{eqnarray*}
&&    dA_X(z)= \frac i 2 e^{2u_X(z)} dz\wedge d\bar z =\frac i 2 e^{2u_Y(f(z))} |f'(z)|^2 dz \wedge d\bar z
\\&& dA_Y(w)= \frac i 2 e^{2u_Y(w)} dw\wedge d\bar w.
\end{eqnarray*} Also, we denoted by  $E_X=T_X^{1,0}$ and $E_Y= T_Y^{1,0}$ the 
holomorphic tangent bundles relative to $X$ and $Y$ respectively,  
with hermitian products   
$\<\cdot, \cdot\>_{E_X},$ 
$\<\cdot, \cdot\>_{E_Y}$ and norms $\|\cdot 
\|_{E_X}$  and $\|\cdot \|_{E_X}$ induced by $g_X,$ and $g_Y$ respectively.

Next, we consider the pullback $f^*$ of the map $f$ on the space of Beltrami differentials:
\begin{eqnarray*}
    f^*: A^{0,1}(Y,E_Y)\to A^{0,1}(X,E_X)
\end{eqnarray*}defined as follows: for $\b\in A^{0,1}(Y,E_Y)$ and for $x\in X$, we take $v\in (T^{0,1}_X)_x$, then
\begin{eqnarray*}
    (f^*(\b))_x [v]:= (df^{-1})_y ( \b_y(\overline{(df)_x[\bar v]})), \text{ with } y=f(x).
\end{eqnarray*}
\underline{\textbf{Claim 1:}}\begin{eqnarray}\label{(3)}
    \|f^*(\b)\|_{E_X}= \|\b\|_{E_Y}, ~~ \forall \b \in A^{0,1}(Y,E_Y).
\end{eqnarray}
To establish \eqref{(3)}, for any $x_0\in X$ and $y_0=f(x_0)\in Y$, we use holomorphic $z$-coordinates around $x_0$ and holomorphic $w$-coordinates around $y_0,$ as in (\ref{hypcoord}). 

Therefore we can write: $\b=\beta(w)d\bar w \frac \p {\p w}$ in a small ball around the origin. Hence, in the $z$-coordinates around the origin, we have:
\begin{eqnarray}\label{(4)}
    (f^*(\b))_z= (\beta(f(z)) \bar{f'}(z) (f^{-1})' (f(z))d\bar z \frac \p {\p z} =\beta(f(z)) \frac{ \bar{f'}(z)}{f'(z)} d\bar z \frac{\p }{\p z}.
\end{eqnarray} Consequently, by \eqref{norme} we have:
\begin{eqnarray*}
    \| \b_{y_0}\|_{E_Y}=|\beta(0)|\text{ and } \|f^*(\b)_{x_0}\|_{E_X}= |\beta(f(0))|=|\beta(0)|=\|\b_{y_0}\|_{E_Y}
\end{eqnarray*}
and \eqref{(3)} is established.
\vspace{0.2cm}

\underline{\textbf{Claim 2:}} If $\b\in  A^{0,1}(Y,E_Y)$ is harmonic with respect to $g_Y$ then $f^*(\b)\in A^{0,1}(X,E_X)$ is harmonic with respect to $g_X$. 

To establish Claim 2, we denote still by $f^*$ the following pullback by the map f :
\begin{eqnarray*}
    f^*: A^0(E_{Y}^*) \to A^0(E_{X}^*)
\end{eqnarray*}
defined locally by setting: $f^*(b(w)dw)=b(f(z))f'(z)dz.$

When $f$ is biholomorphic we see that the tensor product of the above map $f^*$, namely:
\begin{eqnarray*}
    f^*\otimes f^*: A^0(E_{Y}^*\otimes E_{Y}^*) \to A^0(E_{X}^*\otimes E_{X}^*),
\end{eqnarray*} maps the space $C_2(Y)$ into $C_2(X). $ At this point Claim 2 follows esaily by the commutation relation :
\begin{eqnarray}\label{(6)}
    (f^*\otimes f^*)(*_y)=*_x\circ f^*.
\end{eqnarray}
To check \eqref{(6)}, for a Beltrami differential $\beta$ on $Y$  we use the local holomorphic  coordinates above and set:  $\b=\beta(w)d \bar w \frac{\p}{\p w},$  then by formula \eqref{**bis} we have:

 \begin{eqnarray*}
    *_y \left(\beta(w) d\bar w \frac \p {\p w} \right) = - \frac i 2\overline{ \beta}(w) e^{2u_Y(w)}(dw)^2,
\end{eqnarray*}
clearly a similar local formula holds for the local expression $*_x$ applied to a Beltrami differential on $X.$
Therefore, by using \eqref{(2)} and \eqref{(4)}, locally we derive:
\begin{eqnarray*}
  &&  (f^*\otimes f^*)(*_y\b)=-\frac i 2 \bar{\beta}(f(z)) e^{2u_Y(f(z))} (f'(z))^2 (dz)^2
    \\&=&  -\frac i 2 \bar{\beta}(f(z)) \frac{(f'(z))^2}{|f'(z)|^2} e^{2u_Y(f(z))} |f'(z)|^2 (dz)^2
    \\&=& -\frac i 2\overline{\left( \beta (f(z)) \frac{\bar {f'}(z)} { f'(z)}  \right) } e^{2u_X(z) } (dz)^2 =*_x(f^* (\beta)).
\end{eqnarray*}
Next, still by $f^*$ we denote the pullback by the map f acting on sections, namely:
\begin{eqnarray*}
    f^*: A^0(E_{Y}) \to A^0(E_{X})
\end{eqnarray*}
defined as follows:
$$
\text{if } \eta \in A^0(E_Y)  \text{ then } (f^*(\eta))_x= (df^{-1})_y(\eta_y)\text{ for } x\in X \text{ and } y=f(x).
$$

\underline{\textbf{Claim 3:}}  $\forall \eta\in  A^0(E_Y)$, we have: \begin{eqnarray}\label{(7)}
\bar\p (f^*(\eta)) =f^*(\bar\p \eta)\end{eqnarray}

To establish \eqref{(7)} we use again the local expression: $\eta=a(w)\frac \p {\p w}$ and compute:
\begin{eqnarray*}
    (f^*(\eta))_z= (df^{-1}) (f(z)) a(f(z)) \frac \p {\p z} =\frac 1 {f'(z)} a(f(z)) \frac \p {\p z}.
\end{eqnarray*}Since $f$ is biholomorphic and in view of \eqref{(4)}, we find:
\begin{eqnarray*}
    \bar\p (f^*(\eta))_z&=& \left( \frac 1 {f'(z)} \frac{\p a}{\p \bar w} (f(z)) \frac{\p \bar f}{\p \bar z} \right) d\bar z \frac \p {\p z }\\&=&  \frac{\p a}{ \p \bar w} (f(z)) \frac{\bar{ f'}(z)}{ f'(z)} d\bar z \frac \p {\p z} =f^* (\bar \p \eta).
\end{eqnarray*}
\underline{\textbf{Claim 4:}}  $\forall u \in C^{\infty}(Y)$ and $\forall x \in X$, we have: 
\begin{equation}\label{grad} |\nabla(u \circ f) |_{g_X}(x)  = |\nabla(u)|_{g_Y}(f(x)).\end{equation}
To establish \eqref{grad} we use the definition of  gradient together with the fact that $f$ is an isometry with respect to the hyperbolic metrics, $g_X$ and $g_Y$ respectively. Therefore, $\forall x \in X $ and $ \forall v \in E_x,$  we have:
$$ \begin{array}{l} <df_x(\nabla(u \circ f) (x)),df_x(v)>_{g_Y} =  <\nabla(u \circ f) (x)),v>_{g_X} = \\ d(u \circ f)_x(v) = du_{f(x)}(df_x(v))= <\nabla(u)(f(x)),df_x(v)>_{g_Y}. \end{array} $$
Hence $df_x(\nabla(u \circ f) (x)) = \nabla(u)(f(x))$ and consequently: $ |\nabla(u)|_{g_Y}(f(x))= |df_x(\nabla(u \circ f)(x))|_{g_Y}=|\nabla(u \circ f) |_{g_X}(x) $
and the claim is established.\\ \\

At this point, we are ready to establish the invariance of the Donaldson functional defined in \eqref{F_t}.

For harmonic $\b_0\in A^{0,1}(Y,E_Y)\setminus \{0\}$, consider the Donaldson functional relative to the pair $(Y,[\b_0])$:
\begin{eqnarray*}
    F_t^{\b_0}[Y](u,\eta)=\int_Y \left( \frac{|\n u|^2_{g_Y}}{4}-u+te^u+4e^u\|\b_0+\bar{\partial} \eta \|^2\right) dA_Y,
\end{eqnarray*}with $(u,\eta)\in H^1(Y) \times W^{1,2}(Y,E_Y)$.

Thus, by Claim 2, it is well defined the Donaldson functional relative to the pair $(X,f^*(\b_0))$ as follows:
\begin{eqnarray*}
    F_t^{f^*(\b_0)} [X](v,\xi)=\int_X \left( \frac{|\n v|^2_{g_X}}{4}-v+te^v +4e^v \|f^*(\b_0)+\bar\p \xi\|^2_{E_X} \right) dA_X,
\end{eqnarray*}with $(v,\xi)\in H^1(X)\times W^{1,2}(X,E_X)$. The following holds:
\begin{proposition}\label{propA.1}
    Let $f:X\to Y$ biholomorphic then
    \begin{eqnarray*}
        F_t^{\b_0}[Y](u,\eta)=F_t^{f^*(\b_0)}[X](u\circ f,f^*(\eta)),
    \end{eqnarray*}$\forall (u,\eta) \in H^1(Y)\times W^{1,2}(Y,E_Y).$
\end{proposition}
\begin{proof}
    In view of \eqref{(2)}, \eqref{(3)}, Claim 2, claim 4  and \eqref{(7)}, we derive:
    \begin{eqnarray*}
      &&  \int_X \left( \frac{|\n (u\circ f)|^2_{g_X}} {4} -(u\circ f)+te^{u\circ f} +4 e^{u\circ f}\|f^*(\b_0)+\bar\p f^*(\eta)\|^2 \right) dA_X
        \\&=& \int_X \left( \frac{|\n u|^2_{g_Y}(f(x))} {4} -u(f(x)+te^{u(f(x))} +4 e^{u(f(x))}\|f^*(\b_0)+\bar\p f^*(\eta)\|^2_{E_X} \right) dA_X
        \\&=&\int_Y \left( \frac{|\n u(y)|^2_{g_Y}} {4} -u(y)+te^{u(y)} +4 e^{u(y)}\|\b_0+\bar \p\eta\|^2_{E_Y}\right) dA_Y=F_t^{\b_0} [Y](u,\eta),
    \end{eqnarray*}$\forall (u,\eta)\in H^1(X)\times W^{1,2}(Y,E_Y)$.
\end{proof}

The invariance property of Proposition \ref{propA.1} is crucial to establish Theorem \ref{thmg2}. 
It will be applied to the Riemann surface $X$ of genus $2,$ where the biholomorphism is given by the unique involution:
\begin{eqnarray*}
    j: X\to X
\end{eqnarray*} (cf. \cite{Miranda}, \cite{Griffiths_Harris}). Moreover in the proof  of Theorem \ref{thmg2} also we need the following: 
\begin{proposition}\label{propA.2}
    If $\b_0\in A^{0,1}(X,E)\setminus\{0\}$ is harmonic with respect to $g_X$ then
    \begin{eqnarray*}
        j^*(\b_0)=\b_0.
    \end{eqnarray*}
\end{proposition}
\begin{proof}
    Recall that $j$ admits exactly six fixed points, say: $\{q_1,\ldots, q_6\}$.  Furthermore, since $dj_q$ has order $2$ then $j'(q)=\pm 1$. 

    Therefore, by means of \eqref{(4)}, we see that 
    \begin{eqnarray*}
        (j^*(\b_0))_q= (\b_0)_q,\forall q\in\{q_1,\ldots, q_6\}.
    \end{eqnarray*}
    As a consequence, the harmonic Beltrami differential $j^*(\b_0)-\b_0$ vanishes at those  six points.

    Thus also the holomorphic quadratic differential:
    \begin{eqnarray*}
        \a^0=*_E(j^*(\b_0)-\b_0)\in H^0(X,E^*\otimes E^*)=C_2(X)
    \end{eqnarray*} vanishes at those six points as well. But we have shown in \eqref{2.01} that
    $E^*\otimes E^*=K_X\otimes K_X$ admits degree $4$ (for $\gg=2$) and so $\a^0$ must be identically zero. That implies that, $j^*(\b_0)=\b_0$ as claimed.
\end{proof}

\subsection{Appendix 3: Secant varieties.  }\label{append3}
Our goal here is to establish Proposition \ref{***}.\\
Since the statement of Proposition \ref{***} is well  known for $\gg=2$ (i.e. $\nu=1$), from now on we suppose that $X$ admits genus $\gg \geq 3.$ As a consequence, the corresponding Kodaira map $\tau$ is one to one,  moreover $\tau(X)$ is a non degenerate complex curve.\\
 \
For this purpose let, 
$$
N=3(\gg-1) \text{ and } V:=C_2(X).
$$

For fixed $\nu\in\{1,\ldots, 2\gg-3\},$ let $G_{N-\nu}(V)$ be the Grassmannian of $N-\nu$ dimensional complex subspaces of  $V$ and define the map: 
\begin{equation}\label{psi}
 \Psi : X^{(\nu)} \to G_{N-\nu}(V) \,\,\text{ with } \Psi(D) = Q(D).
 \end{equation}

It is well known that $\Psi$ is holomorphic, but for the reader convenience we will give a direct proof of this fact. 
 Recall that, to give a structure of complex manifold on $G_{N-\nu}(V)$, we choose a basis $\a_1,\ldots, \a_N$ of $V$ and identify $G_{N-\nu}(V)$ with $G_{N-\nu}(\CC^N)$. The structure on $G_{N-\nu}(\CC^N)$ is the quotient of $\O$ modulo the group $GL(\nu)$ acting by left multiplication, where $\O\subset \CC^{\nu\times N}$ is the open set of  $\nu\times N$ complex matrices of maximal rank. The holomorphic quotient map $q:\O\to G_{N-\nu}(\CC^N)$ is defined by setting: $A\to \text{Ker}(A)$. Given $A\in \CC^{\nu\times N}$ a  $\nu\times N$ complex matrix, we denote by $A^j\in \CC^\nu$ the $j$-th column of $A$, for $j=1,\ldots,N$. Fix a set $I$ of $\nu$ positive integers, namely: $I=\{1\leq i_1<i_2\cdots <i_\nu \leq N\},$ let $A_I$ the $\nu\times \nu$ matrix formed with the columns $A^{i_1},A^{i_2},\ldots,A^{i_\nu}$ and define:  $\O_I:=\{A\in\O: \det(A_I)\neq 0\}$. Then $q(A)=q((A_I)^{-1} A)$ for every $A\in \O_I$. Consider the subset $\Lambda_I\subset \O_I$ given as follows: $\Lambda_I=\{A\in \O_I : A_I= Id_\nu\}$ with $Id_{\nu}$ the identity metrix,  then the restriction of $q$ on $\Lambda_I$ gives a local holomorphic chart of $G_{N-\nu}(\CC^N)$.

Now let $p_0=(x_1^0, \ldots, x_\nu^0)\in X^{(\nu)},$ by \eqref{2.07tris} we find $\a^0\in H^0(X,\otimes^2(K_X))$ which does not vanish at $x_j^0$ for all $1\leq j\leq \nu.$

Hence, we can choose a small open neighborhood $\textbf{W}=W_1\times W_2\cdots \times W_\nu $ of $p_0,$ such that  $\a^0$ is nowhere zero on each $W_j$, for $1\leq j\leq \nu.$

We consider the map $\tilde\Psi: \textbf{W}\to \O$ so that, for given indices $1\leq j\leq \nu$ and $1\leq h\leq N$ and for $x=(x_1,\ldots, x_\nu)$, we set:
\begin{eqnarray*}
[\tilde\Psi(x)]_{j,h}= \left( \frac{\a_h}{\a^0}\right)^0 (x_j), \ldots, \left(\frac{\a_h}{\a^0}\right)^{\nu_j-1} (x_j),\\ \text{  \ if the point \ } x_j \text{ appears } \nu_j \text{ times in } x;
\end{eqnarray*}
where the upper index denotes the holomorphic derivative and $\{\a_1,\ldots, \a_N\}$ is the fixed basis in $V=C_2(X)$ . Observe that with the given local holomorphic charts above, we have $\Psi|_\textbf{W}=q\circ \tilde\Psi.$ In order to show that the map $\Psi$ is holomorphic, we need the following:
\begin{lemma}
    Let $U: D(0,\varepsilon)\to \CC^N$ be a holomorphic map where $D(0,\varepsilon)\subset \CC$ is the open disk with center $0$ and radius $\varepsilon$. Let $z_1,z_2,\ldots,z_k$ be distinct points in $D(0,\varepsilon)$, and define inductively the following vectors:
    $$ \begin{array}{l}
        F_1(z_1):=U(z_1), \\  F_{j+1}(z_1,z_2,\ldots,z_{j-1},z_j,z_{j+1}):= \frac{F_j(z_1,\ldots, z_{j-1},z_{j+1})- F_j(z_1,\ldots, z_{j-1},z_j)}{z_{j+1}-z_j}.
    \end{array} $$ 
    We have:
    \begin{enumerate}
        \item[1)] $\text{span}\{F_1(z_1),F_2(z_1,z_2),\ldots, F_k(z_1,z_2,\ldots, z_k)\} =\text{span}\{U(z_1),U(z_2),\ldots, U(z_k)\}$,

        \item[2)] \begin{eqnarray*}
            \lim_{(z_1,\ldots, z_j)\to (0,\ldots,0)} F_j(z_1,\ldots,z_j)= \frac{U^{j-1}(0)}{(j-1)!},
        \end{eqnarray*}for all $1\leq j\leq k$.

    \end{enumerate}
\end{lemma}
\begin{proof}
    First, we show the statement 1). For each $j\geq 1$, we claim that we can find  suitable functions $c_h, 1 \leq h \leq j$ such that 
     \begin{eqnarray*}
F_j(z_1,\ldots,z_j)= c_1(z_1,\ldots,z_j)U(z_1)+\;\ldots \; +c_j(z_1,\ldots,z_j)U(z_j),
    \end{eqnarray*}  with $c_j(z_1,\ldots,z_j) \neq 0$. Indeed, while this is clear for $j=1$, if by induction we assume it is true for $j$, then by using the definition, for $j+1$ we find that, 
    \begin{eqnarray*}
        c_{j+1}(z_1,\ldots, z_{j+1})=\frac{c_j(z_1,\ldots, z_{j-1},z_{j+1})}{ z_{j+1}-z_j} \neq 0.
    \end{eqnarray*}
    Hence,  each $F_j(z_1,\ldots, z_j)$ is in the $
\text{span}(U(z_1), U(z_2), \ldots, U(z_j)\}$. To show the vice-versa, again by induction, it suffices to show that, for $1\leq j\leq k$ then $U(z_{j+1})$ is in the $\text{span}\{F_1(z_1), F_2(z_1,z_2), \ldots, F_{j+1}(z_1,z_2,\ldots, z_{j+1})\}.$

But, using again the definition, we have:
\begin{eqnarray*}
    && F_{j+1}(z_1,\ldots, z_{j-1},z_j, z_{j+1}):= \frac{F_j(z_1,\ldots, z_{j-1},z_{j+1})- F_j(z_1,\ldots, z_{j-1},z_j)}{z_{j+1}-z_j}
    \\&=& \frac 1 {z_{j+1}-z_j} \big(c_1(z_1,\ldots, z_{j-1},z_{j+1})U(z_1)+c_2(z_1,\ldots, z_{j-1},z_{j+1}) U(z_2)  
    \\&& +\cdots c_{j-1}(z_1,\ldots, z_{j-1},z_{j+1} )U(z_{j-1}) \big) + \left( \frac{c_j(z_1,\ldots, z_{j-1},z_{j+1})}{z_{j+1}-z_j}\right) U(z_{j+1}) \end{eqnarray*}  
    $$ - \frac{F_j(z_1,\ldots, z_{j-1},z_j)}{z_{j+1}-z_j}, $$
and by the previous argument, the term:
\begin{eqnarray*}
    \frac{F_j(z_1,\ldots, z_{j-1},z_j)}{z_{j+1}-z_j} \in \text{span}\{U(z_1), U(z_2),\ldots, U(z_j)\},
\end{eqnarray*} with 
$$ \left( \frac{c_j(z_1,\ldots, z_{j-1},z_{j+1})}{z_{j+1}-z_j}\right)\neq 0.$$ Since,
by the induction assumption we know that,\\
$\text{span}\{U(z_1), U(z_2),\ldots, U(z_j)\}\subset \text{span}\{F_1(z_1), F_2(z_1,z_2),\ldots, F_{j+1}(z_1,\ldots,z_{j+1})\}$, by means of the identity above, we conclude that,\\
$U(z_{j+1}) \in \text{span}\{F_1(z_1), F_2(z_1,z_2),\ldots, F_{j+1}(z_1,\ldots,z_{j+1})\}$ as claimed, and the proof of statement 1) is completed.

To prove 2), we will show that, for given $2\leq j\leq k$ and $\t:=(t_1,t_2,\ldots, t_{j-1})$ there exist polynomials $\lambda_s^j(\t)$ with $1\leq s\leq k$ and a polynomial $Q_j(\t)$ such that  $\l_s^j(\t)$
 and $Q(\t)$ are strictly positive on the open cube $(0,1)^{j-1}$, they do not dependent on $U$ and:
\begin{eqnarray}\label{2.*}
    F_j(z_1,\ldots,z_j)=\int_{[0,1]^{j-1}} U^{j-1}\left( \sum_{s=1}^j \l_s^j(\t) z_s\right) Q_j(\t) d\t,
\end{eqnarray}with $d\t:=dt_1\cdots dt_{j-1}$. Since each derivative of $U$ is bounded on compact sets of $D(0,\varepsilon)$ then statement 2) follows from the above formula and from the dominated convergence theorem, as we shall see also that: $\frac 1 {(j-1)!}=: \int_{[0,1]^{j-1}} Q(\t) d\t $. 

To show \eqref{2.*}, we proceed by induction as we have: $$F_2(z_1,z_2)=\frac {U(z_2)-U(z_1)}{ z_2-z_1}=\frac{\int_0^1 U(\g(t_1))\g'(t_1)dt_1}{z_2-z_1}$$ with $\g(t)= tz_2+(1-t)z_1$. However $\g'(t)=z_2-z_1$, and therefore, if we let $\l_1^2(t_1)= 1-t_1$, $\l_2^2(t_1)= t_1$, and $Q_2(t_1)\equiv 1,$ we get the formula above for $j=2$. Next, assume that the formula is valid for $j$. We let $Z'=(z_1,\ldots,z_{j-1})$, $\l^\j(\t):=(\l_1^j(\t),\ldots, \l_{j-1}^j (\t))$, and $\<\l^j,Z'\>= \sum_{s=1}^{j-1} \l_s^j(\t) z_s$, with $\t=(t_1,\ldots, t_{j-1})$. We find: 
{\footnotesize\begin{eqnarray*}
 &&   F_{j+1}(Z', z_j,z_{j+1})=
    \\&=& \frac{\left( \int_{[0,1]^{j-1}} \left( U^{j-1}\left (\<\l^j,Z'\>+\l_j^j(\t) z_{j+1}\right)- U^{j-1} \left( \<\l^j, Z'\>+\l^j_j (\t) z_j\right) \right)  Q_j(\t) d\t\right)}{z_{j+1}-z_j}  
    \\&=&  \int_{[0,1]^{j-1}} \frac{U^{j-1}\left (\<\l^j,Z'\>+\l_j^j(\t) z_{j+1}\right)- U^{j-1} \left( \<\l^j, Z'\>+\l^j_j (\t) z_j\right) }{\l_j^j(\t)(z_{j+1}-z_j) } (\l_j^j Q_j)(\t) d\t
    \\&=& \int_{[0,1]^{j-1}}\int_{[0,1]} U^j(\g(\tau)) (\l_j^j Q_j) (\t) d\t d\tau,
\end{eqnarray*}}with 
\begin{eqnarray*}
    \g(\tau)&=& \tau(\<\l^j, Z'\> +\l_j^j(\t) z_{j+1})+(1-\tau) (\<\l^j, Z'\>+\l_j^j(\t) z_j) 
    \\&=& \<\l^j,Z'\>+ \tau \l_j^j(\t) z_{j+1}+(1-\tau) \l_j^j(\t) z_j.
\end{eqnarray*}
Consequently, we obtain the formula above also for $2\leq j\leq k$,  by setting:

\begin{eqnarray*} 
\left\{
\begin{array}{llll}
 \l_s^{j+1}(t_1,\ldots, t_{j-1},t_j) =\l_s^j(t_1,\ldots, t_{j-1}),\text{for } 1\leq s\leq j-1,  \\ \l_j^{j+1}(t_1,\ldots, t_j)=(1-t_j) \l_j^j (t_1,\ldots, t_{j-1}),\\
\l_{j+1}^{j+1}(t_1,\ldots, t_{j-1},t_j) =t_j \l_j^j(t_1,\ldots, t_{j-1}).\\
Q_{j+1}(t_1,\ldots,t_j)= \l_j^j (t_1,\ldots,t_{j-1}) Q_j(t_1,\ldots, t_{j-1}),\\
Q_2(t_1)\equiv 1\\
\l_1^2(t_1)=1-t_1, \l_2^2(t_1)=t_1,
\end{array} \right.
\end{eqnarray*}
So by using induction, we find that,

\begin{eqnarray*}
    \l_j^j(t_1,\ldots, t_{j-1})=t_1t_2\cdots t_{j-1},
\end{eqnarray*}and
\begin{eqnarray*}
    Q_j(t_1,t_2,\ldots,t_{j-1})= t_1^{j-2}t_2^{j-3} t_3^{j-4} \cdots t_{j-2}.
\end{eqnarray*}
In particular, we obtain:  $\int_{[0,1]^{j-1}} Q_j(t_1,\ldots, t_{j-1}) dt_1\cdots dt_{j-1}= \frac  1 {(j-1)!}$.

\end{proof}
Now the map $\Psi$ is clearly holomorphic on the open subset of $X^{(\nu)}$ where all points are distinct. Moreover, the kernel of the matrix $A$ defined above is the complex conjugate of the orthogonal to the span of the rows of $A$ with respect to the standard hermitian product in $\CC^N$. The lemma above shows that the map $\Psi$ is continuous in $X^{(\nu)}$, therefore it is holomorphic by Riemann extension theorem. Hence by the proper mapping Theorem (see \cite{Demailly} Chapter II Section 8.8)  the image $S$ of $\Psi$ (i.e. $S=\Psi(X^{(\nu)})$ ) is a closed subvariety of $G_{N-\nu}(V)$ of dimension $s\leq \nu$. Since $X$ is connected then $S$ is irreducible.

Let $V^*$ be the dual space of $V=C_2(X)$ and actually, as observed above, $V^*$ is the space of harmonic Beltrami differentials. In other words, $V^*\simeq \mathcal{H}^{0,1}(X,E)$, and thus the projective space $P(V^*)$  (associated to $V^*$)  satisfies : $P(V^*) \simeq P(\mathcal{H}^{0,1}(X,E))$. We let $\S_\nu\subset P(V^*)\times S$ be given by:
\begin{eqnarray*}
    \Sigma_\nu :=\{(\pi(\beta), W) \in P(V^*)\times S: \beta|_W\equiv 0\},
\end{eqnarray*} where $\pi: V^*\setminus\{0\}\to P(V^*)$ is the natural projection. The set $\S_\nu$ is closed in $P(V^*)\times S$, we claim that $\S_\nu$ is a subvariaty of $P(V^*)\times S$ as well. 
Let us identify $V^*$ with $\CC^N$ by means of the dual basis of $\{\a_1,\cdots, \a_N\}$ fixed above. With this identification, if $W=\text{Ker}(A)$ and $A\in  \Omega,$ then $\beta_0 \in V^*$ is zero on the vector subspace $W$  if and only if $\beta_0$ is a linear combination of the rows of $A$. This is true if and only if the $(\nu+1)\times N$ matrix $A'$ with rows $A_1,\ldots, A_\nu,\b_0$ has rank $\nu$. In turn, this is true if and only if all $(\nu+1)\times (\nu+1)$ minors of $A'$ vanish. Those minors then give local holomorphic defining functions for the subvariety $\S_\nu$. We compute the dimension of $\S_\nu$. Let $p_1: P(V^*)\times S\to P(V^*)$ and $p_2: P(V^*) \times S \to S$ be the projection on the first and second factor respectively. 
The map $p_2$ restricted to $\S_\nu$ is onto, and the fiber over the space $W\in G_{N-\nu} (V)$ is the projectivisation of the vector space of linear functionals on $V$ vanishing on $W$. Such a vector space has dimension $N-(N-\nu)=\nu$. So the fiber of $W$ has dimension $\nu-1$ and the dimension of $\S_\nu$ is $s+\nu-1\leq 2\nu-1$. Since the fibers of $W$ are projective spaces and $P(V^*)$ is smooth and connected the sets $\S_\nu$ is irreducible, hence  $p_1(\S_\nu)$ is also closed analytic and irreducible. 

It follows that the dimension of $p_1(\S_\nu)$ is not greater than $2\nu-1$. Observe that given the divisors: $D_1\in X^{(\nu_1)}$ and $D_2\in X^{(\nu_2)}$ such that $D_1\leq D_2$, then $Q(D_2) \subset Q(D_1)$, so $p_1(\S_1)\subset p_1(\S_\nu) \subset p_1(\S_{\nu+1}) \subset p_1(\S_{\gg-1})$. 

Moreover, since there is a unique non-zero linear form (up to scalars) which vanishes on an hyperplane, then $p_1(\S_1)$ is the one-dimensional image of the Kodaira map in \eqref{kodaira map}.

\begin{remark}\label{dimensione} 
In case $1\leq \nu\leq (\gg-1)$, we have: $\text{dim}(p_1(\S_\nu)) \leq 2(\gg-1)-1,$ and since $\gg>1$, then $2(\gg-1)-1<3(\gg-1)-1=N-1=\text{dim}(P(V^*)),$ and so $p_1(\S_\nu)\neq P(V^*)$. 
Actually, in this case, it can be shown that the subvaraity $p_1(\S_\nu)$ has dimension exactly
$2\nu-1.$ \end{remark}

By means of the identification explicitly given in \eqref{dual}, we have the isomorphism:
\begin{eqnarray*}
    i: & V^*&\to \mathcal{H}^{0,1}(X,E)
    \\&\b_0& \to [\b_0]
\end{eqnarray*}
($\b_0$ harmonic Beltrami differential) 
which can be lifted into the isomorphism:
\begin{eqnarray*}
    \hat i: \PP(V^*)\to \PP(\H^{0,1}(X,E)), \,\,\,  V=C_2(X).
\end{eqnarray*} 
Thus, by setting:
\begin{eqnarray} \label{secante}
  \tilde{\Sigma}_{\nu}:= \hat i \circ p_1(\S_\nu), \quad \nu=1,\cdots, \gg-1.
\end{eqnarray} 
In account of the above properties, we know that $ \tilde{\Sigma}_{\nu}$  defines a closed sub-variety, and by construction, it satisfies the properties \eqref{1} and \eqref{2}. Therefore, we are left to prove that $ \tilde{\Sigma}_{\nu}$ coincides with the $\nu$-secant variety of $\tau(X)$, a property pointed out to us from a previous version of the paper.\\
 
 To this purpose, let  $Z = \tau(X) \subseteq \PP(V^*),$   and  consider the canonical projection:
 $$\pi: V^* \setminus \{0\} \to \PP(V^*)\,\,\mbox {so that} \,\, \pi ( \beta)=[\beta]_{\PP},\,\, \forall \beta \in V^* \setminus \{0\} .$$  
 
 For given $\nu=1,\cdots, \gg-1,$  and distinct points $p_j \in X, \,\, j=1,\cdots, \nu, $ we consider the divisor $D := \sum_j p_j \in X^{(\nu)}$ and let $\beta_j \in V{^*} \setminus \{0\},$   satisfying:  $\pi(\beta_j)=[\beta_{j}]_{\PP} := \tau(p_j )\in Z,$ for $1 \leq j \leq \nu.$ By means of  \eqref{2.07bis} it is easy to check that   $\{\beta_1, \ldots, \beta_{\nu}\} \in V^*$ are linearly independent. Hence  $W_{D} := \pi( span(\beta_1, \ldots \beta_\nu) \setminus \{0\})$ is a projective subspace of  $\PP(V^*)$  (depending only of $D$)  with dim($W_{D}$)=$\nu-1.$ \\
 \medskip 

By definition, the $\nu$-secant variety $Y_{\nu}(Z)$ of $Z$ is  the closure in the Zariski topology of the union of all such projective subspaces $W_D,$
for any divisor $D\in X^{(\nu)}$ as specified above (namely, where all points in its support have multiplicity one),  see \cite{ACGH} Chapter VI section 1 for details, and  
clearly: $Y_1(Z)=Z.$ 

\begin{lemma}\label{secant}

If $1 \leq \nu \leq g-1$ the sub-variety $\tilde{\Sigma}_\nu$ in \eqref{secante} coincides with the $\nu$-secant variety $Y_{\nu}(Z)$ of  $Z=\tau(X).$ 
\end{lemma}

\begin{proof}
Let $D\in X^({\nu})$ be an effective divisor in $X$ of degree $\nu$ with $1 \leq \nu \leq \gg-1.$ We set: $$Q(D)^{\perp} = \{ \beta \in V^* : \,\, \int_X \beta \wedge \alpha = 0 \,\, \ \forall \alpha \in Q(D) \},$$ so that  $\tilde{\Sigma}_{\nu}$ is a closed subvariety (in Zariski topology) formed by union of all the projective subspaces $\pi (Q(D)^{\perp}),$ for any $D\in X^{(\nu)},$  and it is \\ 
On the other hand, if $D = \sum_{j=1, ..., \nu} p_j \in X^({\nu}) $ then 
$$Q(D) \subseteq Q(p_j) \ \mbox{so} \   Q(p_j)^{\perp} \subseteq Q(D)^{\perp} \,\, \forall \, j \in \{1, \ldots, \nu\},$$

and therefore, $ \tau(p_j) \in \pi(Q(D)^{\perp}), \,\, \forall \,j \in \{1, \ldots, \nu\}.$ 
 
 As a consequence, $\pi(span(\beta_1,\ldots, \beta_\nu)) \subseteq \pi(Q(D)^{\perp}),$ and since  dim($ \pi(span(\beta_1,\ldots, \beta_\nu)$)= $\nu -1$ = dim($\pi(Q(D)^{\perp}$) (recall \eqref {2.07*}) we conclude that, $$\pi(span(\beta_1,\ldots, \beta_\nu)) = \pi(Q(D)^{\perp}).$$   Thus, we obtain that,  $Y_{\nu} (\tau(X)) \subseteq \tilde{\Sigma}_\nu.$ 

It is easy to show that, if $\gg=3$  then equality holds, since as well known, $dim Y_{\nu} (\tau(X))=3$  whereas $dim \tilde{\Sigma}_\nu \leq 3,$ and consequently, $Y_{\nu} (\tau(X)) = \tilde{\Sigma}_\nu,$ as claimed.\\ 
To attain the same conclusion for genus $\gg \geq 4,$ we recall that any  divisor $D \in X^{(\nu)}$  is the limit (in the topology of $X^{(\nu)}$) of a sequence of divisors  $D_k  = \sum_{j=1, ..., \nu} p_j \in X^({(\nu})$ (i.e. $p_{j,k}$ all distinct). Consequently, from \eqref{psi} we find that, $Q(D_{k}) \to Q(D), \,\, \mbox{as}\, k \to \infty$ (with respect to the topology of the Grassmannian $G_{N-\nu}(V)$). This implies that (dually) we have: $Q(D_{k})^{\perp} \to Q(D)^{\perp}, \,\, \mbox{as}\, k \to \infty$ (with respect to the topology of  $G_{\nu}(V^{*})$). At this point, by using a (finite) orthonormal basis for $Q(D_{k})^{\perp}$  and arguing exactly as in the approximation  Lemma \eqref {approssimazione}  we see that, for any  $\beta \in Q(D)^{\perp}$ there exists $\beta_k \in Q(D_k)^{\perp}$ such that $\beta_k \to \beta, \,\, \mbox{as}\, k \to \infty.$ Hence, $\pi(\beta)\in Y_\nu(\tau(X))$ that is, $ \pi(Q(D)^{\perp})\subset Y_{\nu}(\tau(X)),$ a therefore: $\tilde{\Sigma}_\nu = Y_{\nu}(\tau(X)),$ as claimed. \end{proof}

\vskip.1cm
{\bf{Acknowledgement:}} 
\vskip.1cm
The authors have been partially supported by MIUR Excellence Department Projects awarded to the Department of Mathematics, University of Rome Tor Vergata, 2018-2022 CUP E83C18000100006, and 2023-2027 CUP E83C23000330006; moreover G.T. is partially supported  by PRIN \emph{Variational and Analytical aspects of Geometric PDEs} n.2022AKNSE4; and S.T. is partially supported by PRIN \emph{Real and Complex Ma\-ni\-folds: Topology, Geometry and holomorphic dynamics} n.2017JZ2SW5. 
 \vskip.1cm
   \noindent{\sc Universit\`a di Roma TorVergata, \\
  Via della Ricerca Scientifica 1, 00133, Roma, Italy, \\

\tt{tarantel@mat.uniroma2.it}} 

\tt{trapani@mat.uniroma2.it}

\bibliography{bib.bib}

\end{document}